\documentclass[12pt]{amsart}
\usepackage[utf8]{inputenc}
\usepackage{kantlipsum}
\usepackage[pagebackref=true, colorlinks=true,linkcolor=blue, citecolor={red}]{hyperref}
\usepackage{lmodern}\usepackage[T1]{fontenc}
\usepackage{amsmath,amssymb,amsfonts,euscript,graphicx,hyperref,indentfirst}
\usepackage{enumerate}
\usepackage[all]{xy}
\usepackage{relsize,exscale}
\calclayout

\newtheorem{theorem}{{\bf Theorem}}[section]

\newtheorem{preexample}[theorem]{{\bf Example}}

\newenvironment{example}{\begin{preexample}\rm{\hspace{-0.5
               em}{\bf}}}{\end{preexample}}

\newtheorem{prelem}[theorem]{{\bf Lemma}}

\newenvironment{lemma}{\begin{prelem}{\hspace{-0.5
               em}{\bf}}}{\end{prelem}}

\newtheorem{preprop}[theorem]{{\bf Proposition}}

\newenvironment{proposition}{\begin{preprop}{\hspace{-0.5
               em}{\bf}}}{\end{preprop}}

\newtheorem{prequ}[theorem]{{\bf Question}}

\newtheorem{predef}[theorem]{{\bf Definition}}

\newenvironment{definition}{\begin{predef}\rm{\hspace{-0.5
               em}{\bf}}}{\end{predef}}

\newtheorem{precor}[theorem]{{\bf Corollary}}

\newenvironment{corollary}{\begin{precor}{\hspace{-0.5
               em}{\bf}}}{\end{precor}}

\newtheorem{preconj}[theorem]{{\bf Conjecture}}

\newtheorem{preex}[theorem]{{\bf Exercise}}

\newtheorem{prerem}[theorem]{{\bf Remark}}

\newenvironment{remark}{\begin{prerem}\rm{\hspace{-0.5
               em}{\bf}}}{\end{prerem}}

\makeatletter
\def\l@subsection{\@tocline{2}{0pt}{2.5pc}{5pc}{}} 
\makeatother

\numberwithin{equation}{section}

\title{Real Entropy Rigidity under Quasi-conformal Deformations}
\author{Khashayar Filom}
\address{Department of Mathematics,
  University of Michigan;
 530 Church St, Ann Arbor, MI 48109,
USA}
\email{filom@umich.edu}
\begin{document}

\begin{abstract}
We set up a real entropy function $h_\Bbb{R}$ on the space $\mathcal{M}'_d$ of Möbius conjugacy classes of  real rational maps of degree $d$
by assigning to each class  the real entropy of a representative $f\in\Bbb{R}(z)$; namely, the topological entropy of its restriction 
$f\restriction_{\hat{\Bbb{R}}}$
to the real circle. 
We prove a  rigidity result stating that $h_\Bbb{R}$ is locally constant on the subspace determined by real maps 
quasi-conformally conjugate to $f$. As examples of this result, we analyze real analytic stable families of hyperbolic and  flexible Latt\`es maps with real coefficients along with numerous families of degree $d$ real maps of real entropy $\log(d)$. The latter discussion moreover entails a complete classification of maps of maximal real entropy.
\end{abstract}
\maketitle

\renewcommand{\baselinestretch}{0.9}\normalsize
\small
\tableofcontents
\normalsize
\renewcommand{\baselinestretch}{1}\normalsize

\section{Introduction}\label{introduction}
This article focuses on the entropy of real rational maps as a function defined on an appropriate moduli space.  It is well known that  a rational map $f:\hat{\Bbb{C}}\rightarrow\hat{\Bbb{C}}$ of degree $d\geq 2$ has topological entropy $\log(d)$ and admits a unique measure of maximal entropy $\mu_f$ whose support is the Julia set $\mathcal{J}(f)$ \cite{MR684138,MR736568,MR736567}. 
In contrast, for $f\in\Bbb{R}(z)$, the topological entropy 
\begin{equation}\label{real entropy 1}
h_{\Bbb{R}}(f):=h_{\rm{top}}\left(f\restriction_{\hat{\Bbb{R}}}:\hat{\Bbb{R}}\rightarrow\hat{\Bbb{R}}\right)
\end{equation}
of the induced dynamics  on the invariant circle $\hat{\Bbb{R}}:=\Bbb{R}\cup\{\infty\}$ can take any value between 
$0$ and $\log\left(\deg f\right)$.
In this paper, we study the \textit{real entropy} $h_{\Bbb{R}}(f)$
both in relation with the dynamics of the ambient map $f:\hat{\Bbb{C}}\rightarrow\hat{\Bbb{C}}$ and also as $f$ varies in a family of real rational maps.  Our main goal is to prove a \textit{rigidity} statement showing that the real entropy is preserved in  a family of quasi-conformally conjugate real maps. 
Let 
$$\mathcal{M}_d(\Bbb{C}):={\rm{Rat}}_d(\Bbb{C})\big/{\rm{PGL}}_2(\Bbb{C})$$
be the moduli space of one-dimensional holomorphic dynamical systems of degree $d$ where the group ${\rm{PGL}}_2(\Bbb{C})$ of Möbius transformations acts by conjugation on the space ${\rm{Rat}}_d(\Bbb{C})$ of rational maps   of degree $d$.
Consider the real subvariety 
$$\mathcal{M}'_d:={\rm{PGL}}_2(\Bbb{C}).\left({\rm{Rat}}_d(\Bbb{R})\right)\big/{\rm{PGL}}_2(\Bbb{C})$$
of dimension $2d-2$ formed by classes with real representatives. The subvariety 
$\mathcal{M}'_d$
is contained in the real locus  $\mathcal{M}_d(\Bbb{R})$
of the moduli space. We have equality $\mathcal{M}'_d=\mathcal{M}_d(\Bbb{R})$ if and only if $d$ is even; see Proposition \ref{antipodal}.

\begin{theorem}\label{temp1}
The real entropy \eqref{real entropy 1} defines a continuous and surjective function 
$$h_\Bbb{R}:\mathcal{M}'_d-\mathcal{S}'\rightarrow \left[0,\log(d)\right]$$
that for every $f\in\Bbb{R}(z)$ is constant on connected components of   $\mathcal{M}'_d\bigcap{\rm{M}}(f)-\mathcal{S}'$. 
\end{theorem}
\noindent
Here,  
${\rm{M}}(f)$ is the subspace of Möbius conjugacy classes of rational maps which are quasi-conformally conjugate to $f$ globally. The subspace $\mathcal{S}'$ is the 
intersection with 
$\mathcal{M}'_d$ 
of the symmetry locus $\mathcal{S}(\Bbb{C})$, which is the subset of  conjugacy classes of degree $d$ maps that admit  non-trivial Möbius automorphisms. In Proposition \ref{codimension>1} we prove that  
$\mathcal{S}(\Bbb{C})$  
is a closed subvariety of $\mathcal{M}_d(\Bbb{C})$ of
dimension $d-1$; therefore, the domain $\mathcal{M}'_d-\mathcal{S}'$ of $h_\Bbb{R}$ is a real variety of dimension $2d-2$. It is irreducible as a real variety but consists of $d+1$ connected components in its analytic topology; see Proposition \ref{degree components}. The continuity of $h_\Bbb{R}$ follows from \cite{MR1372979}.\\
\indent
The subvariety  $\mathcal{S}'=\mathcal{M}'_d\bigcap\mathcal{S}(\Bbb{C})$ is excluded so that $h_\Bbb{R}$ is well defined:
the dynamics on  $\hat{\Bbb{R}}$ should be independent of the real representative picked from a Möbius conjugacy class so one has to omit real maps admitting \textit{twists}; that is, real rational maps that are Möbius conjugate only over $\Bbb{C}$, and it is well known that twists are always associated with non-trivial automorphisms \cite[\S 4.8]{MR2316407}.

The \textit{dynamical moduli space} ${\rm{M}}(f)$ associated with a rational map $f$ has been introduced and thoroughly studied in \cite{MR1620850}. Any structurally-stable holomorphic family $\left\{f_{\lambda}\right\}_{\lambda\in\Lambda}$ of degree $d$ rational maps including $f$ comes with a canonical map  
$\Lambda\rightarrow{\rm{M}}(f)$
that to each parameter $\lambda$ assigns the conformal conjugacy class 
$\langle f_{\lambda}\rangle\in {\rm{M}}(f)\subset\mathcal{M}_d(\Bbb{C})$.
This is due to the fact that the notions of quasi-conformal and topological stability coincide \cite[Theorem 7.1]{MR1620850}. Hence the moduli spaces ${\rm{M}}(f)$ provide a natural framework for Theorem \ref{temp1} which is concerned with the constancy of the real entropy function $h_\Bbb{R}$  in a structurally-stable family of real rational maps. Nevertheless, as it is the Julia set that captures the relevant part of the dynamics of a rational map, it is not surprising that an analogous result can be formulated for families of real maps that are $\mathcal{J}$-\textit{stable} in the sense of \cite{MR732343}: 
\begin{proposition}\label{temp1'}
The real entropy is constant along any analytic $1$-parameter $\mathcal{J}$-stable family $\left\{f_t\right\}_{t\in [0,1]}$ of real rational maps of degree $d$.
\end{proposition}

Hyperbolic components in $\mathcal{M}_d(\Bbb{C})$ are examples of $\mathcal{J}$-stable components of the moduli space. As for the entropy values over \textit{real hyperbolic components}, the claim of Proposition \ref{temp1'} will be directly proved via invoking the  \textit{kneading theory} of Milnor and Thurston \cite{MR970571} in \S\ref{hyperbolic subsection}. A non-hyperbolic but algebraic example of a structurally stable family contained in a single isentrope is the family of flexible Latt\`es maps for which the real entropy can be calculated explicitly; see \S\ref{Lattes subsection}.  There are also interesting families of quasi-conformally constant real rational maps  where $h_\Bbb{R}(f)$ attains its maximum $\log\left(\deg f\right)$. Characterizing such rational maps turns out to be an interesting problem in its own right: 
\begin{theorem}\label{temp3}
Let $f$ be a real rational map of degree $d\geq 2$ with  $h_\Bbb{R}(f)=\log(d)$. Then the Julia set $\mathcal{J}(f)$
is completely contained in $\hat{\Bbb{R}}$ and either coincides with it or is a subinterval or a Cantor subset of it. \\
\indent
Moreover, in terms of Blaschke products 
\begin{equation}\label{Blaschke temp1}
{\rm{e}}^{2\pi{\rm{i}}c}\prod_{i=1}^{d}\left(\frac{z-a_i}{1-\bar{a_i}z}\right)\quad 
\left(|a_1|,\dots,|a_d|<1;\,  \, c\in\Bbb{R}/\Bbb{Z}\right)
\end{equation}
that preserve the unit circle instead of the real circle $\hat{\Bbb{R}}$:
\begin{enumerate}[(i)]
\item The map in \eqref{Blaschke temp1} is hyperbolic with its Julia set a Cantor subset of the unit circle if and only if there is a
$\theta_0\in\Bbb{R}/\Bbb{Z}$ satisfying 
\begin{equation}\label{temp equation 1}
\sum_{i=1}^d\frac{1-|a_i|^2}{|{\rm{e}}^{2\pi{\rm{i}}\theta_0}-a_i|^2}<1
\end{equation}
for which
\begin{equation}\label{temp equation 2}
c\equiv (d+1).\theta_0-\frac{1}{\pi}\sum_{i=1}^d\arg({\rm{e}}^{2\pi{\rm{i}}\theta_0}-a_i)\pmod{\Bbb{Z}}.
\end{equation}
\item  Up to conjugation with the \textit{Cayley transform} $\left(z\mapsto\frac{z-{\rm{i}}}{z+{\rm{i}}}\right)$ and a biholomorphism of the unit disk, hyperbolic degree $d$ maps $f$ with $\mathcal{J}(f)=\hat{\Bbb{R}}$  are the same as  Blaschke products
\begin{equation}\label{Blaschke temp2}
z.\prod_{i=1}^{d-1}\left(\frac{z-a_i}{1-\bar{a_i}z}\right)\quad 
\left(|a_1|,\dots,|a_{d-1}|<1\right)
\end{equation}
or their post-composition with $z\mapsto\frac{1}{z}$.  
\item Every non-hyperbolic degree $d$ map with $\mathcal{J}(f)=\hat{\Bbb{R}}$ after conjugation with a suitable  real Möbius map would be a map of the form 
$\epsilon\left(z+\frac{P(z)}{Q(z)}\right)$ with non-real critical points
where $\epsilon\in\{\pm1\}$ and $P$, $Q$ are coprime real polynomials satisfying $\deg P\leq d-1$ and $\deg Q=d-1$ with the  inequality strict when $\epsilon=+1$. 
\item Every degree $d$ map $f$ with the Julia set a subinterval of $\hat{\Bbb{R}}$ is semi-conjugate to a map from one of the classes appeared in the last two parts: up to a real Möbius conjugacy, $f$ is either the quotient of a map such as
\begin{equation}\label{Blaschke temp3}
\pm z.\prod_{i=1}^{d-1}\left(\frac{z-a_i}{1-\bar{a_i}z}\right)\quad 
\left(\begin{matrix}
|a_1|,\dots,|a_{d-1}|<1\\
\left\{a_1,\dots,a_{d-1}\right\}=\left\{\overline{a_1},\dots,\overline{a_{d-1}}\right\}
\end{matrix}
\right)
\end{equation}
by the action of $z\mapsto\frac{1}{z}$ that commutes with it;  
or is the quotient of a map of the form $\epsilon\left(z+\frac{P(z)}{Q(z)}\right)$ appeared in (iii) by the action of $z\mapsto -z$ provided that it commutes with $\epsilon\left(z+\frac{P(z)}{Q(z)}\right)$; namely, either $P(z)$ is odd and $Q(z)$ is even or vice versa.
\end{enumerate}
\end{theorem}
Part (i) furnishes us with a $\mathcal{J}$-stable family of rational maps whose Julia sets are Cantor subsets of the unit (or equivalently real) circle while the other parts classify all rational maps of maximal real entropy whose Julia sets are one of the other alternatives, namely a real subinterval or the whole real circle. Rational maps with circle or interval Julia sets have also been studied in \cite{MR2833563} and \cite[\S 25]{MR1504787}.

\indent \textbf{Motivation.}
An important question regarding the real entropy function on the moduli space of real rational maps is the nature of its level sets,  called isentropes. We view this article as a first step towards the study of the level sets of the function 
$h_\Bbb{R}$ defined on the moduli of real rational maps. There is an extensive literature on the monotonicity of entropy for various families of polynomials where the central problem is the connectedness of isentropes. The monotonicity was first established by Milnor and Thurston for the quadratic polynomial family  in \cite{MR970571} (also see \cite{MR762431,MR1351519}), and for bimodal cubic polynomials by Milnor and Tresser in \cite{MR1736945}. The general case of degree $d$ polynomials with $d$ real non-degenerate critical points was  settled  in \cite{MR3264762} by van Strien and Bruin. The setting of rational maps is different from the polynomial setting because we are dealing with circle maps 
 $f\restriction_{\hat{\Bbb{R}}}:\hat{\Bbb{R}}\rightarrow\hat{\Bbb{R}}$ instead of interval maps: In all aforementioned references one essentially deals with a boundary-anchored  map of an interval outside which the orbits either escape to infinity or converge to a cycle of period at most two
\cite[Theorem 3.2]{MR1736945}.\footnote{It has to be mentioned  that in general the smallest subinterval that carries all the ``entropy data'', namely the smallest compact interval that includes the real points of the filled Julia set, is not necessarily invariant; see the discussion in  \cite[\S 2]{MR1181083}.}
Of course, if the restriction is not surjective, then
$f\restriction_{\hat{\Bbb{R}}}:\hat{\Bbb{R}}\rightarrow\hat{\Bbb{R}}$
can be replaced with an (not necessarily boundary-anchored) interval map of the same entropy; as
\begin{equation}\label{real entropy}
h_{\Bbb{R}}(f)=h_{\rm{top}}\left(f\restriction_{\hat{\Bbb{R}}}:\hat{\Bbb{R}}\rightarrow\hat{\Bbb{R}}\right)
=h_{\rm{top}}\left(f\restriction_{f(\hat{\Bbb{R}})}:f(\hat{\Bbb{R}})\rightarrow f(\hat{\Bbb{R}})\right).
\end{equation}
For example, when the rational map $f$ is quadratic; $f\restriction_{\hat{\Bbb{R}}}:\hat{\Bbb{R}}\rightarrow\hat{\Bbb{R}}$ is either a covering map (and thus of entropy $\log(2)$ according to 
\cite[Theorem 1$'$]{MR579440}) or is not surjective in which case one can work with the interval map $f\restriction_{f(\hat{\Bbb{R}})}:f(\hat{\Bbb{R}})\rightarrow f(\hat{\Bbb{R}})$ instead; cf. \cite[\S 10]{MR1246482}.\footnote{For higher degrees there are rational maps that take the circle $\hat{\Bbb{R}}$ onto itself via a map which is not covering; e.g., 
an odd degree polynomial map with at least one non-degenerate real critical point or a map such as
$x\mapsto\frac{P(x)}{(x^2-1)^2}$ where $P(x)$ is a real polynomial of degree $d\geq 4$ with $P(1).P(-1)<0$.}
It turns out that in general the domain $\mathcal{M}'_d-\mathcal{S}'$ of the real entropy function is disconnected 
(Proposition \ref{degree components}) and thus the monotonicity has to be studied among maps whose restrictions to $\hat{\Bbb{R}}$ are of a common  topological degree. 
There are even  smaller  natural analytic domains in $\mathcal{M}'_d$ that are dynamically defined according to the degree along with the modality of  
$f\restriction_{f(\hat{\Bbb{R}})}$; and the monotonicity of the restriction of $h_\Bbb{R}$ to such regions is still worthy to investigate.\footnote{Modality can be considered to be finer than degree; for instance a polynomial self-map of the real axis is of topological degree $0$ or $\pm 1$ according to the parity of the degree of the polynomial while it can have various numbers of laps. In the literature, the monotonicity problem has been  studied mostly for a special class of degree $d$ real polynomial maps whose lap numbers are $d$ as well; see \cite{MR3264762}, \cite{MR1736945} and \cite{MR970571}.} As an example, for $d=2$ a natural partition of $\mathcal{M}'_2=\mathcal{M}_2(\Bbb{R})\cong\Bbb{R}^2$
to \textit{degree} $\pm 2$, \textit{monotonic},  \textit{unimodal} and \textit{bimodal} regions has been outlined in \cite[\S 10]{MR1246482}.
The entropy behavior over $\mathcal{M}'_2$ is the subject of the article \cite{2019arXiv190103458F} that focuses on the monotonicity problem.

\indent \textbf{Outline.}
Putting the real entropy function  $h_\Bbb{R}$ (as described in Theorem \ref{temp1}) on a firm footing is the main goal of second  section and is done in \S\S\ref{symmetry subsection},\ref{Blaschke subsection},\ref{properties subsection}. 
The last subsection of \S\ref{setup} discusses how
$\mathcal{M}'_d$
 is related to the real locus  $\mathcal{M}_d(\Bbb{R})$. 
\\ 
\indent
The main result, Theorem  \ref{temp1}, is formulated in the context of the theory of  Teichmüller and moduli spaces of rational maps developed by McMullen and Sullivan in \cite{MR1620850}. After a brief overview of this theory, we prove the theorem in \S\ref{proof section} and discuss its implications for the dimension of isentropes; see Corollary \ref{main-corollary}. \\
\indent
The fourth section is devoted to various examples of (structurally or $\mathcal{J}$-) stable families of rational maps with real coefficients that, according to our rigidity result, should come with the same real entropy. 
First, in \S\ref{hyperbolic subsection} we directly prove that $h_\Bbb{R}$ is locally constant on the real locus of a hyperbolic component in $\mathcal{M}_d(\Bbb{C})$; see Theorem \ref{entropy constant over hyperbolic}.  In \S\ref{Lattes subsection} we calculate the real entropy for families of flexible Latt\`es maps with real coefficients. Finally, in \S\ref{maximal subsection} we turn to the isentrope 
$h_\Bbb{R}=\log(d)$ where $h_\Bbb{R}:\mathcal{M}'_d-\mathcal{S}'\rightarrow \left[0,\log(d)\right]$ attains its maximum. A careful analysis in Theorems \ref{extremal classification} and \ref{extremal classification 1} of the dynamics induced on the real circle by a map $f\in{\rm{Rat}}_d(\Bbb{R})$ with $h_\Bbb{R}(f)=\log(d)$  culminates in families of degree $d$ maps of real entropy $\log(d)$ outlined in Theorem \ref{temp3}. Each of these families is structurally stable for a generic choice of parameters and (once projected to $\mathcal{M}'_d$) parametrizes the real locus of a single dynamical moduli space 
${\rm{M}}(f)$ associated with a degree $d$ real map $f$ with $h_\Bbb{R}(f)=\log(d)$. \\
\indent There are two appendices. The first one proves a structure theorem for the \textit{real Julia set} 
$\mathcal{J}_{\Bbb{R}}(f)=\mathcal{J}(f)\cap\hat{\Bbb{R}}$ of a real rational map $f$ which is defined as the intersection of the complex Julia set with the extended real line. The second appendix addresses the entropy values realized by rational maps  $\hat{\Bbb{R}}\rightarrow\hat{\Bbb{R}}$ of a fixed topological degree.

\textbf{Notation and Terminology.} The \textit{non-wandering set} of a topological dynamical system 
$f:X\rightarrow X$ is denoted by ${\rm{NW}}\left(f:X\rightarrow X\right)$. Assuming $X$ is compact, it is a standard fact that the subsystem obtained from restricting $f$ to the non-wandering set is of full entropy
\cite{MR1255515}. When $X$ is an interval, a \textit{lap} of $f$ is defined to be a maximal monotonic subinterval and the number of laps (i.e. modality$+1$) is denoted by $l(f)$ and is called the \textit{lap number}. 
It is a standard fact that the entropy of a multimodal (i.e. piecewise monotone or of finite modality) interval map is the exponential growth rate 
of the number of laps of its iterates \cite[Theorem 1]{MR579440}:
\begin{equation}\label{lap number}
h_{\rm{top}}(f)=\lim_{n\to\infty}\frac{1}{n}\log\left(l(f^{\circ n})\right)=\inf_{n}\,\frac{1}{n}\log\left(l(f^{\circ n})\right).
\end{equation}
The result remains valid for a multimodal self-map of the circle $S^1$ as long as  ``laps''
are interpreted as those of the (possibly discontinuous) transformation of $[0,1)$ obtained from conjugating the  transformation of $S^1$ with the bijection $[0,1)\rightarrow S^1: x\mapsto {\rm{e}}^{2\pi{\rm{i}}x}$; see
\cite[Theorem 1$^{'}$]{MR579440} for details. In particular, the $n^{\rm{th}}$ iterate of a degree $d$ covering $S^1\rightarrow S^1$ lifts to a self-map of $[0,1)$ with $d^n$ continuous pieces, all of them monotonic. Its entropy is thus $\log(d)$.
\\
\indent
A metric space is called a \textit{Cantor space} if it is compact, perfect (i.e. without an isolated point) and totally disconnected. 
We refer the reader to the book \cite{MR1886084} for the terminology and background on point-set topology that will be used in this paper. 
\\
\indent
Spaces $\hat{\Bbb{C}}$ and $\hat{\Bbb{R}}$ denote the compactifications $\Bbb{C}\cup\{\infty\}$
and $\Bbb{R}\cup\{\infty\}$ of $\Bbb{C}$ and $\Bbb{R}$ respectively and we use $z$ and $x$ for the coordinates on them. When the Riemann sphere is considered as a complex algebraic curve, the notation 
$\Bbb{P}^1(\Bbb{C})$ is used instead.  The degree of a rational map is denoted by $d$ and is always assumed to be greater than one. The Julia set of a rational map $f$ is shown by $\mathcal{J}(f)$.\\
\indent 
For notations related to the moduli space of rational maps, we mainly follow \cite{MR1246482} and \cite{MR2316407}: the 
\textit{moduli space  of  rational maps of degree $d$} is an affine  variety $\mathcal{M}_d$ with a model over the rational numbers constructed by Silverman in  \cite{MR1635900} as the geometric quotient ${\rm{Rat}}_d\big/{\rm{PSL}}_2$ where ${\rm{Rat}}_d\subset\Bbb{P}^{2d+1}$ is the parameter space of degree $d$ rational maps.\footnote{For technical reasons, Silverman works with ${\rm{PSL}}_2$ instead of ${\rm{PGL}}_2$. Of course, it does not make a difference over the algebraically closed field $\Bbb{C}$ as 
${\rm{PSL}}_2(\Bbb{C})={\rm{PGL}}_2(\Bbb{C})$.}  Given a variety $V$ over a field $K$, the set of its $K'$-points  is written as $V(K')$ for any  larger field $K'\supseteq K$.  So here it makes sense to write $\mathcal{M}_d(K)$ for any subfield $K$ of $\Bbb{C}$, and the complex variety $\mathcal{M}_d(\Bbb{C})$ coincides with the orbifold 
${\rm{Rat}}_d(\Bbb{C})\big/{\rm{PGL}}_2(\Bbb{C})$ consisting of Möbius conjugacy classes of degree $d$ rational maps on the Riemann sphere. The conjugacy class of a rational map $f$ will be denoted by $\langle f\rangle$. The subvariety $\mathcal{S}(\Bbb{C})$ of $\mathcal{M}_d(\Bbb{C})$, called the \textit{symmetry locus}, is the subvariety determined by rational maps $f$ for which the group ${\rm{Aut}}(f)$ of Möbius transformations commuting with $f$ is non-trivial.
The rational map $\bar{f}$ is the map obtained by applying the complex conjugation to coefficients of $f\in\Bbb{C}(z)$. For the broader class of complex-valued functions on $\hat{\Bbb{C}}$, we use the notation $\tilde{f}$ instead which is the result of conjugating $z\mapsto f(z)$ with $z\mapsto\bar{z}$; cf. \eqref{involution}. As for special transformations, the \textit{antipodal involution} and the \textit{Joukowsky transform} are given by  $\gamma:=z\mapsto\frac{-1}{\bar{z}}$ and 
$\mathfrak{j}:=z\mapsto z+\frac{1}{z}$ respectively. Finally, in speaking of quasi-conformal (q.c. for short) homeomorphisms of the Riemann sphere, the dilatation is always denoted by $\mu$.

\section{The real entropy function}\label{setup}
The real entropy  \eqref{real entropy 1} can be considered as a well defined continuous function on a certain open submanifold of $\mathcal{M}'_d\subset\mathcal{M}_d(\Bbb{C})$ obtained from omitting those ${\rm{PGL}}_2(\Bbb{C})$-conjugacy classes of real maps that include more than one ${\rm{PGL}}_2(\Bbb{R})$-conjugacy class. Studying this domain is the subject of the first two subsections \S\S\ref{symmetry subsection},\ref{Blaschke subsection}. Properties of the real entropy function stated in Theorem \ref{temp1} will be proved in  Proposition \ref{dimension} of \S\ref{properties subsection}. 
In \S\ref{comparison subsection} we address the natural question of  how the real locus $\mathcal{M}_d(\Bbb{R})$ of  $\mathcal{M}_d(\Bbb{C})$ is related to the locus $\mathcal{M}'_d$ determined by rational maps that actually admit a model over $\Bbb{R}$; see Proposition \ref{antipodal}.

\subsection{Symmetries}\label{symmetry subsection}
The real entropy is assigned  to rational maps with real coefficients, so it is natural to form the subspace 
\begin{equation}\label{M'}
\mathcal{M}'_d:={\rm{PGL}}_2(\Bbb{C}).\left({{\rm{Rat}}_d}(\Bbb{R})\right)\big/{\rm{PGL}}_2(\Bbb{C})
=\left\{\langle f\rangle\in\mathcal{M}_d(\Bbb{C}) \,|\, f\in\Bbb{R}(z)\right\}
\end{equation}
consisting of conjugacy classes in $\mathcal{M}_d(\Bbb{C})$ that contain a real map. The real entropy $h_\Bbb{R}$ is not well defined as a function on the entirety of this real subvariety because the ambiguity in picking a real representative can lead to different entropy values.
\begin{example}\label{symmetry locus entropy}
Given $\mu\in\Bbb{R}-\{0\}$, real quadratic rational  maps $\frac{1}{\mu}\left(z\pm\frac{1}{z}\right)$   are conjugate via 
$z\mapsto{\rm{i}}z$ but exhibit quite different dynamical behavior on the real circle. The critical points of 
$\frac{1}{\mu}\left(z-\frac{1}{z}\right)$ are not real, so  
it induces a degree two covering 
$x\mapsto\frac{1}{\mu}\left(x-\frac{1}{x}\right)$ of $\hat{\Bbb{R}}$
whose entropy is therefore $\log(2)$.
On the other hand, the topological  entropy of 
$x\mapsto\frac{1}{\mu}\left(x+\frac{1}{x}\right)$ vanishes: 
for $|\mu|\leq 1$ every orbit is attracted to the fixed point $\infty$ of multiplier $\mu$; for $\mu>1$ orbits in the invariant interval $(0,\infty)$ tend to the attracting fixed point $\frac{1}{\sqrt{\mu-1}}$ while those in the invariant interval 
$(-\infty,0)$ tend to the attracting fixed point $-\frac{1}{\sqrt{\mu-1}}$; and finally, for $\mu<-1$ there is no finite real fixed point and any point of $\hat{\Bbb{R}}$, other than the fixed point $\infty$ and its preimage $0$,  converges under iteration to the $2$-cycle consisting of $\pm\frac{1}{\sqrt{-\mu-1}}$ whose multiplier 
is $\left(\frac{2+\mu}{\mu}\right)^2<1$. 
\end{example}
The problem with the pair $\frac{1}{\mu}\left(z\pm\frac{1}{z}\right)$ from the preceding example is that they are conjugate over the complex numbers but not over the reals. We are interested in the dynamics of $f\restriction_{\hat{\Bbb{R}}}$ and that is invariant only under real conjugacies, i.e. elements of ${\rm{PGL}}_2(\Bbb{R})$. In the literature of arithmetic dynamics, such examples of rational maps over a field $K$ which are conjugate only over a strictly larger field $K'$ are called \textit{twists} and they happen only if maps admit symmetries in ${\rm{PGL}}_2(\overline{K})$; see \cite[\S\S 4.7,4.8,4.9]{MR2316407} for details. In our context, only twists  over $\Bbb{R}$ are relevant.

\begin{proposition}\label{symmetry locus}
Fix  $f\in\Bbb{R}(z)$ with ${\rm{Aut}}(f)=\{\mathbf{1}\}$.
If $g\in\Bbb{R}(z)$ is ${\rm{PGL}}_2(\Bbb{C})$-conjugate to $f$, then $f,g$ are in fact ${\rm{PGL}}_2(\Bbb{R})$-conjugate.
\end{proposition}

\begin{proof}
{Assume the contrary;  the complex conjugacy class $\langle f\rangle$  contains at least two distinct real conjugacy classes. This implies the existence of a Möbius transformation 
$\alpha\in {\rm{PGL}}_2(\Bbb{C})-{\rm{PGL}}_2(\Bbb{R})$ for which $\alpha\circ f\circ\alpha^{-1}\in\Bbb{R}(z)$. But then taking complex conjugates implies that 
$$\bar{\alpha}\circ f\circ\overline{\alpha^{-1}}=\alpha\circ f\circ\alpha^{-1}\Rightarrow 
(\alpha^{-1}\circ\bar{\alpha})\circ f\circ(\alpha^{-1}\circ\bar{\alpha})^{-1}=f.$$
Hence the non-identity Möbius map  $\alpha^{-1}\circ\bar{\alpha}$ lies in ${\rm{Aut}}(f)$; a contradiction.}
\end{proof}
Thus, in order to have a well defined real entropy function,  we should remove from the space $\mathcal{M}'_d$ in \eqref{M'} 
the subspace of complex conjugacy classes of real maps with non-trivial Möbius symmetries; the subspace which will be denoted by
\begin{equation}\label{S'}
\mathcal{S}':=\left\{\langle f\rangle\in\mathcal{M}_d(\Bbb{C}) \,|\, f\in{\rm{Rat}}_d(\Bbb{R}), {\rm{Aut}}(f) \neq \{\mathbf{1}\}\right\}
\end{equation}
herein. This is obviously the intersection of the complex symmetry locus 
\begin{equation}\label{S}
\mathcal{S}(\Bbb{C})=\left\{\langle f\rangle\in\mathcal{M}_d(\Bbb{C})\,|\, f\in{\rm{Rat}}_d(\Bbb{C}), {\rm{Aut}}(f)\neq\{\mathbf{1}\}\right\}
\end{equation}
with $\mathcal{M}'_d$.
So we arrive at a well defined entropy function:
\begin{definition}\label{the real entropy function}
For any $d\geq 2$, \textit{the real entropy function} is defined as 
\begin{equation}\label{generalized function}
\begin{cases}
h_{\Bbb{R}}:\mathcal{M}'_d-\mathcal{S}'\rightarrow \left[0,\log(d)\right]\\
\langle f\rangle\mapsto h_{\rm{top}}\left(f\restriction_{\hat{\Bbb{R}}}:\hat{\Bbb{R}}\rightarrow\hat{\Bbb{R}}\right)
\end{cases}
\quad (f\in\Bbb{R}(z)).
\end{equation}
\end{definition}
The domain of definition is 
$\left\{\langle f\rangle\in\mathcal{M}_d(\Bbb{C}) \,|\, f\in\Bbb{R}(z), {\rm{Aut}}(f)=\{\mathbf{1}\}\right\}$
and the codomain is $\left[0,\log(d)\right]$ since 
$$ 
h_{\Bbb{R}}(f)=h_{\rm{top}}\left(f\restriction_{\hat{\Bbb{R}}}:\hat{\Bbb{R}}\rightarrow\hat{\Bbb{R}}\right)
\leq
h_{\rm{top}}\left(f:\hat{\Bbb{C}}\rightarrow\hat{\Bbb{C}}\right)=\log(d).
$$

It would be also useful to know the dimension of the symmetry locus; this is the content of the next proposition. See \cite{MR3709645} for a more refined version that determines the dimension of the locus in $\mathcal{M}_d(\Bbb{C})$ of rational maps with a prescribed automorphism group.

\begin{proposition}\label{codimension>1}
For all $d\geq 2$ the dimension of the symmetry locus 
\eqref{S}
is $d-1$ and, as a subvariety of $\mathcal{M}_d$ (which is a variety over $\Bbb{Q}$), it has a model over $\Bbb{Q}$ (so makes sense to talk of its 
$\Bbb{R}$-points). Moreover, 
$\dim_\Bbb{R}\mathcal{S}'=\dim_\Bbb{R}\mathcal{S}(\Bbb{R})=d-1$.
\end{proposition}
\begin{proof}
{Let $f\in\Bbb{C}(z)$ be a rational map admitting a non-trivial automorphism 
$\alpha\in{\rm{PGL}_2}(\Bbb{C})$. Since    
${\rm{Aut}}(f)$ is finite (\cite[Proposition 4.65]{MR2316407}), 
after conjugation with an appropriate Möbius map one may assume that the cyclic subgroup generated by 
$\alpha$ has a generator of the form
$\omega_n:=z\mapsto{\rm{e}}^{\frac{2\pi{\rm{i}}}{n}}z$ where $n\geq 2$. The elements of 
${\rm{Rat}}_d(\Bbb{C})$ commuting with $\omega_n$ form the following union:
$$\bigcup_{0\leq r\leq n}\left\{\frac{\sum_{0\leq i\leq d,\,n\mid i-r-1}a_iz^i}{\sum_{0\leq i\leq d,\,n\mid i-r}b_iz^i}\in {\rm{Rat}}_d(\Bbb{C})\right\}.$$
Observe that $n$ cannot exceed $d+1$ because otherwise, the set above will be vacuous: there must be $i,j\in\{0,\dots,d\}$ with both $i-r-1$ and $j-r$ divisible by $n$ and these two numbers differ by at most $d+1$. The dimension of each of the sets in this union is at most $d$:  the number of coefficients $a_i, b_i$ appearing is precisely $d+1$ for $n=2$  and at most $2\left\lceil\frac{d+1}{3}\right\rceil\leq d+1$  for $n\geq 3$. Thus the projection from the affine space of coefficients into the quasi-projective variety  ${\rm{Rat}}_d(\Bbb{C})$ yields a  $d$-dimensional subset of 
${\rm{Rat}}_d(\Bbb{C})$. 
There is  one degree of freedom due to conjugation with scaling maps which preserves the forms appeared above. Hence the symmetry locus $\mathcal{S}(\Bbb{C})$ is of dimension at most $d-1$. Indeed, the equality is achieved; otherwise the generic fiber of the projection map from the aforementioned $d$-dimensional subspace of ${\rm{Rat}}_d(\Bbb{C})$ 
into $\mathcal{M}_d(\Bbb{C})$ is of dimension at least two. This means that there is a rational map $f\in\Bbb{C}(z)$ commuting with an $\omega_n$ such that for any $\alpha$ from a two-dimensional subset $Z$ of ${\rm{PGL}}_2(\Bbb{C})$, $\alpha\circ f\circ\alpha^{-1}$
commute with $\omega_n$ as well. In particular, there is a morphism 
$$
\begin{cases}
Z\to {\rm{Aut}}(f)\\
\alpha\mapsto \alpha^{-1}\circ\omega_n\circ\alpha
\end{cases}
$$
from the two-dimensional variety $Z$ into the finite set ${\rm{Aut}}(f)$. So the morphism has to be constant over a two-dimensional subvariety which cannot be the case as the centralizer of $\omega_n$ in ${\rm{PGL}}_2(\Bbb{C})$ is the subgroup of scaling maps $z\mapsto kz$ which is one-dimensional. \\
\indent
For the last part, notice that according to the above union, $\mathcal{S}(\Bbb{C})$ is the image of the algebraic set 
\footnotesize
$$\bigcup_{2\leq n\leq d+1}\bigcup_{0\leq r\leq n}\left\{\left[a_0:\dots :a_d: b_0:\dots: b_d\right]\in {\rm{Rat}}_d(\Bbb{C})\subset \Bbb{P}^{2d+1}(\Bbb{C})\,|\,
a_i=0\, \text{ if } n\nmid i-r-1;\,\, b_i=0\, \text{ if } n\nmid i-r \right\}$$
\normalsize 
 defined over the rationals.\\
\indent 
Finally, notice that the dimension count for $\mathcal{S}(\Bbb{C})$ implies upper bounds for dimensions of real subvarieties 
$\mathcal{S}'\subseteq \mathcal{S}(\Bbb{R})$:
$$\dim_\Bbb{R}\mathcal{S}'\leq\dim_\Bbb{R}\mathcal{S}(\Bbb{R})\leq\dim_{\Bbb{C}}\mathcal{S}(\Bbb{C})= d-1.
$$
It is not hard to show that the equality is achieved here: The complex dimension of $\mathcal{S}(\Bbb{C})$ coincides with the real dimension of its real locus $\mathcal{S}(\Bbb{R})$ provided that the complex variety $\mathcal{S}(\Bbb{C})$ has a smooth $\Bbb{R}$-point  in a highest dimensional irreducible component. For a generic choice of complex or real  numbers 
$a_{2k+1}$ and $b_{2k}$, 
the automorphism group of the rational map
$$\frac{\sum_{0< 2k+1\leq d}a_{2k+1}z^{2k+1}}{\sum_{0\leq 2k\leq d}b_{2k}z^{2k}}$$
is generated by $z\mapsto -z$ that on coefficients acts via $a_{2k+1}\mapsto -a_{2k+1}$
and $b_{2k}\mapsto b_{2k}$. Hence, assuming that these numbers are furthermore  positive, we get a $d$-dimensional submanifold of 
${\rm{Rat}}_d(\Bbb{R})$ that bijects onto a $(d-1)$-dimensional real submanifold of the set of smooth points of $\mathcal{S}(\Bbb{C})$ which is a subset of $\mathcal{S}'$ because every point of it represents the conjugacy class of a real map. Therefore, 
$ \dim_\Bbb{R}\mathcal{S}(\Bbb{R})\geq \dim_\Bbb{R}\mathcal{S}'\geq d-1$.}
\end{proof}

\subsection{Blaschke products}\label{Blaschke subsection}
The rational maps we are interested in are those with real coefficients; or equivalently, those that preserve the real circle 
$\hat{\Bbb{R}}$. But under a suitable Möbius change of coordinates (e.g. the Cayley transform) $\hat{\Bbb{R}}$ can be identified with the unit circle; and in studying holomorphic maps preserving the unit circle Blaschke products come up naturally. In the next proposition, we use Blaschke products to investigate the domain of the function $h_\Bbb{R}$ from Definition \ref{the real entropy function}; compare with \cite{MR0469386}.
\begin{proposition}\label{degree components}
The domain 
$\mathcal{M}'_d-\mathcal{S}'\subset\mathcal{M}_d(\Bbb{C})$
of the real entropy function \eqref{generalized function}
is an irreducible real variety of dimension $2d-2$. Nevertheless,  in its analytic topology, it decomposes to $d+1$ connected components of the same dimension corresponding to  topological degrees of  self-maps of the circle $\hat{\Bbb{R}}=\Bbb{R}\cup\{\infty\}$
that elements of ${{\rm{Rat}}_d}(\Bbb{R})$ induce:
\begin{equation}\label{disjoint union}
\mathcal{M}'_d-\mathcal{S}'=\bigsqcup_{-d\leq s\leq d,\,2\mid d-s}\left(\mathcal{M}'_{d,s}-\mathcal{S}'\right).
\end{equation}
Here, $\mathcal{M}'_{d,s}$ is the subspace of complex conjugacy classes of degree $d$ real rational maps $f$ for which the topological degree of the 
restriction 
$f\restriction_{\hat{\Bbb{R}}}:\hat{\Bbb{R}}\rightarrow\hat{\Bbb{R}}$
is $s$.
\end{proposition}

\begin{proof}
We have
$$\dim_{\Bbb{R}}\mathcal{M}'_d\leq\dim_{\Bbb{R}}\mathcal{M}_d(\Bbb{R})\leq\dim_{\Bbb{C}}\mathcal{M}_d(\Bbb{C})
=2d-2.$$
The surjective morphism ${{\rm{Rat}}_d}(\Bbb{R})\big/{\rm{PGL}}_2(\Bbb{R})\rightarrow\mathcal{M}'_d$
 is injective on the Zariski open subset of classes without twists (e.g. away from the preimage of the closed subvariety $\mathcal{S}'$).
 The projection also indicates that $\mathcal{M}'_d$ is irreducible, being a surjective image of  ${{\rm{Rat}}_d}(\Bbb{R})$ which is irreducible itself as it is the complement of the resultant hypersurface in $\Bbb{P}^{2d+1}(\Bbb{R})$.
Therefore, $\mathcal{M}'_d$ is irreducible in the Zariski topology; and is of real dimension $2d-2$. The same statements hold for its Zariski open subset $\mathcal{M}'_d-\mathcal{S}'$.\\
\indent
Since the field $\Bbb{R}$ is not algebraically closed, the irreducibility of ${\rm{Rat}}_d(\Bbb{R})$ or $\mathcal{M}'_d-\mathcal{S}'$  does not guarantee their connectedness in the analytic topology. As a matter of fact, maps in ${\rm{Rat}}_d(\Bbb{R})$
whose restrictions to the real circle have different topological degrees cannot lie in the same connected component. 
To see this, notice that the topological degree $s\in\left\{-d,\dots,0,\dots,d\right\}$ of the restriction of $f\in{\rm{Rat}}_d(\Bbb{R})$  to 
 $\hat{\Bbb{R}}$
can be thought of as the value of the integral 
$$
\mathlarger{\int}_{\hat{\Bbb{R}}}f^*\left(\frac{{\rm{d}}x}{\pi(1+x^2)}\right)
=\mathlarger{\int}_{-\infty}^{\infty}\frac{f'(x){\rm{d}}x}{\pi(1+f(x)^2)}
$$
due to the fact that $\frac{{\rm{d}}x}{\pi(1+x^2)}$ is a normalized volume form for the circle  $\hat{\Bbb{R}}=\Bbb{R}\cup\{\infty\}$.
If $f_n\to f$ in the analytic topology of ${\rm{Rat}}_d(\Bbb{R})$ (i.e. the usual topology on the space of coefficients), then 
on the real axis one has
$$\frac{f'_n}{\pi(1+f_n^2)}\restriction_{\Bbb{R}}\,\,\rightrightarrows\,\,\frac{f'}{\pi(1+f^2)}\restriction_{\Bbb{R}}$$ 
for the integrand appeared above. Consequently, the integer-valued function
$$\deg_{\rm{top}}:{\rm{Rat}}_d(\Bbb{R})\rightarrow\left\{-d,\dots,0,\dots,d\right\}$$
is continuous and hence its level sets are disjoint unions of connected components of ${\rm{Rat}}_d(\Bbb{R})$.
The topological degree of a circle map is preserved under conjugation by a diffeomorphism of the circle and Proposition 
\ref{symmetry locus} implies that if real maps $f,g$ give rise to the same class in $\mathcal{M}_d(\Bbb{C})$ away from 
$\mathcal{S}'$, then they are conjugate by a Möbius transformation that preserves $\hat{\Bbb{R}}$. 
Therefore, the topological degree descends to a continuous function 
$$\deg_{\rm{top}}:\mathcal{M}'_d-\mathcal{S}'\rightarrow\left\{-d,\dots,0,\dots,d\right\}$$
any  level set $\mathcal{M}'_{d,s}-\mathcal{S}'$ of which is a union of connected components of  $\mathcal{M}'_d-\mathcal{S}'$. Here  
\begin{equation}\label{component}
\mathcal{M}'_{d,s}:=\left\{\langle f\rangle\in\mathcal{M}_d(\Bbb{C})\mid f\in{\rm{Rat}}_d(\Bbb{R}),
\deg_{\rm{top}}\left(f\restriction_{\hat{\Bbb{R}}}:\hat{\Bbb{R}}\rightarrow\hat{\Bbb{R}}\right)=s\right\}
\end{equation}
as defined in the statement of the proposition.\footnote{It has to be mentioned that in $\mathcal{M}_d(\Bbb{C})$ these subsets might intersect along the lower dimensional subset $\mathcal{S}'$  because a class on the symmetry locus may have real representatives that  restricted to $\hat{\Bbb{R}}$ have
different topological degrees; e.g., the quadratic maps in Example \ref{symmetry locus entropy} that come with topological degrees $\pm 2$ and $0$. Equivalently, $\deg_{\rm{top}}$ is not a well defined function on the entirety of $\mathcal{M}'_d$.}\\
\indent
We next show that $\mathcal{M}'_{d,s}$ is non-vacuous if and only if $d$ and $s\in\left\{-d,\dots,0,\dots,d\right\}$
have the same parity and for any such $s$ the subspace $\mathcal{M}'_d-\mathcal{S}'$ is connected. This will establish the partition 
\eqref{disjoint union} of $\mathcal{M}'_d-\mathcal{S}'$ to its connected components and will conclude the proof. \\
\indent
Let us proceed  with a parametrization of degree $d$ real rational maps with $\deg_{\rm{top}}=s$. It is more convenient to exchange $\hat{\Bbb{R}}$ with $\left\{z\in\Bbb{C}\,\big|\, |z|=1\right\}$ via the Cayley transform
\begin{equation}\label{Cayley}
\begin{cases}
\text{the unit circle}\rightarrow\text{the real circle}\\
z\mapsto x={\rm{i}}\frac{1+z}{1-z}
\end{cases}
\end{equation}
and concentrate on degree $d$ rational maps $f\in\Bbb{C}(z)$ that keep the unit circle $\left\{z\in\Bbb{C}\,\big|\, |z|=1\right\}$
invariant instead. The reflection $z\mapsto\frac{1}{\bar{z}}$ with respect to the unit circle commutes with such an $f$, so  the roots and the poles of $f$ on the Riemann sphere occur in $\left(\text{root},\text{pole}\right)$ pairs like $\left(q,\frac{1}{\bar{q}}\right)$ where $q$ is away from the unit circle. Therefore, assuming that $f$ has $k$ roots $a_1,\dots,a_k$ inside the unit circle and $d-k$ roots
$b_1,\dots,b_{d-k}$ outside of it (both written by multiplicity), its poles would be reciprocals of 
$\bar{a_1},\dots,\bar{a_k}$ outside the unit circle and 
reciprocals of $\bar{b_1},\dots,\overline{b_{d-k}}$ inside it. The rational maps $f(z)$ is thus a scalar multiple of the Blaschke product
$$
\prod_{i=1}^{k}\left(\frac{z-a_i}{1-\bar{a_i}z}\right)
\prod_{j=1}^{d-k}\left(\frac{z-b_j}{1-\bar{b_j}z}\right)
$$
that  preserves the unit circle as well. We conclude that $f(z)$ has a unique description as 
\begin{equation}\label{Blaschke'}
f(z)={\rm{e}}^{2\pi{\rm{i}}c}\prod_{i=1}^{k}\left(\frac{z-a_i}{1-\bar{a_i}z}\right)
\prod_{j=1}^{d-k}\left(\frac{z-b_j}{1-\bar{b_j}z}\right)\quad 
\left(|a_1|,\dots,|a_k|<1;\, |b_1|,\dots,\big|b_{d-k}\big|>1; \, c\in\Bbb{R}/\Bbb{Z}\right).
\end{equation}
 Pulling back 
$\frac{{\rm{d}}x}{\pi(1+x^2)}$
with the Cayley transform \eqref{Cayley} results in the normalized volume form 
$$\frac{1}{2\pi}{\rm{d}}\theta=\frac{1}{2\pi{\rm{i}}}\frac{{\rm{d}}z}{z}$$
for the unit circle. The topological degree $s$ of the restriction $f\restriction_{\{|z|=1\}}:\{|z|=1\}\rightarrow\{|z|=1\}$
of the map in \eqref{Blaschke'} then is
$$\frac{1}{2\pi}\int_{|z|=1}f^*({\rm{d}}\theta)=
\frac{1}{2\pi{\rm{i}}}\int_{|z|=1}f^*\left(\frac{{\rm{d}}z}{z}\right)=
\frac{1}{2\pi{\rm{i}}}\int_{|z|=1}\frac{f'(z)}{f(z)}{\rm{d}}z\stackrel{\text{argument principle}}{=}k-(d-k)=2k-d.$$
We observe that $k=\frac{d+s}{2}$ and an integer $s$ from $\left\{-d,\dots,0,\dots,d\right\}$ can be realized as the topological degree of the restriction precisely when its parity is the same as that of $d$. Now it is easy to verify that there are $d+1$ choices for $s$. \\
\indent
Fixing such an $s$, putting $k$ to be $\frac{d+s}{2}$ in \eqref{Blaschke'} and then varying $a_i$'s, $b_j$'s and $c$ respectively inside, outside and on the unit circle, and 
finally conjugating with the Cayley transform parametrizes a submanifold of ${\rm{Rat}}_d(\Bbb{R})$ diffeomorphic  to 
$\Bbb{R}^{2d}\times S^1$ which is the subspace of real rational maps of degree $d$ with $\deg_{{\rm{top}}}=s$. 
This full-dimensional submanifold of ${\rm{Rat}}_d(\Bbb{R})$  surjects onto the subspace 
$\mathcal{M}'_{d,s}$ of $\mathcal{M}'_d$ appeared in \eqref{component}; a subspace that is therefore connected and of codimension zero. Finally, we argue that removing $\mathcal{S}'$ in 
\eqref{disjoint union}
cannot affect the connectedness of  $\mathcal{M}'_{d,s}$: If $d\geq 3$, invoking Proposition \ref{codimension>1}, the codimension of 
$\mathcal{M}'_{d,s}\bigcap\mathcal{S}'$ in $\mathcal{M}'_{d,s}$
is 
$$(2d-2)-(d-1)\geq 2.$$
Taking preimages under ${\rm{Rat}}_d(\Bbb{R})\rightarrow\mathcal{M}'_d$, we deduce that in the $\deg_{\rm{top}}=s$ component of ${\rm{Rat}}_d$ -- which is a manifold -- the closed subset of maps with non-trivial automorphisms is of codimension at least two, and hence the complement -- which surjects onto $\mathcal{M}'_{d,s}-\mathcal{S}'$ -- is connected. For $d=2$, we rely on the results of 
\cite{MR1246482}: $\mathcal{S}'=\mathcal{S}(\Bbb{R})$ is precisely the curve in
$\mathcal{M}'_2=\mathcal{M}_2(\Bbb{R})=\Bbb{R}^2$
that separates
$\mathcal{M}'_{2,+2}$,
$\mathcal{M}'_{2,-2}$
and 
$\mathcal{M}'_{2,0}$;
see \cite[fig. 15]{MR1246482}. 
\end{proof}

\begin{example}\label{Blaschke1}
For a given $0\neq s\in\left\{\pm 1,\dots,\pm d\right\}$ with $s\equiv d \pmod 2$,
the Blaschke product 
\begin{equation}\label{Blaschke}
f(z)={\rm{e}}^{2\pi{\rm{i}}c}\prod_{i=1}^{\frac{d+s}{2}}\left(\frac{z-a_i}{1-\bar{a_i}z}\right)
\prod_{j=1}^{\frac{d-s}{2}}\left(\frac{z-b_j}{1-\bar{b_j}z}\right)\,
\left(|a_1|,\dots,|a_{\frac{d+s}{2}}|<1;\, |b_1|,\dots,\big|b_{\frac{d-s}{2}}\big|>1; \, c\in\Bbb{R}/\Bbb{Z}\right)\end{equation}
appeared in the proof of Proposition \ref{degree components}
 induces a (unramified) degree $s$ covering of $|z|=1$ provided that 
$a_i$'s and $b_j$'s are sufficiently close to $0$ and $\infty$ respectively. Conjugation with the Cayley transform \eqref{Cayley} then yields a degree $d$ rational map that restricts to a degree $s$ cover $\hat{\Bbb{R}}\rightarrow\hat{\Bbb{R}}$; thus    
a degree $d$ real map of real entropy $\log(|s|)$. For generic choices of  $a_i$'s and $b_j$'s, the corresponding map $f$ will be away from the symmetry locus.
We conclude that $h_\Bbb{R}$ is $\log(|s|)$ over some non-empty analytic open subset of the connected component
$\mathcal{M}'_{d,s}$.\\
\indent
The dynamics of \eqref{Blaschke} on the unit circle is well understood once the map is \textit{expanding} (e.g. as $a_i\to 0$ and $b_j\to\infty$) since then, according to \cite{MR0240824}, it should be conjugate to $z\mapsto z^s$. Such expanding maps do occur once $|s|\geq 2$. Notice that for a Blaschke product $f$ from \eqref{Blaschke} which is expanding on the unit circle, the Julia set indeed contains the unit circle.   Finally, we point out that although when $|s|<d$ the unit circle does not capture all of the entropy of the ambient endomorphism of the Riemann sphere, the dynamics on the circle demonstrates certain ``rigidity'' here in the sense that  once two expanding circle maps of the form \eqref{Blaschke} are absolutely continuously conjugate, the ambient rational maps must be Möbius conjugate  
\cite[Theorem 4]{MR796755}.\footnote{The Blaschke products we consider here are more general  than \cite{MR796755} where all zeros have been assumed to be inside the unit circle. Nevertheless, the proof carries over with no difficulty: by the Lefschetz fixed point formula, expanding Blaschke products of form \eqref{Blaschke} have 
$|s-1|<d+1$ fixed points on the unit circle and hence at least one fixed point inside. This implies that the circle map admits an invariant measure given by the Poisson density with respect to this fixed point \cite[Proposition 2]{MR703758}. A Möbius conjugacy puts this fixed point at the origin and turns the aforementioned invariant measure to the Lebesgue measure on the circle. One can then invoke the rigidity result \cite[Proposition 1]{MR796755}  concerning the expanding circle maps that preserve Lebesgue.  } 
\\
\indent
For $s= 1$ and $d$ odd,  this type of Blaschke product is used in a construction of rational maps with Herman rings; see \cite[\S15]{MR2193309}. In such a situation, 
$$f\restriction_{\{|z|=1\}}:\{|z|=1\}\rightarrow\{|z|=1\}$$
 would be an orientation-preserving analytic diffeomorphism  whose rotation number $\rho$ can be any desired element in $\Bbb{R}/\Bbb{Z}$ after a suitable adjustment of $c$ in \eqref{Blaschke} \cite[\S 15, Lemma 15.3]{MR2193309}. The entropy of this subsystem is then zero. 
 If $\rho$ is rational, then every orbit on the circle converges to a periodic orbit 
(see \cite[Exercise 7.1.4]{MR1963683}) which is a non-repelling orbit of the rational map $f$. 
So the only possible Julia points on the invariant circle are parabolic points. As an example, consider the rational map
$$f(z)=z+\frac{1}{z^2+2}$$
that induces an orientation-preserving self-diffeomorphism of $\hat{\Bbb{R}}$ with rotation number $\rho=0$, where 
the orbit of any real number tends to the parabolic fixed point $\infty$ as $f(x)>x$ for all $x\in\Bbb{R}$.\\
\indent
For irrational $\rho$, by Denjoy's theorem  there is a topological conjugacy between 
$f\restriction_{\{|z|=1\}}:\{|z|=1\}\rightarrow\{|z|=1\}$
 and the irrational rotation $t\in\Bbb{R}/\Bbb{Z}\mapsto t+\rho\in\Bbb{R}/\Bbb{Z}$; in particular, every single orbit of this subsystem is dense. 
So the unit circle is  entirely included in either the Julia set or the Fatou set. 
In the latter situation,  the Fatou component of $f$ having this circle is a fixed rotation domain that has  $|z|=1$ as a leaf of its natural foliation. In particular, the diffeomorphism 
$f\restriction_{\{|z|=1\}}:\{|z|=1\}\rightarrow\{|z|=1\}$
 is real analytically linearizable.  Conversely, if the conjugacy to the rotation map $t\in\Bbb{R}/\Bbb{Z}\mapsto t+\rho\in\Bbb{R}/\Bbb{Z}$ is real analytic, then it can be extended to a small annulus around the unit circle, so the circle is in a Herman ring. There is a complete classification based on Diophantine properties
due to Yoccoz of rotation numbers for which the existence of a real analytic linearization
is guaranteed \cite{MR1924912}. For  irrational rotation numbers $\rho$ that are ``too well approximated'' by rational numbers,  there are real analytic diffeomorphisms of the unit circle of rotation number $\rho$ that do not admit even $C^\infty$ linearizations.  Hence  if in \eqref{Blaschke} we fix $a_1,\dots,a_{\frac{d+1}{2}}$ near  $0$ and  
$b_1,\dots,b_{\frac{d-1}{2}}$ near $\infty$ and then adjust $c$ so that we get such a rotation number for the 
induced self-diffeomorphism of $|z|=1$, then all points on the unit circle would be Julia despite the fact that  the real entropy vanishes. 
\end{example}

\subsection{Properties of $h_\Bbb{R}:\mathcal{M}'_d-\mathcal{S}'\rightarrow \left[0,\log(d)\right]$}\label{properties subsection}

We now prove the first part of Theorem \ref{temp1}.
\begin{proposition}\label{dimension}
For any $d\geq 2$,  the function 
$h_\Bbb{R}:\mathcal{M}'_d-\mathcal{S}'\rightarrow \left[0,\log(d)\right]$
is surjective and continuous (in the analytic topology).
\end{proposition} 

\begin{proof}
 Fix an $s\in\left\{-d,\dots,0\dots,d\right\}$ with $s\equiv d \pmod{2}$. 
As observed in the proof of Proposition \ref{degree components} and also in Example \ref{Blaschke1}, 
the component  $\deg_{{\rm{top}}}=s$  of ${\rm{Rat}}_d(\Bbb{R})$ can be identified with the space of Blaschke products
$$
{\rm{e}}^{2\pi{\rm{i}}c}\prod_{i=1}^{\frac{d+s}{2}}\left(\frac{z-a_i}{1-\bar{a_i}z}\right)
\prod_{j=1}^{\frac{d-s}{2}}\left(\frac{z-b_j}{1-\bar{b_j}z}\right)\,
\left(|a_1|,\dots,|a_{\frac{d+s}{2}}|<1;\, |b_1|,\dots,\big|b_{\frac{d-s}{2}}\big|>1; \, c\in\Bbb{R}/\Bbb{Z}\right)
$$
appeared in \eqref{Blaschke}. A convergence of the parameters  of the above product  inside
$$\{|z|<1\}^{\frac{d+s}{2}}\times\{|z|>1\}^{\frac{d-s}{2}}\times\Bbb{R}/\Bbb{Z}$$
results in the uniform convergence of the corresponding 
holomorphic functions over some thin enough annulus around the unit circle $|z|=1$ 
and therefore, the convergence of the induced maps of the unit circle in the $C^\infty$ topology.     
The continuity claim follows immediately from \cite{MR1372979}: in the $C^1$ topology, the topological entropy is continuous
for circle or interval maps of bounded modality. 
Consequently, the real entropy is a continuous function on the space ${\rm{Rat}}_d(\Bbb{R})$ of  real rational maps of degree $d$. 
This function, after being factored  through the local homeomorphism  from  
the open subset of maps without Möbius  symmetries in ${\rm{Rat}}_d(\Bbb{R})$
onto $\mathcal{M}'_d-\mathcal{S}'$,
 descends to the function
$h_\Bbb{R}:\mathcal{M}'_d-\mathcal{S}'\rightarrow \left[0,\log(d)\right]$ which is thus continuous as well.\\
\indent
Because of the continuity, to obtain the surjectivity of $h_\Bbb{R}$ it suffices to construct a family of 
real rational maps of degree $d$ parametrized over a connected space for which the real entropy gets arbitrarily close to 
both extremes $0$ and $\log(d)$. We invoke the result \cite[Theorem 3.2]{MR1736945}
 that allows us to parametrize the class of boundary-anchored polynomial interval maps of full modality via their critical values:
\begin{itemize}
\item[] \textit{Given numbers $v_1,\dots,v_{d-1}\in [-1,1]$ with $(-1)^i(v_i-v_{i-1})>0$ for every $0<i\leq d$ where $v_0:=1$ and
$v_d:=(-1)^d$, there is a unique boundary-anchored polynomial map $f:[-1,1]\rightarrow [-1,1]$
of degree $d$ that has distinct critical points 
$$-1<c_1<\dots<c_{d-1}<1$$
such that 
$f(c_i)=v_i$ for all $0<i<d$
and $f(-1)=v_0, f(1)=v_d$ on the boundary.}
\end{itemize} 
The space of these tuples $(v_1,\dots,v_{d-1})\in [-1,1]^{d-1}$ is obviously connected. As $v_i\to 0^{+}$ for 
$0<i<d$ even and $v_i\to 0^{-}$ for $0<i<d$ odd, the corresponding maps tend to $x\mapsto (-x)^d$
whose real entropy is zero, whereas when $v_i=(-1)^i$ for each
$0<i<d$, the corresponding degree $d$ polynomial map 
$f:[-1,1]\rightarrow [-1,1]$ 
would have $d$ surjective monotonic pieces;
so the iterate $f^{\circ n}$  needs to have $d^n$ laps and therefore the exponential growth rate of modality of iterates is $\log(d)$.
\end{proof}

\begin{remark}\label{degree-bound}
Notice that for the family exhibited at the end of the preceding proof the topological degree $s$ is zero for $d$ even and is $-1$ for $d$ odd. Indeed, although surjective on $\mathcal{M}'_d-\mathcal{S}'$, $h_\Bbb{R}$ is not surjective on all constituent parts of 
$\mathcal{M}'_d-\mathcal{S}'$ that appeared in \eqref{disjoint union}; a well known result from \cite{MR0458501} asserts that for  a $C^1$ self-map of a compact connected differentiable oriented manifold of degree $s$ the topological entropy is at least $\log\left(\max(|s|,1)\right)$. This is much easier to prove for circle maps; Appendix \ref{appendix B} presents a self-contained proof for the fact that $\left[\log\left(\max(|s|,1)\right),\log(d)\right]$ is the range of the restriction of the entropy function to $\mathcal{M}'_{d,s}-\mathcal{S}'$.
\end{remark}

\begin{remark}
It is natural to ask  what jumps in the values of the function 
$h_{\Bbb{R}}:\mathcal{M}'_d-\mathcal{S}'\rightarrow \left[0,\log(d)\right]$  occur
when one crosses a multi-valued point (which necessarily lies on $\mathcal{S}'$). In general, 
we expect only finitely many jumps at each point of discontinuity because the set of twists of an
$f\in\Bbb{R}(z)$ over $\Bbb{R}$ can be identified with the kernel of the morphism
$$H^1\left({\rm{Gal}}\left(\Bbb{C}/\Bbb{R}\right), {\rm{Aut}}(f)\right)
\rightarrow
H^1\left({\rm{Gal}}\left(\Bbb{C}/\Bbb{R}\right), {\rm{PGL}}_2(\Bbb{C})\right)
$$
(see \cite[Theorem 4.79]{MR2316407} for details) and is therefore finite.
\\
\indent
It is not hard to check  that for $d=2$
$$\mathcal{S}'=\mathcal{S}(\Bbb{R})=
\left\{\left\langle\frac{1}{\mu}\left(z+\frac{1}{z}\right)\right\rangle
=\left\langle\frac{1}{\mu}\left(z-\frac{1}{z}\right)\right\rangle\Big|\,\mu\in\Bbb{R}-\{0\}\right\};$$
see \cite[\S5]{MR1246482}. Therefore, Example \ref{symmetry locus entropy} indicates that the entropy jumps from
$0$ to $\log(2)$ or vice versa as we cross the curve $\mathcal{S}(\Bbb{R})$ in $\mathcal{M}_2(\Bbb{R})$ (which turns out to be 
the affine plane $\Bbb{R}^2$; \cite[\S10]{MR1246482}).\\
\indent 
A more complicated behavior is anticipated in higher degrees since the modalities of real representatives may differ more drastically. 
As an example, consider the  $1$-parameter family $\left\{z^3-az\right\}_{0\leq a\leq 3}$ of real cubics that all admit the symmetry $z\mapsto -z$.
Conjugating with $z\mapsto{\rm{i}}z$, each $z^3-az$ has the alternative real model 
$-z^3-az$ for which the restriction to the real axis is strictly decreasing and hence of zero entropy, whereas  
the original family restricts to $1$-parameter family of self-maps of $\Bbb{R}$ that starts with $x\mapsto x^3$ whose entropy is zero and ends with the (monic) Chebyshev polynomial $x\mapsto x^3-3x$
whose entropy is $\log(3)$ as there is a finite semi-conjugacy $\theta\mapsto 2\cos(2\pi\theta)$ onto it from $\theta\in\Bbb{R}/\Bbb{Z}\mapsto 3\theta\in\Bbb{R}/\Bbb{Z}$. We observe that, unlike the case of $d=2$, every value from the range of 
$h_\Bbb{R}:\mathcal{M}'_3-\mathcal{S}'\rightarrow\left[0,\log(3)\right]$ happens as a jump of values in vicinity of a point of discontinuity.
\end{remark}

\subsection{Comparing $\mathcal{M}'_d$ with the real locus $\mathcal{M}_d(\Bbb{R})$}\label{comparison subsection}
Next,  we elaborate more on the domain 
$\mathcal{M}'_d-\mathcal{S}'=\mathcal{M}'_d-\mathcal{S}(\Bbb{R})$ of $h_\Bbb{R}$ in \eqref{generalized function}; we will establish that it differs from $\mathcal{M}_d(\Bbb{R})-\mathcal{S}(\Bbb{R})$
by an irreducible component which is relevant only in odd degrees. The subspace $\mathcal{M}'_d$ defined in \eqref{M'} is formed by the conjugacy classes of elements of ${{\rm{Rat}}_d}(\Bbb{C})$ for which $\Bbb{R}$ is a \textit{field of definition}. On the other hand, $\Bbb{R}$-points of the moduli space
$\mathcal{M}_d$ are  Möbius conjugacy classes of elements of
${{\rm{Rat}}_d}(\Bbb{C})$ whose \textit{field of moduli} is included in $\Bbb{R}$; see \cite[\S\S 4.4, 4.10]{MR2316407} for the background material. So 
$\mathcal{M}'_d\subseteq\mathcal{M}_d(\Bbb{R})$ and in particular 
$\mathcal{M}'_d-\mathcal{S}'\subseteq\mathcal{M}_d(\Bbb{R})-\mathcal{S}(\Bbb{R})$. The latter containment (and consequently the former) can indeed be strict: for $d=2k+1$ the automorphism group of the rational map 
$\phi(z):={\rm{i}}\left(\frac{z-1}{z+1}\right)^{2k+1}$ is trivial and its field of moduli is $\Bbb{Q}$ while it cannot be defined over the reals; see   \cite[Exercise 4.39, Example 4.85]{MR2316407}. It is worth noting that this question of ``FoM vs. FoD'' is relevant only when  $d$ is odd \cite[Theorem 4.92]{MR2316407}.  \\
\indent
Let us try to see how the real variety
$\mathcal{M}_d(\Bbb{R})-\mathcal{S}(\Bbb{R})$
is related to its (possibly proper) Zariski closed subset 
$\mathcal{M}'_d-\mathcal{S}(\Bbb{R})=\mathcal{M}'_d-\mathcal{S}'$.\footnote{It is clear (at least in the analytic topology) that the subset of points of $\mathcal{M}_d(\Bbb{C})$ corresponding to maps which admit a model over a prescribed closed subfield of $\Bbb{C}$ is closed.}
We claim that classes in the complement 
$\mathcal{M}_d(\Bbb{R})-\left(\mathcal{S}(\Bbb{R})\bigcup\mathcal{M}'_d\right)=\mathcal{M}_d(\Bbb{R})-\left(\mathcal{S}(\Bbb{C})\bigcup\mathcal{M}'_d\right)$ can be represented by \textit{antipodal} maps; i.e. 
maps which commute with the anti-holomorphic involution 
$\gamma(z):=-\frac{1}{\bar{z}}$. Notice that  the example $\phi(z)={\rm{i}}\left(\frac{z-1}{z+1}\right)^{d}$ for $d$ odd that appeared before  is indeed antipodal.   
For more on the dynamics of antipodal-preserving maps see \cite{2015arXiv151201850B}. They are not relevant to our treatment of entropy  as a generic map of this class does not preserve any circle.

\begin{proposition}\label{antipodal}
The real subvariety $\mathcal{M}'_d$ of $\mathcal{M}_d(\Bbb{C})$  coincides with the real locus $\mathcal{M}_d(\Bbb{R})$ of the moduli space for $d$ even while 
for  $d$ odd, the latter 
is reducible and has $\mathcal{M}'_d$ and 
the $(2d-2)$-dimensional real subvariety of conjugacy classes of 
antipodal-preserving maps as irreducible components. Any other irreducible  component of $\mathcal{M}_d(\Bbb{R})$ has to be contained in the symmetry locus 
$\mathcal{S}(\Bbb{R})$ and hence is of dimension at most $d-1$.
\end{proposition}
\begin{proof}
Invoking \cite[Proposition 4.86]{MR2316407}, the obstruction to the field of moduli $\Bbb{R}$ being a field of definition for a map $f\in{\rm{Rat}}_d(\Bbb{C})$ with ${\rm{Aut}}(f)=\{\mathbf{1}\}$ is encoded by the Galois cohomology class  determined by the cocycle ${\rm{Gal}}\left(\Bbb{C}/\Bbb{R}\right)\rightarrow {\rm{PGL}}_2(\Bbb{C})$ defined by $\sigma\mapsto \alpha$
 with $\sigma$  being the complex conjugation and $\alpha$ a Möbius transformation with 
$\alpha\circ f\circ\alpha^{-1}=\bar{f}$
that therefore satisfies the cocycle condition
$\alpha\circ\bar{\alpha}=1$. These transformations have to be considered modulo modification via a 
$1$-coboundary, i.e. replacing $\alpha$ with $u\circ\alpha\circ\bar{u}^{-1}$ that amounts to replacing $f$ with 
$\bar{u}\circ f\circ\bar{u}^{-1}$ from the same conjugacy class. But the group\footnote{Here
${\rm{Br}}(\Bbb{R})$ denotes the \textit{Brauer group} of the field of real numbers.} 
$$H^1\left({\rm{Gal}}\left(\Bbb{C}/\Bbb{R}\right), {\rm{PGL}}_2(\Bbb{C})\right)
\cong H^2\left({\rm{Gal}}\left(\Bbb{C}/\Bbb{R}\right), \Bbb{C}^\times\right)\cong{\rm{Br}}(\Bbb{R})$$
is of order two
and is thus generated by any $1$-cocycle non-cohomologous to a coboundary. An example of such is the  $1$-cocycle associated to the rational map 
$\phi(z)={\rm{i}}\left(\frac{z-1}{z+1}\right)^{2k+1}$
appeared before:
\small
$$-\frac{1}{\phi(\frac{-1}{z})}=-\frac{1}{{\rm{i}}\left(\frac{-\frac{1}{z}-1}{-\frac{1}{z}+1}\right)^{2k+1}}=-{\rm{i}}\left(\frac{z-1}{z+1}\right)^{2k+1}=\bar{\phi}(z);$$
\normalsize
that furnishes us with the $1$-cocycle determined by $\alpha(z)=-\frac{1}{z}$.
It is easy to verify that  this is not  $1$-coboundary; that is, not in the form of $u\circ\bar{u}^{-1}$ for another Möbius transformation; see \cite[p. 209]{MR2316407} for details. Hence 
$\mathcal{M}_d(\Bbb{R})-\left(\mathcal{S}(\Bbb{R})\bigcup\mathcal{M}'_d\right)$  is precisely the subset of classes in
$\mathcal{M}_d(\Bbb{R})-\mathcal{S}(\Bbb{R})$ which admit a representative $f\in\Bbb{C}(z)$ with $\bar{f}\left(-\frac{1}{z}\right)=-\frac{1}{{f(z)}}$. 
Applying the complex conjugation map to both sides, this condition  means that $f$ commutes with the anti-holomorphic involution 
$\gamma(z)=-\frac{1}{\bar{z}}$, a Zariski closed condition over the reals cutting out the \textit{antipodal-preserving} locus of the moduli space. Notice that the constraint automatically guarantees that the field of moduli is inside $\Bbb{R}$ as the complex conjugate map $\bar{f}$ is Möbius conjugate to $f$ via $z\mapsto -\frac{1}{z}$. But there might be such maps which cannot be defined over the reals; e.g., the example $\phi(z)={\rm{i}}\left(\frac{z-1}{z+1}\right)^{d}$ for $d$ odd. 
Consequently,  the real variety $\mathcal{M}_d(\Bbb{R})-\mathcal{S}(\Bbb{R})$ is the union of 
$\mathcal{M}'_d-\mathcal{S}(\Bbb{R})$ -- that as discussed before in \S\ref{Blaschke subsection} is irreducible and of  dimension $2d-2$ -- and the antipodal locus.  
\\
\indent
Here is a simple dimension count for the antipodal-preserving locus in $\mathcal{M}_d(\Bbb{C})$.  Picking a degree 
$d$ map 
$f$
which commutes with 
$\gamma:z\mapsto -\frac{1}{\bar{z}}$,
after a Möbius conjugation, without any loss of generality we may assume that the roots of $f$ lie in the finite plane. If $q$ is a root of $f$, 
$-\frac{1}{\bar{q}}$ has to be a pole. Thus roots and poles can be coupled in pairs such as 
$\left(q_i,-\frac{1}{\bar{q_i}}\right)$. We conclude that 
$f(z)$ is a scalar multiple of a function in the form of 
$\prod_{i=1}^d\frac{z-q_i}{1+\bar{q_i}z}$. Now $f\circ\gamma=\gamma\circ f$ is satisfied for a multiple of such a product if and only if $d$ is odd and the scalar factor is of norm one.\footnote{As a matter of fact,  only odd degrees are relevant here since a classical theorem of Borsuk states that an antipodal-preserving map $S^n\rightarrow S^n$ must be of odd degree. Aside from this topological obstruction, there is  an arithmetic obstruction due to a theorem of Silverman
 which states  that ``FoM=FoD'' whenever the degree is even \cite[Theorem 4.92]{MR2316407}.}
This argument indicates that any antipodal-preserving map of odd degree 
$d$, after a suitable conjugation, can be uniquely written as 
$$u\prod_{i=1}^d\frac{z-q_i}{1+\bar{q_i}z}$$  
where $|u|=1$ and $q_i$'s are complex numbers with $q_i\bar{q_j}\neq -1$. 
This parametrizes 
a connected subspace of real dimension $2d+1$ 
of
${\rm{Rat}}_d(\Bbb{C})$
that projects onto the antipodal-preserving locus in $\mathcal{M}_d(\Bbb{C})$. Consequently,  like $\mathcal{M}'_d$, the antipodal locus is irreducible as well, being a surjective image of the irreducible real algebraic subset $S^1\times (\Bbb{C}^\times)^{d}$ of $\Bbb{A}^{2d+1}(\Bbb{R})$.
Let us find the dimension of the intersection of the whole conjugacy class of a generic map of this kind with the antipodal-preserving locus in
${\rm{Rat}}_d(\Bbb{C})$.
Given an antipodal map 
$f$
with 
${\rm{Aut}}(f)=\{\mathbf{1}\}$,
suppose for an $\alpha\in{\rm{PGL}}_2(\Bbb{C})$ the map $\alpha\circ f\circ\alpha^{-1}$ is antipodal too.
Then one can write:
$$\gamma\circ\left(\alpha\circ f\circ\alpha^{-1}\right)\circ\gamma^{-1}=\alpha\circ f\circ\alpha^{-1}
=\alpha\circ \left(\gamma\circ f\circ\gamma^{-1}\right)\alpha^{-1},$$
which indicates that the Möbius transformation $(\alpha\circ\gamma)^{-1}\circ(\gamma\circ\alpha)$ is an automorphism of 
$f$ and thus $\alpha$ commutes with $\gamma$. It is not hard to verify that any 
Möbius transformation commuting with $z\mapsto -\frac{1}{\bar{z}}$ can be uniquely written 
either as 
$z\mapsto \frac{z+a{\rm{e}}^{2\pi{\rm{i}}r}}{a{\rm{e}}^{2\pi{\rm{i}}s}z-{\rm{e}}^{2\pi{\rm{i}}(r+s)}}$
with 
$r,s\in\Bbb{R}/\Bbb{Z}$, $a>0$
or in one of forms $z\mapsto vz$, $z\mapsto \frac{v}{z}$ where $v$ lies on the unit circle. We conclude that the space of 
Möbius transformation commuting with the antipodal involution is of  real dimension three and thus the antipodal locus in 
$\mathcal{M}_d(\Bbb{C})$ is of dimension  $(2d+1)-3=2d-2$. 
\end{proof}
\noindent
For more on $\mathcal{M}_d(\Bbb{R})$, see \cite{2015arXiv150205306H}.

\section{Rigidity of real entropy}\label{proof section}

Equipped with the definition of the function $h_\Bbb{R}:\mathcal{M}'_d-\mathcal{S}'\rightarrow [0,\log(d)]$
from \S\ref{setup}, we prove Theorem \ref{temp1} in this section. 
Away from the antipodal and symmetry loci, $\mathcal{M}'_d$ can be thought of as the set of fixed points of the involution $\langle f\rangle\mapsto \langle\bar{f}\rangle$
induced by conjugating coefficients in ${\rm{Rat}}_d(\Bbb{C})$.
For future references, we record this involution as acting not only on rational maps but on functions defined on the Riemann sphere:
\begin{definition}\label{involution-def}
For any function  $h:\hat{\Bbb{C}}\rightarrow\Bbb{C}\cup\{\infty\}$,
the function $\tilde{h}:\hat{\Bbb{C}}\rightarrow\Bbb{C}\cup\{\infty\}$
is defined as:
\begin{equation}\label{involution}
\tilde{h}:z\mapsto \overline{h(\bar{z})}.
\end{equation}
\end{definition}
It is easy to check that $h\mapsto\tilde{h}$ is an involution which respects the ring structure of the set of $\Bbb{C}$-valued functions on the Riemann sphere;
takes homeomorphisms to homeomorphisms and finally, it commutes with differential operators $\frac{\partial}{\partial z}$
and $\frac{\partial}{\partial\bar{z}}$:
\begin{equation}\label{differentiation}
\frac{\partial\tilde{h}}{\partial z}=\widetilde{\frac{\partial h}{\partial z}},\quad
\frac{\partial\tilde{h}}{\partial\bar{z}}=\widetilde{\frac{\partial h}{\partial\bar{z}}}.
\end{equation} 

We now briefly review the theory of quasi-conformal  deformations of rational maps developed in 
\cite{MR1620850}. A concise treatment could be found in one of the additional chapters of  \cite{MR2241787}.
Let $f$ be an arbitrary real rational function of degree $d$. The quasi-conformal (q.c.) conjugacy class of $f$, denoted by ${\rm{qc}}(f)$,  consists of those rational maps $g$ which are quasi-conformally conjugate to $f$. The space 
$${\rm{M}}(f):={\rm{qc}}(f)\Big/\text{ Möbius Equivalence}$$  
of Möbius conjugacy classes of maps in ${\rm{qc}}(f)$ is called the \textit{moduli space} of the rational map $f$ and 
the natural map ${\rm{M}}(f)\hookrightarrow\mathcal{M}_d(\Bbb{C})$ is an injection of complex orbifolds.
 In analogy with the theory of the moduli spaces of Riemann surfaces, 
the moduli space  ${\rm{qc}}(f)$ is the quotient of a \textit{Teichmüller space} under a discrete group action:
\begin{equation}\label{teich}
{\rm{T}}(f)=\left\{(g,h)\,|\, g \text{ a rational map}, h \text{ a q.c.-homeomorphism with } h\circ f=g\circ h\right\}\Big/\sim
\end{equation}
where $(g_1,h_1)\sim (g_2,h_2)$ if in the conjugacy 
$$g_2=\left(h_2\circ h_1^{-1}\right)\circ g_1\circ\left(h_2\circ h_1^{-1}\right)^{-1}$$ 
the quasi-conformal homeomorphism $h_2\circ h_1^{-1}$ is isotopic to a Möbius transformation through an isotopy which preserves the conjugacy. In particular, $g_1,g_2$ must be Möbius conjugate. 
There is 
an obvious map 
$$
\begin{cases}
{\rm{T}}(f)\rightarrow{\rm{M}}(f)={\rm{qc}}(f)\Big/\text{ Möbius Equivalence}\\
[(g,h)]\mapsto \langle g\rangle
\end{cases}$$
sending the class $[(g,h)]$ of a pair to the class $\langle g\rangle$. The fiber above the Möbius class of $f$ can be identified with the group of isotopy classes of q.c.-homeomorphisms commuting with $f$ where the isotopy has to remain within this space of q.c.-automorphisms as well. The aforementioned group is called the \textit{modular group} of the rational map $f$ and will be denoted by ${\rm{Mod}}(f)$. Clearly,  ${\rm{Mod}}(f)$ acts (from right) on ${\rm{T}}(f)$ by $[(g,h)].[k]=[(g,h\circ k)]$ and the quotient 
${\rm{T}}(f)\big/{\rm{Mod}}(f)$ can be identified with ${\rm{M}}(f)$ via the projection above. It is known that 
${\rm{Mod}}(f)$ is discrete and acts properly discontinuously on ${\rm{T}}(f)$. Furthermore, there is a description of the space ${\rm{T}}(f)$ as a product of ordinary Teichmüller spaces  based on the dynamics of 
$f:\hat{\Bbb{C}}\rightarrow\hat{\Bbb{C}}$; see
\cite[Theorems 2.2 \& 2.3]{MR1620850}.\\
\indent 
There is a naturally defined entropy function $\widetilde{h_\Bbb{R}}$ on some appropriate subset of ${\rm{T}}(f)$
sending a class $[(g,h)]$ with ${\rm{Aut}}(g)=\{\mathbf{1}\}$ to $h_\Bbb{R}(g)$. This is well defined and fits in the commutative diagram below:
\begin{equation}\label{lift}
\xymatrix{\left\{[(g,h)]\in{\rm{T}}(f)\,|\, g\in\Bbb{R}(z), {\rm{Aut}}(g)=\{\mathbf{1}\}\right\}\ar[d]_{[(g,h)]\mapsto\langle g\rangle} \ar[rd]^{\widetilde{h_\Bbb{R}}} & \\
\mathcal{M}'_d\bigcap{\rm{M}}(f) -\mathcal{S}'\ar[r]_{h_\Bbb{R}} & \left[0,\log(d)\right] }
\end{equation}

Suppose $g$ is another real rational map  quasi-conformally conjugate to $f$; that is, $(g,h)$ determines a class in 
${\rm{T}}(f)$ for a suitable q.c.-homeomorphisms $h$ satisfying $h\circ f=g\circ h$. In order to compare 
the topological entropies of $f\restriction_{\hat{\Bbb{R}}}$ and $g\restriction_{\hat{\Bbb{R}}}$, 
we need to investigate how the  quasi-circle $h(\hat{\Bbb{R}})$ is placed with respect to $\hat{\Bbb{R}}$. Pulling back the complex conjugate map $z\mapsto\bar{z}$ via $h$ yields a reflection $z\mapsto h^{-1}\left(\overline{h(z)}\right)$ commuting with $f$. This differs from the usual reflection $z\mapsto\bar{z}$ -- that also preserves $f\in\Bbb{R}(z)$ -- by a q.c. automorphism of $f$:
\begin{equation}\label{difference}
u(z):=\overline{h^{-1}\left(\overline{h(z)}\right)};
\end{equation}
a map which satisfies $u\left(\overline{u(\bar{z})}\right)=z$ or equivalently, using the notation in \eqref{involution}
\begin{equation}\label{functional equation}
u^{-1}=\tilde{u}
\end{equation}
holds.\footnote{The obstruction to a q.c.-automorphism $u$ of $f$ satisfying the functional equation \eqref{functional equation} being in the form of \eqref{difference} for some appropriate q.c.-automorphism $h$ is encoded by the first cohomology of the group $\Bbb{Z}/2\Bbb{Z}$ with coefficients in the group of q.c.-automorphisms of 
$f$ on which $\Bbb{Z}/2\Bbb{Z}$ acts by the involution \eqref{involution}.}
The homeomorphism $u$ is identity if and only if $\overline{h(\bar{z})}=z$ which in particular indicates that $h$ preserves $\hat{\Bbb{R}}$. Next, we form a continuous map that assigns to a $[(g,h)]\in{\rm{T}}(f)$ (with $g\in\Bbb{R}(z)$)  the class in ${\rm{Mod}}(f)$ of the corresponding automorphism $u$ from \eqref{difference}. The appropriate domain of definition for such a map turns out to be that of the function $\widetilde{h_\Bbb{R}}$ from \eqref{lift}: 
\begin{proposition}
The map
\begin{equation}\label{the map}
\begin{cases}
\left\{[(g,h)]\in{\rm{T}}(f)\,|\, g\in\Bbb{R}(z), {\rm{Aut}}(g)=\{\mathbf{1}\}\right\}\to {\rm{Mod}}(f)\\
[(g,h)]\mapsto \left[u:z\mapsto \overline{h^{-1}\left(\overline{h(z)}\right)}\right]
\end{cases}
\end{equation}
is well defined.
\end{proposition}
\begin{proof}
Pick two equivalent pairs 
$(g_1,h_1)$ and $(g_2,h_2)$ with $g_1,g_2\in\Bbb{R}(z)$.  So there is an isotopy 
$\left\{H_t:\hat{\Bbb{C}}\rightarrow\hat{\Bbb{C}}\right\}_{t\in [0,1]}$ with $H_t\circ g_1=g_2\circ H_t$, $H_1$ being a Möbius transformation and $H_0=h_2\circ h_1^{-1}$.
\begin{equation*}
\xymatrix{\hat{\Bbb{C}} \ar[r]^f\ar[d]^{h_1} & \hat{\Bbb{C}}\ar[d]^{h_1}\\
\hat{\Bbb{C}}\ar[r]^{g_1} & \hat{\Bbb{C}}}
\quad\quad
\xymatrix{\hat{\Bbb{C}} \ar[r]^f\ar[d]^{h_2} & \hat{\Bbb{C}}\ar[d]^{h_2}\\
\hat{\Bbb{C}}\ar[r]^{g_2} & \hat{\Bbb{C}}}
\quad\quad
\xymatrix{\hat{\Bbb{C}} \ar[r]^{g_1}\ar[d]^{H_t} & \hat{\Bbb{C}}\ar[d]^{H_t}\\
\hat{\Bbb{C}}\ar[r]^{g_2} & \hat{\Bbb{C}}}
\end{equation*}
\noindent
Then     
\begin{equation}\label{isotopy}
(t,z)\mapsto\overline{\left(H_t\circ h_1\right)^{-1}\left(\overline{H_t\circ h_1(z)}\right)}\quad (t\in [0,1], z\in\hat{\Bbb{C}})
\end{equation}
is an isotopy through quasi-conformal automorphisms of $f$  
\small
\begin{equation*}
\begin{split}
&f\left(\overline{\left(H_t\circ h_1\right)^{-1}\left(\overline{H_t\circ h_1(z)}\right)}\right)
=\overline{\left(f\circ h_1^{-1}\circ H_t^{-1}\right)\left(\overline{H_t\circ h_1(z)}\right)}
=\overline{\left(\underbrace{\underbrace{f\circ h_1^{-1}}_{=h_1^{-1}\circ g_1}\circ H_t^{-1}}_{=h_1^{-1}\circ H_t^{-1}\circ g_2}\right)\left(\overline{H_t\circ h_1(z)}\right)}\\
&=\overline{\left( h_1^{-1}\circ H_t^{-1}\right)\left(\overline{\underbrace{\underbrace{g_2\circ H_t}_{=H_t\circ g_1}\circ h_1}_{=H_t\circ h_1\circ f}(z)}\right)}
=\overline{\left(H_t\circ h_1\right)^{-1}\left(\overline{H_t\circ h_1\left(f(z)\right)}\right)};
\end{split}
\end{equation*}
\normalsize
that furthermore starts from $u_2:z\mapsto\overline{h_2^{-1}\left(\overline{h_2(z)}\right)}$ and ends with 
$\overline{\left(H_1\circ h_1\right)^{-1}\left(\overline{H_1\circ h_1(z)}\right)};$
a quasi-conformal homeomorphism that we claim coincides with $u_1:z\mapsto\overline{h_1^{-1}\left(\overline{h_1(z)}\right)}$. This holds if the Möbius map $H_1$ lies in ${\rm{PGL}}_2(\Bbb{R})$. But we have a Möbius conjugacy 
$g_2=H_1\circ g_1\circ H_1^{-1}$
between to real maps $g_1,g_2\in\Bbb{R}(z)$ and, by Proposition \ref{symmetry locus}, the failure of  $H_1$ to be real amounts to ${\rm{Aut}}(g_1)\cong {\rm{Aut}}(g_2)$ being non-trivial; a possibility that has been ruled out. 
\end{proof}
The map \eqref{the map} is obviously continuous, but attains its values in the discrete group ${\rm{Mod}}(f)$, so must be constant on each connected components of the domain. 
\begin{proposition}\label{technical}
If  $[(g_1,h_1)], [(g_2,h_2)]$ are mapped to the same element of the modular group via \eqref{the map}, then $h_\Bbb{R}(g_1)$ and $h_\Bbb{R}(g_2)$  must coincide.
\end{proposition} 
Before giving the proof, notice that if the automorphisms 
\begin{equation}\label{auxiliary1}
u_1:z\mapsto\overline{h_1^{-1}\left(\overline{h_1(z)}\right)},\quad 
u_2:z\mapsto\overline{h_2^{-1}\left(\overline{h_2(z)}\right)}
\end{equation}
corresponding to $(g_1,h_1)$ and $(g_2,h_2)$ determine same classes in ${\rm{Mod}}(f)$, they must differ by a third quasi-conformal automorphism $v$ of $f$ isotopic to the identity: $u_2=u_1\circ v$. 
If $v\restriction_{h_1^{-1}(\hat{\Bbb{R}})}$  is identity, then 
$$\forall z\in h_1^{-1}(\hat{\Bbb{R}}):\overline{h_1^{-1}\left(\overline{h_1(z)}\right)}=\overline{h_2^{-1}\left(\overline{h_2(z)}\right)};$$
  which by setting $x$ to be $h_1(z)$ implies 
$$\forall x\in\hat{\Bbb{R}}: h_2\circ h_1^{-1}(x)=\overline{h_2\circ h_1^{-1}(x)};$$
meaning that the conjugacy $h_2\circ h_1^{-1}$ 
between $g_1,g_2$ preserves $\hat{\Bbb{R}}$ and restricts to a conjugacy between 
$g_1\restriction_{\hat{\Bbb{R}}},g_2\restriction_{\hat{\Bbb{R}}}$ which are thus of the same entropy.  
Pulling back via $h_1$ and $h_2$, this can also be stated as systems 
$f=h_1^{-1}\circ g_1\circ h_1\restriction_{h_1^{-1}(\hat{\Bbb{R}})}$
and 
$f=h_2^{-1}\circ g_2\circ h_2\restriction_{h_2^{-1}(\hat{\Bbb{R}})}$
being of the same entropy. 
The general case is more complicated; there is an isotopy $\{v_t\}_{t\in [0,1]}$ 
from $v_0=v$ to $v_1=\mathbf{1}$
through q.c.-automorphisms of $f$. Thus 
$\left\{v_t\left( h_1^{-1}(\hat{\Bbb{R}})\right)\right\}_{t\in [0,1]}$
 is a $1$-parameter family of  $f$-invariant quasi-circles  terminating at $t=1$ with 
 $h_1^{-1}(\hat{\Bbb{R}})$. 
Asking for
$v_t$'s
to restrict to identity on the quasi-circle $h_1^{-1}(\hat{\Bbb{R}})$ is too much as there might be some open intervals  of Fatou points  that can be wiggled within the corresponding Fatou component; think about the natural foliation of a rotation domain.  Nevertheless, it would be sufficient if  the isotopy fixes
the ``important'' part of the dynamics that determines the entropy, namely the intersection of the quasi-circle 
$h_1^{-1}(\hat{\Bbb{R}})$ with the Julia set of $f$.\footnote{The  non-wandering set of  $f:\hat{\Bbb{C}}\rightarrow\hat{\Bbb{C}}$ is the disjoint union  of the Julia set with all (finitely many) cycles of rotation domains and (finitely many) attracting periodic points; cf. 
\cite[Problem 19-a]{MR2193309}. But the entropy vanishes restricted to the latter two. Hence, for any closed $f$-invariant subset 
$A$ of the Riemann sphere, the topological entropy of $f\restriction_A$ coincides with that of the subsystem
$f\restriction_{\mathcal{J}(f)\cap A}$ as any point outside it is either wandering or in a zero entropy closed subsystem.}

\begin{proof}[Proof of Proposition \ref{technical}]
{By symmetry, it suffices to argue that the topological entropy of the system 
$g_1=h_1\circ f\circ h_1^{-1}\restriction_{\hat{\Bbb{R}}}$ cannot exceed that of the system 
$g_2=h_2\circ f\circ h_2^{-1}\restriction_{\hat{\Bbb{R}}}$. But these systems are conjugate with 
$f\restriction_{h_1^{-1}(\hat{\Bbb{R}})}$ and $f\restriction_{h_2^{-1}(\hat{\Bbb{R}})}$ respectively (keep in mind that the quasi-circles $h_1^{-1}(\hat{\Bbb{R}})$,  $h_2^{-1}(\hat{\Bbb{R}})$
are $f$-invariant as the rational maps $g_1=h_1\circ f\circ h_1^{-1}$, $g_2=h_2\circ f\circ h_2^{-1}$
are with real coefficients.). Therefore, we only need to show that 
$h_{\rm{top}}\left(f\restriction_{h_1^{-1}(\hat{\Bbb{R}})}\right)\leq h_{\rm{top}}\left(f\restriction_{h_2^{-1}(\hat{\Bbb{R}})}\right).$\\
\indent
Given a periodic point $z_0$ of $f$, $\{v_t(z_0)\}_{t\in[0,1]}$ is a curve of periodic points of $f$ since $v_t$'s commute with $f$. The set of periodic points of $f$ is countable; therefore, all of these points coincide with $v_1(z_0)=x$. Consequently, $v_0=v=u_1^{-1}\circ u_2$ fixes each periodic point of $f$ and thus restricts to identity on the Julia set of $f$ since $\mathcal{J}(f)$ is the closure of repelling periodic points. 
Recalling the definitions of $u_1,u_2$ in \eqref{auxiliary1}, this indicates
$h_1^{-1}\left(\overline{h_1(x)}\right)=h_2^{-1}\left(\overline{h_2(x)}\right)$
for any $x\in\mathcal{J}(f)\cap h_1^{-1}(\hat{\Bbb{R}})$.
The left-hand side is just $x$ and hence  $h_2(x)$ should be real; thus $\mathcal{J}(f)\cap h_1^{-1}(\hat{\Bbb{R}})$ is a closed subsystem of $f\restriction_{h_2^{-1}(\hat{\Bbb{R}})}$ as well and then:  
$$h_{\rm{top}}\left(f\restriction_{h_1^{-1}(\hat{\Bbb{R}})}\right)
=h_{\rm{top}}\left(f\restriction_{\mathcal{J}(f)\cap h_1^{-1}(\hat{\Bbb{R}})}\right)
\leq h_{\rm{top}}\left(f\restriction_{h_2^{-1}(\hat{\Bbb{R}})}\right).$$
}
\end{proof}

The same idea of Schwarz reflection appearing above can be employed to establish the $\mathcal{J}$-stable case as well. 
\begin{proof}[Proof of Proposition \ref{temp1'}]
One characterization of $\mathcal{J}$-stability is the existence of a holomorphic motion of Julia sets 
\cite[Theorem B]{MR732343}. Consequently, there is a continuous map 
\begin{equation}\label{auxiliary2}
i:[0,1]\times\mathcal{J}(f_0)\rightarrow\hat{\Bbb{C}}
\end{equation}
such that each $i_t:=i(t,.)$ is a conjugacy
\begin{equation}\label{conjugacy}
\left(\mathcal{J}(f_0),f_0\restriction_{\mathcal{J}(f_0)}\right)\rightarrow
\left(\mathcal{J}(f_t),f_t\restriction_{\mathcal{J}(f_t)}\right)
\end{equation}
 with $i_0$ being identity. As every $f_t$ is with real coefficients, the Julia sets $\mathcal{J}(f_t)$ are invariant under the complex conjugation; and therefore, for any $x_0\in\mathcal{J}(f_0)$ the following is a well defined continuous map:
\begin{equation}\label{auxiliary3}
t\in[0,1]\mapsto \overline{i_t^{-1}\left(\overline{i_t(x_0)}\right)}.
\end{equation}
But if $x_0$ is picked to be periodic, e.g. $f_0^{\circ n}(x_0)=x_0$, then, as all the maps commute with the complex conjugation and $f_t\circ i_t=i_t\circ f_0$:
\begin{equation*}
\begin{split}
&f_0^{\circ n}\left(\overline{i_t^{-1}\left(\overline{i_t(x_0)}\right)}\right)
=\overline{f_0^{\circ n}\circ i_t^{-1}\left(\overline{i_t(x_0)}\right)}
=\overline{ i_t^{-1}\circ f_t^{\circ n}\left(\overline{i_t(x_0)}\right)}
=\overline{ i_t^{-1}\left(\overline{f_t^{\circ n}\circ i_t(x_0)}\right)}\\
&=\overline{ i_t^{-1}\left(\overline{i_t\left(f_0^{\circ n}(x_0)\right)}\right)}
=\overline{i_t^{-1}\left(\overline{i_t(x_0)}\right)};
\end{split}
\end{equation*}
which implies that \eqref{auxiliary2} parametrizes a curve passing through the periodic points of $f_0$ once 
$x_0\in\mathcal{J}(f_0)$ is periodic.  In such a situation, \eqref{auxiliary2} must be constant of value 
$$\overline{i_0^{-1}\left(\overline{i_0(x_0)}\right)}=x_0;$$
or equivalently $\overline{i_t(x_0)}=i_t(\overline{x_0})$ for any $t\in[0,1]$. But $\mathcal{J}(f_0)$ contains all repelling periodic points of $f_0$ as a dense subset. So by continuity, $\overline{i_t(x)}=i_t(\bar{x})$ holds for all 
$x\in\mathcal{J}(f_0)$, $t\in[0,1]$. In other words, the motion \eqref{auxiliary3} of Julia sets respects the complex conjugation; in particular, it preserves the real points.  The conjugacy \eqref{conjugacy} thus restricts to a conjugacy 
$$
\left(\mathcal{J}(f_0)\cap\hat{\Bbb{R}},f_0\restriction_{\mathcal{J}(f_0)\cap\hat{\Bbb{R}}}\right)\rightarrow
\left(\mathcal{J}(f_t)\cap\hat{\Bbb{R}},f_t\restriction_{\mathcal{J}(f_t)\cap\hat{\Bbb{R}}}\right)
$$
between two dynamical systems of topological entropies $h_\Bbb{R}(f_0)$ and $h_\Bbb{R}(f_t)$. 
\end{proof}

This discussion culminates in the following theorem:
\begin{theorem}\label{main}
Given a real rational map $f\in\Bbb{R}(z)$ of degree $d\geq 2$, the real entropy function 
$h_\Bbb{R}:\mathcal{M}'_d-\mathcal{S}'\rightarrow \left[0,\log(d)\right]$
is constant on connected components of  $\mathcal{M}'_d\bigcap{\rm{M}}(f)-\mathcal{S}'$. 
\end{theorem}

\begin{proof}
{Proposition \ref{technical} says that the lift $\widetilde{h_{\Bbb{R}}}$ of $h_\Bbb{R}$ is constant on level sets of the continuous map 
$\left\{[(g,h)]\in{\rm{T}}(f)\,|\, g\in\Bbb{R}(z), {\rm{Aut}}(g)=\{\mathbf{1}\}\right\}\to {\rm{Mod}}(f)$
defined in \eqref{the map}. Since the target space is discrete, we conclude that 
$\widetilde{h_{\Bbb{R}}}: [(g,h)]\mapsto h_\Bbb{R}(g)$ is locally constant on the preceding domain. Now the commutative 
diagram \eqref{lift} implies that $h_\Bbb{R}$ is locally constant on the image 
$\mathcal{M}'_d\bigcap{\rm{M}}(f) -\mathcal{S}'$
of this space in the moduli space $\mathcal{M}_d(\Bbb{C})$ and this finishes the proof.}
\end{proof}
\begin{proof}[Proof of Theorem \ref{temp1}]
An immediate corollary of Proposition \ref{dimension} and Theorem \ref{main}.
\end{proof}

Of course, Theorem \ref{main} is not interesting unless the intersection  $\mathcal{M}'_d\bigcap{\rm{M}}(f)$ is of positive dimension or equivalently, there are plenty of classes  in the Teichmüller  space ${\rm{T}}(f)$ which are represented by real maps.

\begin{corollary}\label{main-corollary}
For a non-antipodal rational map $f\in\Bbb{R}(z)$ of degree $d\geq 2$ with ${\rm{Aut}}(f)=\{\mathbf{1}\}$,
the real entropy function 
$h_\Bbb{R}:\mathcal{M}'_d-\mathcal{S}'\rightarrow [0,\log(d)]$
is constant on a submanifold of real dimension $\dim_{\Bbb{C}}{\rm{M}}(f)$
passing through $\langle f\rangle$. 
\end{corollary}      

\begin{proof}
The involution from Definition \ref{involution-def} acts on the ${\rm{T}}(f)$
via 
\begin{equation}\label{involution1}
\iota:[(g,h)]\mapsto [(\tilde{g},\tilde{h})].
\end{equation}
Keep in mind that  by identities \eqref{differentiation}, if $h$ is a q.c.-homeomorphisms of dilatation $\mu$ then 
$\tilde{h}$ would be another such homeomorphism of dilatation $\tilde{\mu}$. Hence the transformation $\iota$ from \eqref{involution1} acts on both 
${\rm{T}}(f)$ and $\left\{[(g,h)]\in{\rm{T}}(f)\,|\, {\rm{Aut}}(g)=\{\mathbf{1}\} \right\}$ in a  well defined manner because $\tilde{f}=f$ and the involution   in \eqref{involution} respects the composition and preserves the group of  Möbius transformations. If ${\rm{Aut}}(g)=\{\mathbf{1}\} $ and $(g,h)\sim \left(\bar{g},\tilde{h}\right)$, then the maps $g$ and $\tilde{g}=\bar{g}$ are Möbius conjugate. Recalling Proposition \ref{antipodal}, this implies that $g$ is with real coefficients  if it is away from the antipodal locus. Nonetheless, the conditions of being on the antipodal locus or having non-trivial symmetries are closed. So if one can compute the dimension of ${\rm{Fix}}\left(\iota:{\rm{T}}(f)\rightarrow {\rm{T}}(f)\right)$ around a point 
$[(g,h)]=[(\tilde{g},\tilde{h})]$ fixed by $\iota$ where $g\in\Bbb{R}(z)$ is non-antipodal and ${\rm{Aut}}(g)=\{\mathbf{1}\}$, then something can be inferred about the dimension  $\mathcal{M}'_d\bigcap{\rm{M}}(f)$.	 \\
\indent
To do so, suppose $f\in\Bbb{R}(z)$ is neither antipodal nor admits non-trivial Möbius symmetries. 
Clearly $[(f,\mathbf{1})]$ is a fixed point of the $C^\infty$ involution $\iota$  of the complex manifold ${\rm{T}}(f)$. 
The tangent map ${\rm{d}}_{[(f,\mathbf{1})]}\iota$ of $\iota$ at this point is a real-linear involution of the complex vector space ${\rm{T}}_{[(f,\mathbf{1})]}\left({\rm{T}}(f)\right)$. Thus the tangent space decomposes to the direct sum of $+1$ and $-1$ eigenspaces and the former is the tangent space to ${\rm{Fix}}\left(\iota:{\rm{T}}(f)\rightarrow {\rm{T}}(f)\right)$ 
at $[(f,\mathbf{1})]$. We claim the dimension of these two eigenspaces coincide. Any representative $(g,h)$ of a point in the Teichmüller  space ${\rm{T}}(f)$ determines an $f$-invariant Beltrami differential 
$\mu(z)\frac{{\rm{d}}\bar{z}}{{\rm{d}}z}$
where $\mu$ belongs to the unit ball in
$ L^\infty(\hat{\Bbb{C}})$. So there is a $\Bbb{C}$-linear surjection from the space 
$$\left\{\mu\in L^\infty(\hat{\Bbb{C}})\,|\, \mu(z)\frac{{\rm{d}}\bar{z}}{{\rm{d}}z} \text{ is } f\text{-invariant.}\right\}$$
onto the tangent space ${\rm{T}}_{[(f,\mathbf{1})]}\left({\rm{T}}(f)\right)$. 
Given the way $\iota$ is defined in \eqref{involution1} and remembering that for a q.c.-homeomorphism $h$ of dilatation 
$\mu$, the map $\tilde{h}$ is of dilatation $\tilde{\mu}$, we deduce that the involution ${\rm{d}}_{[(f,\mathbf{1})]}\iota$ of 
${\rm{T}}_{[(f,\mathbf{1})]}\left({\rm{T}}(f)\right)$ comes from the involution $\mu\mapsto\tilde{\mu}$ acting on 
$ L^\infty(\hat{\Bbb{C}})$. The description of this involution in \eqref{involution} clearly indicates that this map is conjugate-linear and hence, so is ${\rm{d}}_{[(f,\mathbf{1})]}\iota$. This means scaler multiplication by ${\rm{i}}$ maps the eigenspace of $+1$ to that of $-1$ and vice versa. So, as real vector spaces, they are of dimension 
$$\dim_{\Bbb{C}}{\rm{T}}_{[(f,\mathbf{1})]}\left({\rm{T}}(f)\right)=\dim_{\Bbb{C}}{\rm{T}}(f)=\dim_{\Bbb{C}}{\rm{M}}(f).$$ 
\end{proof}
The complex dimension of the Teichmüller  space ${\rm{T}}(f)$, or equivalently that of the moduli space ${\rm{M}}(f)$, 
can be determined in terms of certain invariants of the holomorphic system $f:\hat{\Bbb{C}}\rightarrow\hat{\Bbb{C}}$
(\cite[Theorem 6.8]{MR1620850},  \cite[p. 128]{MR2241787}):
\begin{equation}\label{dimension formula}
\dim_{\Bbb{C}}{\rm{M}}(f)=n_{AC}+n_{HR}+n_{LF}-n_{PB};
\end{equation}
where
\begin{enumerate}[-]
\item $n_{AC}$ is the number of \textit{foliated equivalence classes of acyclic critical points in the Fatou set};
\item $n_{HR}$ is the number of cycles of Herman rings;
\item $n_{LF}$ is the number of ergodic line fields on the Julia set;
\item $n_{PB}$ is the number of cycles of parabolic basins.
\end{enumerate}
Corollary \ref{main-corollary} thus establishes a lower bound for the real dimension of the isentrope passing through a generic real rational map $f\in\Bbb{R}(z)$ in terms of the complex dynamics of $f$.

\section{Examples of stable families of constant real entropy}\label{families}
This final section is devoted to examples of stable families of rational maps with real coefficients where, in accord with Theorem \ref{temp1} and Proposition \ref{temp1'}, the real entropy remains constant.  We first treat the case of hyperbolic maps in \S\ref{hyperbolic subsection} by a different approach  utilizing the well known techniques of kneading theory \cite{MR970571}. In \S\ref{Lattes subsection} we calculate the real entropy for the prominent family of flexible Latt\`es maps. At last, \S\ref{maximal subsection} entails a detailed discussion on families where the real entropy is the maximum $\log(d)$.

\subsection{Hyperbolic components}\label{hyperbolic subsection}
Recall that a rational map is called hyperbolic if each of its critical orbits converges to an attracting periodic cycle 
\cite[Theorem 19.1]{MR2193309}. Such maps form an open subset of the moduli space 
$\mathcal{M}_d(\Bbb{C})$ whose connected components are called hyperbolic components.
Hyperbolic maps are known to be $\mathcal{J}$-stable and thus, Proposition \ref{temp1'} implies that each connected component of the intersection with $\mathcal{M}'_d$ of a hyperbolic component in $\mathcal{M}_d(\Bbb{C})$  is included in a single isentrope.  Here, we present a completely different proof of this fact that also sheds light on the entropy values realized by real hyperbolic maps.
%Away from the antipodal and symmetry loci, these ``real hyperbolic components'' are basically the set of fixed points the involution $\langle f\rangle\mapsto \langle\bar{f}\rangle$ when it acts on a hyperbolic component in $\mathcal{M}_d(\Bbb{C})$.

\begin{theorem}\label{entropy constant over hyperbolic}
Let $d\geq 3$ and $\mathcal{U}$ be a connected component of the intersection of a hyperbolic component of $\mathcal{M}_d(\Bbb{C})$ with the real subvariety $\mathcal{M}'_d$. Then the function $h_\Bbb{R}$ is constant over 
$\mathcal{U}$ with a value which is the logarithm of an algebraic number. 
\end{theorem}

\begin{remark}
There is a thorough classification of hyperbolic components of $\mathcal{M}_2(\Bbb{C})$ \cite{MR1246482, MR1047139}. In the case of $d=2$ Theorem \ref{entropy constant over hyperbolic} still remains valid but requires a little bit of more work as excluding the symmetry locus might cause the real hyperbolic component $\mathcal{U}$ to become disconnected; cf. \cite{2019arXiv190103458F}.
\end{remark}

\begin{proof}[Proof of Theorem \ref{entropy constant over hyperbolic}]
{The open subset $\mathcal{U}-\mathcal{S}'$ of $\mathcal{M}'_d$ is connected  because 
$\mathcal{U}\bigcap\mathcal{S}'$ is of codimension at least two; see Proposition \ref{codimension>1}.  Any real representative $f\in\Bbb{R}(z)$ of a point in it restricts to a continuous multimodal circle map
$f\restriction_{\hat{\Bbb{R}}}:\hat{\Bbb{R}}\rightarrow\hat{\Bbb{R}}$ 
that satisfies the hypothesis of Lemma \ref{algebraic} below. Therefore, $\log\left(h_\Bbb{R}\left(\langle f\rangle\right)\right)$
is always algebraic for a point $\langle f\rangle$ from this open connected set. The continuity of   
$h_{\Bbb{R}}:\mathcal{M}'_d-\mathcal{S}'\rightarrow\left[0,\log(d)\right]$ then
yields its constancy over $\mathcal{U}-\mathcal{S}'$.}
\end{proof}

\begin{lemma}\label{algebraic}
Let $I$ be an interval $[a,b]$ or a circle. Let $f:I\rightarrow I$ be a continuous multimodal self-map of $I$ whose turning points are attracted by  periodic orbits. Then the number ${\rm{e}}^{h_{\rm{top}}(f)}$ is algebraic.
\end{lemma}

\begin{proof}
Let us first deal with the interval case. We need to use the classical kneading theory of Milnor and Thurston developed in \cite{MR970571}. Suppose 
$f:[a,b]\rightarrow [a,b]$ is a
multimodal self-map of $I:=[a,b]$ with  turning points $c_1<\dots<c_{l-1}$ and laps 
$I_1=[a,c_1], I_2=[c_1,c_2],\dots,I_{l-1}=[c_{l-2},c_{l-1}], I_l=[c_{l-1},b]$. 
The \textit{shape} of the restriction $f\restriction_{I_n}$ of $f$ to the $n^{\rm{th}}$ lap $I_n$ will be denoted by $\epsilon_n\in\{\pm1\}$ 
which is $+1$ if the restriction is increasing and $-1$ when it is decreasing.
To each point $x\in I$,
one can assign infinite vectors $\left(\theta_i(x^+)\right)_{i\geq 0}$ and 
$\left(\theta_i(x^-)\right)_{i\geq 0}$ whose components come from the set of symbols
$\left\{\pm I_1,\dots,\pm I_l\right\}$. Here is the definition: $\theta_i(x^+)=\epsilon I_n$ 
(respectively $\theta_i(x^-)=\epsilon I_n$) means that there is a half-open interval with its left end (resp. right end)
 at $x$ on which $f^{\circ i}$ is monotonic of shape $\epsilon\in\{\pm 1\}$ and is mapped by $f^{\circ i}$ into $I_n$. 
Associated with them are formal power series $\theta(x^+):=\sum_{i=0}^\infty\theta_i(x^+)t^i$
and $\theta(x^-):=\sum_{i=0}^\infty\theta_i(x^-)t^i$ from $V[[t]]$ with $V$ being the free abelian group
$\Bbb{Z}.\left\{I_1,\dots,I_l\right\}$.
The next definition is that of the \textit{kneading increments} of $f$ given by 
$\nu_m:=\theta(c_m^+)-\theta(c_m^{-})$ for any $1\leq m\leq l-1$.  One can then form the $(l-1)\times l$ 
\textit{kneading matrix} $N$ of power series in 
$\Bbb{Z}[[t]]$ with $N_{mn}$ being the coefficient of $I_n$ in $\nu_m$, i.e. 
$\nu_m=\sum_{n=1}^lN_{mn}I_n$. Denoting the orientation of $f\restriction_{I_n}$ by $\epsilon_n\in\{\pm 1\}$ and the determinant of the submatrix of $N$ obtained from deleting the $n^{\rm{th}}$ column by $D_n(t)$, it can be proved that the formal power series
$(-1)^{n+1}D_n(t)/(1-\epsilon_nt)$ is independent of $1\leq n\leq l$. This common power series is called 
\textit{the kneading invariant} of $f$ and will be denoted by $D(t)$. It is easy to observe that the coefficients of $D(t)$ are bounded integers, so 
$D(t)$ defines an analytic function on the disk $|t|<1$. 
Here is the main result (\cite[Theorem 6.3]{MR970571}): 
\begin{itemize}
\item[]\textit{If $h_{\rm{top}}(f)=0$, $D(t)$ does not have any root in the open unit disk and otherwise,
${\rm{e}}^{-h_{\rm{top}}(f)}$
 is the smallest root of $D(t)$ in $[0,1)$.}
\end{itemize}
Thus it suffices to show that under our assumption the kneading invariant $D(t)$ is a rational map with integer coefficients whose roots,
including ${\rm{e}}^{-h_{\rm{top}}(f)}$, are algebraic numbers. In order to do so, one just needs to show that kneading coordinates 
$\left(\theta_i(c_m^+)\right)_{i\geq 0}$ and 
$\left(\theta_i(c_m^-)\right)_{i\geq 0}$
are eventually periodic. We will argue that 
$\left(\theta_i(c_m^+)\right)_{i\geq 0}$ is periodic; the case of $\left(\theta_i(c_m^-)\right)_{i\geq 0}$
is completely similar. Suppose for $i\geq q$, $f^{\circ i}(c_m)$ is in the immediate basin of the periodic point $x_{i\,{\rm{mod}}\,p}$ from the orbit 
$x_0\mapsto x_1\mapsto\cdots\mapsto x_{p-1}\mapsto x_0$.  We claim that for $\delta>0$ small enough, one can take $q'>q$ so large that for $i>q'$ the interval 
$f^{\circ i}\left((c_m,c_m+\delta)\right)$ is in the interior of a lap $I_{n_{i\,{\rm{mod}}\,p}}$ of $f$ which is dependent only on the remainder of $i$ modulo $p$. To see this, fix $0\leq r<p$. Note that $f^{\circ (jp+r)}(c_m)\to x_r$ as $j\to\infty$; and so if $x_r$ is not a turning point, $f^{\circ (jp+r)}(c_m)$ and hence the image  under $f^{\circ (jp+r)}$ of small enough 
non-degenerate  subintervals $[c_m,c_m+\delta]$ of the basin belong to the interior of the lap that has $x_r$ for large enough $j$'s; say for $j\geq j_r$. The same holds even when $x_r$ is a turning point of $f$; we only have to rule out the possibility of $f^{\circ (jp+r)}(c_m)$ alternating between the two laps that have $x_r$ in common:
if the fixed point 
 $$x_r=\lim_{j\to\infty}f^{\circ (jp+r)}\left((c_m,c_m+\delta)\right)$$ 
   of $f^{\circ p}$ is a turning point  as well,  then for $j$ large enough, 
$f^{\circ (jp+r)}\left((c_m,c_m+\delta)\right)$ 
 always lands to the left of $x_r$ if 
$x_r$ is a local maximum of $f^{\circ p}$ and to the right if it is a local minimum. 
Repeating this argument for all $r\in\{0,\dots,p-1\}$,  $p.\left(\max\{j_0,\dots,j_{p-1}\}+1\right)$ then works as the desired $q'$.\\
\indent 
Next, after decreasing $\delta$ if necessary, suppose all iterates $f,\dots, f^{\circ (q'+1)}$ restrict to monotonic maps on $[c_m,c_m+\delta]$. We now show that 
$\left(\theta_i(c_m^+)\right)_{i\geq 0}$ is periodic of period $2p$ for $i>q'$:
$f^{\circ i}\left((c_m,c_m+\delta)\right)$ is contained in the lap $I_{n_{i\,{\rm{mod}}\,p}}$, and the 
degree of $f^{\circ i}\restriction_{[c_m,c_m+\delta]}$ is the product of the degree of the monotonic map 
$f^{\circ (q'+1)}\restriction_{[c_m,c_m+\delta]}$ by the degrees of $f$ on laps $I_{n_{(q'+1)\,{\rm{mod}}\,p}},\dots,I_{n_{(i-1)\,{\rm{mod}}\,p}}$. Replacing $i$ with $i+2p$ does not affect the former  while multiplies the latter product by $\left(\prod_{i=1}^p\epsilon_{n_{i\,{\rm{mod}}\,p}}\right)^2=1$.\\
\indent
Finally, notice that kneading theory has also been developed for continuous multimodal circle maps or equivalently, for multimodal interval maps with finitely many discontinuities; see \cite[appendix]{MR1029100}. 
Therefore, Lemma \ref{algebraic} is valid for multimodal circle map too. 
\end{proof}

\begin{remark}
Much more can be said about arithmetic properties of entropy values of \textit{post-critically finite} multimodal maps. The paper \cite{MR3289916} establishes that a real algebraic integer arises as 
${\rm{exp}}\left(h_{\rm{top}}\left(f:I\rightarrow I\right)\right)$ 
for a critically finite multimodal map $f:I\rightarrow I$ if and only if  it is a \textit{weak Perron number}, i.e. at least as large as the absolute values of its Galois conjugates. 
\end{remark}

\begin{remark}
The intersection with $\mathcal{M}'_d$ of a hyperbolic component in $\mathcal{M}_d(\Bbb{C})$ could be disconnected and thus $h_\Bbb{R}$ may assume different values at real maps from the same complex hyperbolic component. As an example, the real entropy of $z^2+c$ is $0$ for $c>\frac{1}{4}$ (i.e. for real parameters right of the Mandelbrot set) 
and is $\log(2)$ for $c<-2$ (i.e. for real parameters left of the Mandelbrot set). But in both cases the quadratic polynomial lies in the \textit{escape} component of $\mathcal{M}_2(\Bbb{C})$ consisting of quadratic rational maps whose critical points converge to an attracting fixed point. 
More generally, the hyperbolic component of $\mathcal{M}_d(\Bbb{C})$  containing the \textit{polynomial shift locus} (the affine conjugacy classes of degree $d$ maps whose critical points are in the basin of infinity which is known to be connected \cite[Corollary 6.2]{MR2827015}) has disconnected intersection with $\mathcal{M}'_d$. 
\end{remark}

\subsection{The Latt\`es family}\label{Lattes subsection}
A family of \textit{flexible Latt\`es maps} of the same degree is an example of a quasi-conformally trivial  family of non-hyperbolic maps.  (A detailed treatment of \textit{Latt\`es maps} can be found in \cite{MR2348953}.)  Conjecturally, these are the only families of rational maps admitting \textit{invariant line fields}. Any such  family forms a single dynamical moduli space of complex dimension one. The example below verifies Theorem \ref{main} through calculating the real entropy of those  flexible Latt\`es maps that preserve the real circle. It is well known that the Julia set of a Latt\`es map is the whole Riemann sphere and hence, unlike \S\ref{hyperbolic subsection}, the subsystem obtained from restricting to the real circle can never have an attractor. 

\begin{example}\label{Lattes}
Set $d=m^2\geq 4$ for an integer $m\neq 0,\pm 1$ and consider a flexible Latt\`es map $f$ of degree $d$ obtained from the multiplication by $m$ map $[m]:[z]\mapsto [mz]$ on the elliptic curve $E=\Bbb{C}\big/\Bbb{Z}+\Bbb{Z}\tau$ with $\tau$ being in the upper half plane. So $f$ makes the following diagram commutative where the columns are the two-fold ramified covering 
$\pi:E\rightarrow\Bbb{P}^1(\Bbb{C})$ obtained by taking the quotient of $E$ by the action of the involution 
$[z]\mapsto [-z]$; the morphism which is induced by the Weierstrass function $\wp:\Bbb{C}\rightarrow\hat{\Bbb{C}}$ of the lattice $\Lambda:=\Bbb{Z}+\Bbb{Z}\tau$. 
\begin{equation}\label{diagram1}
\xymatrix{E=\Bbb{C}\big/\Bbb{Z}+\Bbb{Z}\tau \ar[d]^\pi\ar[r]^{[m]} & 
E=\Bbb{C}\big/\Bbb{Z}+\Bbb{Z}\tau \ar[d]^\pi\\
E\big/[z]\sim [-z]=\Bbb{P}^1(\Bbb{C})\ar[r]^f& E\big/[z]\sim [-z]=\Bbb{P}^1(\Bbb{C})}
\end{equation}
The map $f$  preserves $\hat{\Bbb{R}}$ if the Weierstrass function
$$\wp(z)=\frac{1}{z^2}+\sum_{0\neq w\in\Bbb{Z}+\Bbb{Z}\tau}\left(\frac{1}{(z-w)^2}-\frac{1}{w^2}\right)$$
commutes with the complex conjugation; for instance, when the lattice $\Bbb{Z}+\Bbb{Z}\tau$ is invariant under the complex conjugation. This happens if and only if ${\rm{Re}}(\tau)\in\frac{1}{2}\Bbb{Z}$. Up to the action of ${\rm{SL}}_2(\Bbb{Z})$,
${\rm{Re}}(\tau)$ can then be assumed to be $0$ or $\frac{1}{2}$. Let us work with the former and in the case of 
${\rm{Re}}(\tau)=\frac{1}{2}$ one merely needs to replace $\tau$ with $\tau-\frac{1}{2}$ in the subsequent discussion.\\
\indent Assuming that the period $\tau$ is purely imaginary, let us investigate the induced map $f$ and  the corresponding dynamics on the real circle. 
We shall do so by pulling back to the dynamics on $\wp^{-1}(\hat{\Bbb{R}})$ or on the invariant subset $\pi^{-1}(\hat{\Bbb{R}})$ of the elliptic curve $E$. 
Since $\wp(\bar{z})=\overline{\wp(z)}$, $\wp(z)$ is real if and only if either $2\,{\rm{Re}}(z)$ or $2\,{\rm{Im}}(z){\rm{i}}$ belongs to the rectangular lattice 
$\Lambda=\Bbb{Z}+\Bbb{Z}\tau$. We conclude that $\wp^{-1}(\hat{\Bbb{R}})$ is the countable set of lines parallel to axes in the complex plane whose $x$ and $y$ intercepts come from $\frac{1}{2}\Bbb{Z}$ and 
$\frac{{\rm{Im}}(\tau)}{2}\Bbb{Z}$. This collection determines a tessellation of the complex plane with the smaller rectangle of vertices $0,\frac{1}{2},\frac{\tau+1}{2}$ and $\frac{\tau}{2}$ (in the counterclockwise order) which is bijectively mapped onto $\hat{\Bbb{R}}$ via $\wp$.\\
\indent  
The  union $\wp^{-1}(\hat{\Bbb{R}})$ of lines in $\Bbb{C}$ projects onto two pairs of parallel circle on the torus  
$E=\Bbb{C}\big/\Bbb{Z}+\Bbb{Z}\tau$;
\begin{equation}\label{circles}
\left(\left\{[x]\,|\, x\in\Bbb{R}\right\}, \left\{\left[x+\frac{\tau}{2}\right]\,\Big|\, x\in\Bbb{R}\right\}\right),
\quad \left(\left\{[x\tau]\,|\, x\in\Bbb{R}\right\}, \left\{\left[x\tau+\frac{1}{2}\right]\,\Big|\, x\in\Bbb{R}\right\}\right);
\end{equation}
that intersect each other in four points (the $2$-torsion points) and constitute $\pi^{-1}(\hat{\Bbb{R}})$. 
So the topological entropy of 
$f\restriction_{\hat{\Bbb{R}}}$ coincides with that of the restriction of $[m]$ to this union of circles because $\pi$ establishes a finite degree semi-conjugacy between these two systems. According to the parity of $m$, on each circle from the preimage $\pi^{-1}(\hat{\Bbb{R}})$ the map 
$[m]$  restricts to the multiplication by $m$  onto another circle from this union. We conclude that $h_\Bbb{R}(f)$ is the topological entropy of 
\begin{equation}\label{Chebyshev}
S^1=\Bbb{R}/\Bbb{Z}\rightarrow S^1=\Bbb{R}/\Bbb{Z}: [x]\mapsto[mx],
\end{equation}
i.e. 
$\log\left(|m|\right)$. 
Notice that under $\pi$ each circle form \eqref{circles} is mapped onto a compact subinterval of $\hat{\Bbb{R}}$; and if the circle is preserved by the multiplication map $[m]$, the dynamics on the interval is given by the quotient of \eqref{Chebyshev} by the involution $[x]\mapsto [-x]$, i.e. a Chebyshev polynomial.  The dynamics on $\hat{\Bbb{R}}$ is illustrated in Figure \ref{fig:1}.

%It should be emphasized that for $|m|\geq 2$ the map $f\restriction_{\hat{\Bbb{R}}}$ is not a covering map since $f$ has local extrema; e.g., $\wp\left(\frac{1}{2m}\right)\in\Bbb{R}$ which by analyzing the  multiplicities in the diagram \eqref{diagram1} is a point of multiplicity two for $f:\hat{\Bbb{C}}\rightarrow\hat{\Bbb{C}}$ and thus a turning point of $f\restriction_{\hat{\Bbb{R}}}$.
\begin{figure}[ht!]
\center
\includegraphics[width=13cm, height=5.6cm]{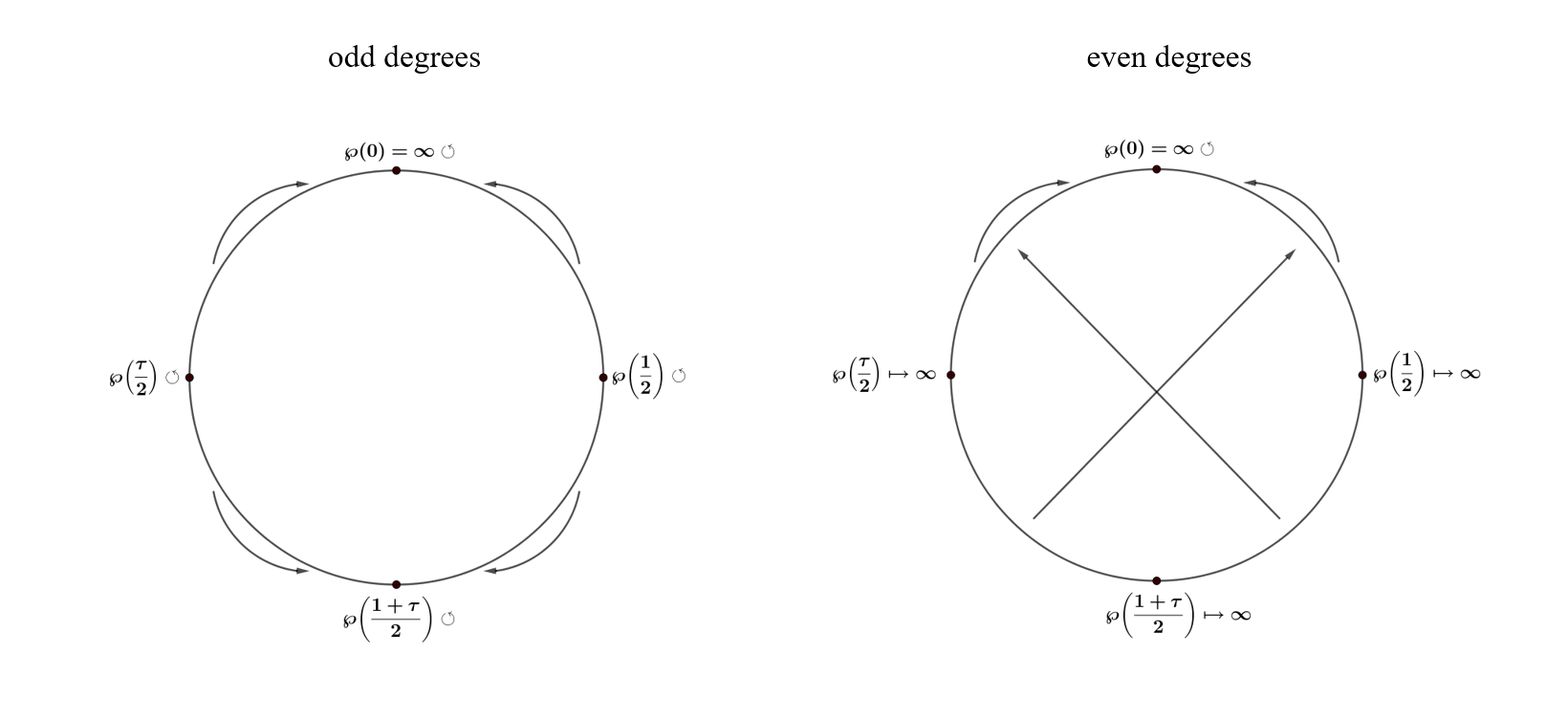}
\caption{The real dynamics of the map $f$ from \eqref{diagram1} in odd and even degrees. The images under $\pi$ of circles \eqref{circles} on the elliptic curve cut $\hat{\Bbb{R}}$ into four intervals determined by the values of the Weierstrass function at $2$-torsion points. Under $f$, each interval is either preserved or is mapped onto another interval. The dynamics on invariant intervals is that of a Chebyshev polynomial.}
\label{fig:1}
\end{figure}

\end{example}

\begin{example}
Let us calculate the real entropy of a couple of \textit{rigid} (not flexible) Latt\`es maps. It is known that any Latt\`es map semi-conjugate to an endomorphism of an elliptic curve $E$ via a morphism $E\rightarrow\Bbb{P}^1(\Bbb{C})$ whose degree (unlike the case of the diagram \eqref{diagram1}) is greater than two must  originate from elliptic curves with extra automorphisms, i.e. the elliptic curves corresponding to \textit{square} and \textit{hexagonal} lattices \cite{MR2348953}. 
\\
\indent 
It may be easily verified that after identifying 
$\Bbb{C}\big/\Bbb{Z}+\Bbb{Z}{\rm{i}}$ with the elliptic curve $y^2=x^3-x$, the multiplication map
$[1+{\rm{i}}]:[z]\mapsto\left[(1+{\rm{i}})z\right]$ factors through the quotient by the order four automorphism 
$[z]\mapsto[{\rm{i}}z]$ and induces the rational map $f(z):=-\frac{1}{4}\left(z+\frac{1}{z}-2\right)$ which fits into the  commutative diagram below.
\begin{equation*}
\xymatrixcolsep{5pc}\xymatrix{\left\{y^2=x^3-x\right\}=\Bbb{C}\big/\Bbb{Z}+\Bbb{Z}{\rm{i}} \ar[d]^{(x,y)\mapsto x^2}\ar[r]^{[1+{\rm{i}}]} & \left\{y^2=x^3-x\right\}=\Bbb{C}\big/\Bbb{Z}+\Bbb{Z}{\rm{i}} \ar[d]^{(x,y)\mapsto x^2}\\
\Bbb{P}^1(\Bbb{C})\ar[r]^{f}& \Bbb{P}^1(\Bbb{C})}
\end{equation*} 
The second iterate of $f$ would be a finite quotient of the endomorphism $[1+{\rm{i}}]^{\circ 2}=[2{\rm{i}}]$. But in 
the Legendre form $y^2=x^3-x$, the automorphism defined via the multiplication by ${\rm{i}}$ is just  
$(x,y)\mapsto (-x,\pm{\rm{i}}y)$; so $f^{\circ 2}$ is in fact induced via the multiplication by two endomorphism and hence, based on the discussion in Example \ref{Lattes}, we deduce that:
$$h_\Bbb{R}(f)=\frac{1}{2}\,h_\Bbb{R}(f^{\circ 2})=\frac{1}{2}\,\log(2)=\log(\sqrt{2}).$$ 
We finish this example with the entropy calculation for a real map which is the quotient of an endomorphism of the hexagonal elliptic curve $\Bbb{C}\big/\Bbb{Z}+\Bbb{Z}\omega$ where 
$\omega:={\rm{e}}^{\frac{2\pi{\rm{i}}}{3}}$ is a primitive third root of unity. This can be written as $y^2=x^3+1$ with the multiplication map 
$[\omega]:[z]\mapsto[\omega z]$ identified with the automorphism $(x,y)\mapsto(\omega x,y)$ of this algebraic curve. Carrying out the subtraction $(x,y)-(\omega x,y)$ results in a formula for the degree three endomorphism 
$[1-\omega]:[z]\mapsto\left[(1-\omega)z\right]$ and the commutative diagram below.
\begin{equation*}
\xymatrixcolsep{5pc}\xymatrix{\left\{y^2=x^3+1\right\}=\Bbb{C}\big/\Bbb{Z}+\Bbb{Z}\omega
\ar@/_2pc/[ddd]_{(x,y)\mapsto x}\ar[d]^{(x,y)\mapsto y}\ar[r]^{[1-\omega]} & 
\left\{y^2=x^3+1\right\}=\Bbb{C}\big/\Bbb{Z}+\Bbb{Z}\omega \ar@/^2pc/[ddd]^{(x,y)\mapsto x}\ar[d]_{(x,y)\mapsto y}\\
\Bbb{P}^1(\Bbb{C})\ar[r]^{y\mapsto\frac{\sqrt{3}{\rm{i}}}{9}\frac{y^3-9y}{y^2-1}}\ar[d]^{y\mapsto y^2}& \Bbb{P}^1(\Bbb{C})\ar[d]\ar[d]_{y\mapsto y^2}\\
\Bbb{P}^1(\Bbb{C})\ar[r]^{y\mapsto-\frac{1}{27}\frac{y(y-9)^2}{(y-1)^2}}& \Bbb{P}^1(\Bbb{C})\\
\Bbb{P}^1(\Bbb{C})\ar[u]_{x\mapsto x^3+1}\ar[r]^{x\mapsto \frac{1+\sqrt{3}{\rm{i}}}{2}\frac{x^3+4}{3x^2}}& \Bbb{P}^1(\Bbb{C})\ar[u]^{x\mapsto x^3+1}}
\end{equation*} 
Only the map appeared in the third row, i.e.  $f(z):=-\frac{1}{27}\frac{z(z-9)^2}{(z-1)^2}$, is real. But, iterating all maps once, we have $[1-\omega]^{\circ 2}=\left[(1-\omega)^2=-3\omega\right]$. Notice that multiplication by $\omega$ amounts to the automorphism $(x,y)\mapsto(\omega x,y)$ of $y^2=x^3+1$. Consequently, the map $x\mapsto x^3+1$ from the diagram above establishes a semi-conjugacy of finite degree between the flexible Latt\`es map induced by multiplication by $[-3]$ and $f^{\circ 2}$. The real entropy of the former has been calculated in Example \ref{Lattes} as $\log(3)$. We conclude that: 
$$h_\Bbb{R}(f)=\frac{1}{2}\,h_\Bbb{R}(f^{\circ 2})=\frac{1}{2}\,\log(3)=\log(\sqrt{3}).$$

\end{example}

\subsection{Maps of maximal real entropy}\label{maximal subsection}

The isentrope $h_\Bbb{R}=\log(d)$ is an interesting one to investigate. In this subsection we carry out a thorough analysis of this isentrope that culminates in the proof of Theorem \ref{temp3}.

\begin{theorem}\label{extremal classification}
Let $f\in\Bbb{R}(z)$ be of degree $d\geq 2$ with $\mu_f$ its measure of maximal entropy. 
Then the following are equivalent:
\begin{enumerate}[(a)]
\item $\mu_f(\hat{\Bbb{R}})>0$;
\item the Julia set $\mathcal{J}(f)$ is a subset of $\hat{\Bbb{R}}$;
\item $h_{\rm{top}}\left(f\restriction_{\hat{\Bbb{R}}}:\hat{\Bbb{R}}\rightarrow\hat{\Bbb{R}}\right)=\log(d)$.
\end{enumerate}
\end{theorem}

\begin{proof}
It is clear that (b)$\Rightarrow$(a),(c); because if $\mathcal{J}(f)\subseteq\hat{\Bbb{R}}$, using 
${\rm{supp}}(\mu_f)=\mathcal{J}(f)$:
$$h_{\rm{top}}\left(f\restriction_{\hat{\Bbb{R}}}\right)
\geq h_{\rm{top}}\left(f\restriction_{\mathcal{J}(f)}\right)
\geq h_{\mu_f}\left(f\restriction_{\mathcal{J}(f)}\right) 
=h_{\mu_f}\left(f:\hat{\Bbb{C}}\rightarrow\hat{\Bbb{C}}\right)=h_{\rm{top}}\left(f:\hat{\Bbb{C}}\rightarrow\hat{\Bbb{C}}\right)=\log(d).$$
Next, suppose  $\mu_f(\hat{\Bbb{R}})>0$. According to Montel's theorem, for any open neighborhood
$U\subseteq\hat{\Bbb{C}}$ of a point of  $\mathcal{J}(f)$, $\bigcup_nf^{-\circ n}(U)$ includes a dense open subset of  
$\mathcal{J}(f)$. 
We use the properties of the measure of maximal entropy to show that this union is  of full measure with respect to $\mu_f$. It is known that $\mu_f$ is ergodic and satisfies $f^{*}\mu_f=d\,.\,\mu_f$ \cite{MR736568}. The union $\bigcup_nf^{-\circ n}(U)$ is backward-invariant, so must be of zero measure if it is not of full measure. But this measure zero set is an open subset of the Julia set; thus there is an iterate of $f$ that maps it onto the set $\mathcal{J}(f)$  with full measure 
\cite[Corollary 14.2]{MR2193309}.  This contradicts $f^{*}\mu_f=d\,.\,\mu_f$; we deduce that 
$\mu_f\left(\bigcup_nf^{-\circ n}(U)\right)=1$.
Now fixing  a countable collection of open sets 
$\{U_m\}_m$ for which $\left\{U_m\cap\mathcal{J}(f)\right\}_m$ is an open basis for the topology of $\mathcal{J}(f)$
and then taking the intersection over $U_m$'s,
we arrive at the full measure subset $\bigcap_m\bigcup_nf^{-\circ n}(U_m)$ which by the Baire category theorem contains a dense subset of 
$\mathcal{J}(f)$. Furthermore, the forward iterates of any of its points form a dense subset of  $\mathcal{J}(f)$. 
As  $\mu_f(\hat{\Bbb{R}})>0$, this full measure subset  intersects $\hat{\Bbb{R}}$:
There is a point on $\hat{\Bbb{R}}$ whose orbit is dense in $\mathcal{J}(f)$. But 
$\hat{\Bbb{R}}$ is closed and forward-invariant so must contain the whole $\mathcal{J}(f)$. \\
\indent At last, we show 
(c)$\Rightarrow$(b) that will yields (c)$\Rightarrow$(a) because we have already established (a)$\Leftrightarrow$(b). 
To this end, we invoke the main result of the article \cite{MR599481} indicating that a multimodal transformation of an interval with positive entropy admits a measure of maximal entropy. 
In that reference, multimodal maps can be discontinuous at turning points. Thus, identifying  $\hat{\Bbb{R}}\cong S^1$
with $[0,1)$ via the bijection $x\mapsto {\rm{e}}^{2\pi{\rm{i}}x}$, the same holds for continuous circle maps of positive entropy.  
So there is a Borel probability measure, say $\nu$, on the circle $\hat{\Bbb{R}}$ with 
respect to which the metric entropy of 
$f\restriction_{\hat{\Bbb{R}}}:\hat{\Bbb{R}}\rightarrow\hat{\Bbb{R}}$
 coincides with its topological entropy $\log(d)$. Pushing forward via the inclusion $i:\hat{\Bbb{R}}\hookrightarrow\hat{\Bbb{C}}$, 
one gets a measure $i_*\nu$ on the Riemann sphere with respect to which the measure theoretic entropy of 
$f:\hat{\Bbb{C}}\rightarrow\hat{\Bbb{C}}$
 is at least $\log(d)$. So $i_*\nu$ is a measure of maximal entropy for this rational map too. The uniqueness of the measure of maximal entropy for rational maps established in \cite{MR736567} then implies that $i_*\nu=\mu_f$. Hence
$\mu_f(\hat{\Bbb{R}})=\nu(\hat{\Bbb{R}})=1$.
\end{proof}

Now that we know $\mathcal{J}(f)$ must be completely real once $h_{\Bbb{R}}(f)=\log(d)$, we will elaborate on the dynamics on the Julia set below; compare with \cite[Theorem 2]{MR2833563}.
\begin{theorem}\label{extremal classification 1}
Let $f\in\Bbb{R}(z)$ be a rational map of degree $d\geq 2$ with  $h_\Bbb{R}(f)=\log(d)$. Then precisely one of the followings holds:
\begin{enumerate}[(a)]
\item There is a real fixed point $p$ of multiplier $\lambda\in[-1,1]$ whose immediate basin of attraction $\mathcal{A}$ is the only Fatou component; and a compact interval $I\subset\hat{\Bbb{R}}$ with 
$f^{-1}(I)\subsetneq I$ whose complement $\hat{\Bbb{R}}-I$ is the largest open interval included in $\mathcal{A}$. 
In this case, the Julia set is a Cantor set included in $I$.
\item The same as above but with $f^{-1}(I)=I$. In this case, $\mathcal{J}(f)$ coincides with the interval $I$ restricted to which $f\restriction_I:I\rightarrow I$ is a boundary-anchored $(d-1)$-modal map with surjective laps. The only Fatou component is the basin 
$\mathcal{A}=\hat{\Bbb{C}}-I$. 
\item The Fatou components are the open half planes and are immediate basins either for a pair of conjugate non-real attracting periodic points, or for a parabolic real fixed point  which is either of multiplier $-1$ or of multiplier $+1$ and multiplicity $3$.
In this case, the Julia set is $\hat{\Bbb{R}}$ and $f\restriction_{\hat{\Bbb{R}}}:\hat{\Bbb{R}}\rightarrow\hat{\Bbb{R}}$ is a $d$-sheeted unramified covering.  
\end{enumerate}
\end{theorem}

\begin{proof}
By Theorem \ref{extremal classification} the Julia set $\mathcal{J}(f)$ is contained in $\hat{\Bbb{R}}$. 
Recall that the Julia set is always either connected or has uncountably many connected components \cite[Corollary 4.15]{MR2193309}. So $\mathcal{J}(f)\subseteq\hat{\Bbb{R}}$  is either a non-degenerate subinterval of $\hat{\Bbb{R}}$, the whole circle or a disjoint union of (forcibly at most countably many) compact non-degenerate subintervals of the real circle with a (necessarily uncountable) closed totally disconnected subset of it. We claim that in the latter case there is not any subinterval and thus $\mathcal{J}(f)$ is a closed totally disconnected subset of $\hat{\Bbb{R}}$ and hence (given the fact that it has no isolated point) a Cantor set. Assume the contrary; $\mathcal{J}(f)$ admits both interval and singleton components. But $f$ takes a connected component of $\mathcal{J}(f)$ to another such component and, being a  non-constant analytic map, cannot collapse a non-degenerate interval to a point. So the forward iterates of an interval component never cover a singleton component. This is a contradiction since the former contains an open subset of $\hat{\Bbb{R}}$ and thus an open subset of  $\mathcal{J}(f)$,
and by Montel's theorem the union of forward images of any non-empty open subset of the Julia set is the whole Julia set.
Consequently, $\mathcal{J}(f)\subseteq\hat{\Bbb{R}}$ is either $\hat{\Bbb{R}}$, a subinterval of it or a Cantor set on it.\\
\indent
Since $\mathcal{J}(f)\subseteq\hat{\Bbb{R}}$, there are at most two Fatou components with equality if and only if 
$\mathcal{J}(f)$ coincides with $\hat{\Bbb{R}}$. Each Fatou component is periodic and cannot be a rotation domain due to the fact that there are only finitely many Fatou component. Hence, invoking the classification of Fatou components, each Fatou components of $f$ is the immediate basin of an attracting or parabolic periodic point $p$. If $p$ lies on $\hat{\Bbb{R}}$, its multiplier $\lambda$ belongs to $[-1,1]$. If $p$ is not real, its complex conjugate $\bar{p}$ determines a different Fatou component of the real map $f$ and so the Fatou components appear in a conjugate pair; hence $\mathcal{J}(f)=\hat{\Bbb{R}}$ and we are in the situation that (c) describes. Notice that $p,\bar{p}$ cannot be parabolic as they do not belong to $\mathcal{J}(f)=\hat{\Bbb{R}}$. For $p$ real, the only other case where $\mathcal{J}(f)$ coincides with $\hat{\Bbb{R}}$ is if $p$ is a parabolic point with the half planes as its immediate basins. The period of $p$ must be one: the Fatou components appear in conjugate pairs and as there are merely two of them, there cannot be a parabolic cycle of period larger than one. Having two (or equivalently more than one) parabolic basins of attraction is immediate when $\lambda=-1$ 
(\cite[Lemma 10.4]{MR2193309}), while requires the multiplicity of the fixed point to be $3$ when $\lambda=1$.\\
\indent
In parts (a) and (b) the Julia set is a proper closed subset of $\hat{\Bbb{R}}$ and hence there is a single Fatou component which is the basin of attraction $\mathcal{A}$ for a fixed point $p\in\hat{\Bbb{R}}$ of multiplier $\lambda\in[-1,1]$. The connected component of $\mathcal{A}\cap\hat{\Bbb{R}}$ around $p$ is a forward-invariant open interval which is the immediate basin of the fixed point $p$ of the real system $f\restriction_{\hat{\Bbb{R}}}:\hat{\Bbb{R}}\rightarrow\hat{\Bbb{R}}$. Thus its complement, denoted by $I$, is a compact subinterval of $\hat{\Bbb{R}}$ containing $\mathcal{J}(f)$ and satisfying 
$f^{-1}(I)\cap\hat{\Bbb{R}}\subseteq I$. We have the equality in (b), where $f^{-1}(I)=I$, while the inclusion is strict in (a). In the latter case where  $f^{-1}(I)\cap\hat{\Bbb{R}}\subsetneq I$, the proper subset $\mathcal{J}(f)$ of $I$ must be Cantor as it cannot be any smaller subinterval of $I$ because then, the connected component around $p$ of the intersection of the Fatou set with $\hat{\Bbb{R}}$  would be larger than $\hat{\Bbb{R}}-I$ contradicting the definition of $I$. We claim that the preimage of $I$ is totally real again and hence $f^{-1}(I)\subsetneq I$. Assume otherwise. The backward-invariant Cantor subset 
$\mathcal{J}(f)$ of $I$ is totally real and hence there must be a subinterval of the complement $I-\mathcal{J}(f)\subset\hat{\Bbb{R}}$ whose preimage is not completely real. This can be written as $\left(y_0,y_0+\epsilon_0\right)$
or $\left(y_0-\epsilon_0,y_0\right)$ where $\epsilon_0>0$ and $y_0\in\mathcal{J}(f)$. There is no loss of generality in going with the former and also taking $\epsilon_0>0$ so small that there is no critical value in $\left(y_0,y_0+\epsilon_0\right)$. The preimage of $(y_0,y_0+\epsilon_0)$ under the degree $d$ rational map $f:\hat{\Bbb{C}}\rightarrow\hat{\Bbb{C}}$ then consists of $d$ disjoint curves on the Riemann sphere each homeomorphic to the open interval $(0,1)$ and either contained in 
$\hat{\Bbb{R}}$ or completely away from it. The non-real ones among them must appear in conjugate pairs as $f$ is with real coefficients. But the closures of members of any such  pair must have a point of 
$f^{-1}(y_0)\subset\mathcal{J}(f)\subset\hat{\Bbb{R}}$ in common, namely a real Julia point $x_0$. We deduce that $f$ is not locally injective at $x_0$; in particular, $y_0$ is a critical value of $f$. Switching to the interval $(y_0-\epsilon_1,y_0+\epsilon_1)$ symmetric with respect to $y_0$ with $0<\epsilon_1<\epsilon_0$ so small that this interval has no other critical value other than $y_0$, the component $S$ of $f^{-1}\left((y_0-\epsilon_1,y_0+\epsilon_1)\right)$
around $x_0\in f^{-1}(y_0)$ is a ``star-shaped'' subset of the plane symmetric with respect to the real axis with ``arms'' homeomorphic to open intervals that are mapped onto $(y_0-\epsilon_1,y_0+\epsilon_1)$ via $f$ and, aside from the one which is a real open subinterval containing $x_0$, the rest are curves that intersect $\hat{\Bbb{R}}$ only at $x_0$. Such ``non-real arms'' of
 $S\subseteq f^{-1}\left((y_0-\epsilon_1,y_0+\epsilon_1)\right)$
 exist since $x_0$ is a critical point of $f$. But $y_0\in\mathcal{J}(f)\subset\hat{\Bbb{R}}$ cannot be an isolated point of 
 $\mathcal{J}(f)$, so every arm of the star $S$ has a Julia point distinct from $x_0$. Therefore, there are non-real Julia points; a contradiction.\\
\indent
To finish the proof, we need to address the extra information that parts (b) and (c) provide for cases where  $\mathcal{J}(f)$ is ``smooth'', namely is an interval or the whole circle. Since the Julia set is  backward-invariant, 
$f\restriction_{\mathcal{J}(f)}:{\mathcal{J}(f)}\rightarrow{\mathcal{J}(f)}$ is a degree $d$ ramified covering whose ramification structure may be readily investigated. First, suppose $\mathcal{J}(f)=\hat{\Bbb{R}}$. If there is a critical point of  $f$ on  ${\hat{\Bbb{R}}}$, then
$$
\hat{\Bbb{R}}-f^{-1}\left(\left\{v\in\hat{\Bbb{R}}| \,v \text{ a critical value of }f:\hat{\Bbb{C}}\rightarrow\hat{\Bbb{C}}\right\}\right)
$$
is a disjoint union of 
$$\#f^{-1}\left(\left\{v\in\hat{\Bbb{R}}| \,v \text{ a critical value of }f:\hat{\Bbb{C}}\rightarrow\hat{\Bbb{C}}\right\}\right)$$
open intervals restricted to which $f$ yields a covering map onto 
$$\hat{\Bbb{R}}-\left\{v\in\hat{\Bbb{R}}| \,v \text{ a critical value of }f:\hat{\Bbb{C}}\rightarrow\hat{\Bbb{C}}\right\};$$
a union of 
$$\#\left\{v\in\hat{\Bbb{R}}| \,v \text{ a critical value of }f:\hat{\Bbb{C}}\rightarrow\hat{\Bbb{C}}\right\}$$
intervals. So $f$ takes each of the former intervals bijectively onto one of the latter. But $\mathcal{J}(f)=\hat{\Bbb{R}}$ is backward-invariant under the degree $d$ map $f$, so each of the latter intervals has to be covered by $d$ of the former ones. 
This  happens only when 
\small
$$\#f^{-1}\left(\left\{v\in\hat{\Bbb{R}}| \,v \text{ a critical value of }f:\hat{\Bbb{C}}\rightarrow\hat{\Bbb{C}}\right\}\right)
=d\left(\#\left\{v\in\hat{\Bbb{R}}| \,v \text{ a critical value of }f:\hat{\Bbb{C}}\rightarrow\hat{\Bbb{C}}\right\}\right);$$
\normalsize
that is, when $f$ admits $d$ points over each critical value lying on ${\hat{\Bbb{R}}}$,
which is absurd and therefore, 
$f\restriction_{\mathcal{J}(f)}:{\mathcal{J}(f)}\rightarrow{\mathcal{J}(f)}$ is a degree $d$ (unramified) covering of circles.   Finally, when $\mathcal{J}(f)$ is an interval, $f\restriction_{\mathcal{J}(f)}:\mathcal{J}(f)\rightarrow\mathcal{J}(f)$ is a multimodal interval map of topological entropy $\log(d)$. By \eqref{lap number} this entropy cannot be realized if the lap number is less than $d$. Therefore, 
$f\restriction_{\mathcal{J}(f)}$ is $(d-1)$-modal. 
On each of its $d$ laps $f\restriction_{\mathcal{J}(f)}$ restricts to a surjective map onto $\mathcal{J}(f)$
because otherwise, there would be a non-empty open subinterval of $\mathcal{J}(f)$ over points of which 
$f\restriction_{\mathcal{J}(f)}$ admits less than $d$ preimages. The same must hold for fibers of $f:\hat{\Bbb{C}}\rightarrow\hat{\Bbb{C}}$ above these points due to the backward-invariance of $\mathcal{J}(f)$ contradicting 
$\deg f=d$. 
The map has to be boundary-anchored too: if $f$ takes an endpoint of the interval $\mathcal{J}(f)$ to a point of its interior, by continuity, it maps points outside the interval sufficiently close to that endpoint inside $\mathcal{J}(f)$; a contradiction since the points outside the interval are Fatou points.
\end{proof}

\begin{proof}[Proof of Theorem \ref{temp3}]
As proved in Theorem \ref{extremal classification}, $h_\Bbb{R}(f)=\log(d)$ implies $\mathcal{J}(f)\subseteq\hat{\Bbb{R}}$. The detailed discussion in the subsequent Theorem \ref{extremal classification 1} showed that $\mathcal{J}(f)$ must be either a Cantor set on the real circle, a subinterval of it or the entirety of $\hat{\Bbb{R}}$, as outlined in parts (a), (b) and (c) of  
Theorem \ref{extremal classification 1} respectively.  Replacing $f\in\Bbb{R}(z)$ with its conjugate
$$
z\mapsto\left(z\mapsto\frac{z-{\rm{i}}}{z+{\rm{i}}}\right)\circ f\circ\left(z\mapsto\frac{z-{\rm{i}}}{z+{\rm{i}}}\right)^{-1},
$$
it is no loss of generality to assume that the degree $d$ rational map $f$ preserves the unit circle $|z|=1$ instead with its Julia set a subset of the circle. This is  in particular the case when $f$ restricts to a degree $\pm d$ self-cover of the unit circle  since then the closed backward-invariant set $|z|=1$ must contain $\mathcal{J}(f)$ by  Montel's theorem. 
Such a map $f$ can be described by a Blaschke product 
$$
{\rm{e}}^{2\pi{\rm{i}}c}\prod_{i=1}^{d}\left(\frac{z-a_i}{1-\bar{a_i}z}\right)\quad 
\left(|a_1|,\dots,|a_d|<1;\,  \, c\in\Bbb{R}/\Bbb{Z}\right)
$$
as in \eqref{Blaschke temp1} or its post-composition with $z\mapsto\frac{1}{z}$, based on whether the degree of the induced circle map is $+d$ or $-d$.  
There is not any critical point on the unit circle and hence, $\mathcal{J}(f)$ cannot be a subinterval of the circle because Theorem \ref{extremal classification 1}(b) indicates that in such a situation, $\mathcal{J}(f)$ contains critical points. Also notice that  any degree $d$ map with the whole unit circle as its Julia set is one of these Blaschke products due to the fact that  in such a situation
$f\restriction_{\left\{|z|=1\right\}}:\left\{|z|=1\right\}\rightarrow\left\{|z|=1\right\}$
must be a $d$-sheeted covering according to Theorem \ref{extremal classification 1}(c). We can next address parts (i) and (ii) of Theorem \ref{temp3}: the map $f$ given by 
${\rm{e}}^{2\pi{\rm{i}}c}\prod_{i=1}^{d}\left(\frac{z-a_i}{1-\bar{a_i}z}\right)$
(or its post-composition with $z\mapsto\frac{1}{z}$) is hyperbolic if and only if it either has an attracting fixed point on $|z|=1$
-- in which case $\mathcal{J}(f)$ would be a Cantor subset of $|z|=1$ -- or has a  fixed point (necessarily attracting by Schwartz Lemma) inside the open unit disk -- in which case $\mathcal{J}(f)$ coincides with the unit circle. These are the possibilities appeared respectively in
(a) and (c) of Theorem \ref{extremal classification 1} with the real circle in place of the unit circle.  Equation \eqref{temp equation 2}
means that the point ${\rm{e}}^{2\pi{\rm{i}}\theta_0}$ on the unit circle is fixed by 
$f(z)={\rm{e}}^{2\pi{\rm{i}}c}\prod_{i=1}^{d}\left(\frac{z-a_i}{1-\bar{a_i}z}\right)$
and \eqref{temp equation 1} guarantees that it is attracting: a simple logarithmic differentiation shows that on the unit circle
(cf. \cite[Proposition 1]{MR703758})
$$
\frac{f'(z)}{f(z)}=\frac{1}{z}\sum_{i=1}^d\frac{1-|a_i|^2}{|z-a_i|^2};
$$
and the multiplier of ${\rm{e}}^{2\pi{\rm{i}}\theta_0}$ is thus of magnitude 
$\sum_{i=1}^d\frac{1-|a_i|^2}{|{\rm{e}}^{2\pi{\rm{i}}\theta_0}-a_i|^2}$. As for the part (ii) of Theorem \ref{temp3} where $f$ possesses a fixed point in the open unit disk, one can put the fixed point at the origin via conjugation with a suitable biholomorphism of the open unit disk. Therefore, there is a conjugacy that preserves the unit circle and transforms $f(z)$ to 
another Blaschke product of the form 
$$
{\rm{e}}^{2\pi{\rm{i}}c'}z.\prod_{i=1}^{d-1}\left(\frac{z-a'_i}{1-\bar{a'_i}z}\right).
$$
One can kill the only degree of freedom left by conjugating with a rotation of the unit circle and getting rid of ${\rm{e}}^{2\pi{\rm{i}}c'}$. 
This leaves us with a Blaschke product of the form 
$$
z.\prod_{i=1}^{d-1}\left(\frac{z-a''_i}{1-\bar{a''_i}z}\right).
$$
(or its composition with $z\mapsto\frac{1}{z}$ that has $0\mapsto\infty\mapsto 0$ as an attracting $2$-cycle). \\
\indent
Part (iii) of Theorem \ref{temp3} is straightforward. (We have formulated part (iii) for the original real circle as calculating the multiplicity of a parabolic point of the Blaschke product \eqref{Blaschke temp1} on the unit circle leads to complicated algebraic expressions.)  Invoking  Theorem \ref{extremal classification 1}(c), the degree $d$ map $f\in\Bbb{R}(z)$ here must have non-real critical points and a parabolic fixed point of multiplier $\epsilon=\pm1$ on $\hat{\Bbb{R}}$ whose multiplicity should be $3$ if $\epsilon=+1$.  Without any loss of generality, we may assume that the fixed point is $\infty$. Therefore, $f$ can be written as 
$f(z)=\epsilon\left(z+\frac{P(z)}{Q(z)}\right)$ with $P,Q$ coprime polynomials satisfying $\deg P\leq d-1$ and $\deg Q=d-1$.
For $\epsilon=1$ we moreover want the multiplicity of the fixed point $z=\infty$ of $f(z)=z+\frac{P(z)}{Q(z)}$ to be more than $2$ and it is not hard to see that this is equivalent to $\deg P<d-1$. \\
\indent 
Finally, we turn to part (iv) of Theorem \ref{temp3}. Given a degree $d$ map $f$ with $\mathcal{J}(f)$ a subinterval of $\hat{\Bbb{R}}$, we claim that there is a lift of $f$ via a two-sheeted branched covering of $\pi:\hat{\Bbb{C}}\rightarrow\hat{\Bbb{C}}$ to another rational map $g$ whose Julia set is an analytic circle on the Riemann sphere. This will be proved separately in Proposition \ref{Joukowsky}. After suitable changes of coordinates in the domain and the co-domain of $\pi$, we may assume that $\mathcal{J}(g)$ is either the unit circle or the real circle and the degree two map $\mathcal{J}(g)\rightarrow\mathcal{J}(f)$ induced by the semi-conjugacy $\pi$ is either 
\begin{equation}\label{semi-1}
\begin{cases}
\left\{|z|=1\right\}\rightarrow [-2,2]\\
z\mapsto z+\frac{1}{z}
\end{cases};
\end{equation}
or 
\begin{equation}\label{semi-2}
\begin{cases}
\hat{\Bbb{R}}\rightarrow [0,+\infty]\\
z\mapsto z^2
\end{cases}.
\end{equation}
The map $g$  then fits in the description provided by parts (ii) and (iii) of Theorem \ref{temp3}. In the latter case (the non-hyperbolic situation) $g$ is in the form of  $\epsilon\left(z+\frac{P(z)}{Q(z)}\right)$ with 
$\hat{\Bbb{R}}$ as its Julia set and  the semi-conjugacy is the map $z\mapsto z^2$ from \eqref{semi-2}. Hence $f$ is the quotient of $\epsilon\left(z+\frac{P(z)}{Q(z)}\right)$ by the action of $z\mapsto -z$. One can easily verify that this Möbius transformation commutes with $\epsilon\left(z+\frac{P(z)}{Q(z)}\right)$ if and only if one of the polynomials $P(z),Q(z)$ is even and the other is odd. In the hyperbolic case, the semi-conjugacy is the Joukowsky map $z\mapsto z+\frac{1}{z}$ from \eqref{semi-1} and thus $f$ must be the quotient of a Blaschke product of the form \eqref{Blaschke temp1} that has the unit circle as its Julia set by the action of 
$z\mapsto\frac{1}{z}$. Following previous discussions regarding the part (ii) of Theorem \ref{temp3}, $g$ must have an attracting fixed point inside the unit disk which can assumed to be zero; and hence is in the form of 
$$
{\rm{e}}^{2\pi{\rm{i}}c}z.\prod_{i=1}^{d-1}\left(\frac{z-a_i}{1-\bar{a_i}z}\right).
$$
It is not hard to check that the map above commutes with $z\mapsto\frac{1}{z}$ if and only if ${\rm{e}}^{2\pi{\rm{i}}c}=\pm 1$ and the points $a_1,\dots,a_{d-1}$ of the open unit disk are located symmetrically with respect to the real axis; hence \eqref{Blaschke temp3}.
\end{proof}

Next, we have the proposition below that concludes the proof of Theorem \ref{temp3} and furthermore, provides a description of rational maps with interval Julia sets.
\begin{proposition}\label{Joukowsky}
Given a rational map $f$ of degree $d$ whose Julia set is the interval $[-2,2]$, there exists a  rational map $g$ of degree $d$ with the unit circle as its Julia set which is semi-conjugate to $f$  
\begin{equation}\label{semi-conjugacy}
\mathfrak{j}\circ g=f\circ\mathfrak{j}
\end{equation}
via the Joukowsky  map 
\begin{equation}\label{alpha}
\mathfrak{j}(z):=z+\frac{1}{z}.
\end{equation}
\end{proposition}

\begin{proof}

We denote the restrictions of \eqref{alpha} to upper and lower unit semi-circles by $\mathfrak{j}_+$ and $\mathfrak{j}_-$.  
\begin{equation}\label{alpha+-}
\begin{cases}
\mathfrak{j}_+:\left\{e^{2\pi{\rm{i}}\theta}\,|\, 0\leq\theta\leq\pi\right\}\rightarrow[-2,2]\\
e^{2\pi{\rm{i}}\theta}\mapsto 2\cos(\theta)
\end{cases};
\quad 
\begin{cases}
\mathfrak{j}_-:\left\{e^{2\pi{\rm{i}}\theta}\,|\, \pi\leq\theta\leq 2\pi\right\}\rightarrow[-2,2]\\
e^{2\pi{\rm{i}}\theta}\mapsto 2\cos(\theta)
\end{cases}.
\end{equation}
These diffeomorphisms  are respectively orientation-reversing and orientation-preserving once the unit semi-circles and the interval $[-2,2]$ are equipped with orientations inherited from the obvious positive orientations of the unit circle and the real line. Theorem \ref{extremal classification 1}(b) indicates that $f\restriction_{[-2,2]}:[-2,2]\rightarrow[-2,2]$ has $d$ surjective monotonic pieces. Hence indexing its turning points as 
\begin{equation}\label{turning}
t_0=-2<t_1<\dots<t_{d-1}<t_d=2,
\end{equation}
the map $f$ takes each subinterval $[t_i,t_{i+1}]$ bijectively onto $[-2,2]$ with these bijections alternatingly increasing and decreasing. Consider the preimages of points in \eqref{turning} under the maps in \eqref{alpha+-}:
\begin{equation}\label{preimages1}
p_{i+}=\mathfrak{j}_+^{-1}(t_i)\quad p_{i-}=\mathfrak{j}_-^{-1}(t_i) \quad(0<i<d). 
\end{equation}
Clearly, one has 
\begin{equation}\label{preimages2}
\mathfrak{j}_\pm^{-1}(2)=1<p_{(d-1)+}<\dots<p_{1+}<\mathfrak{j}_\pm^{-1}(-2)=-1<p_{1-}<\dots<p_{(d-1)-}\,;
\end{equation}
in the counter-clockwise direction of the unit circle.
Moving along the unit circle from $1$ to $-1$ and then back to $1$, the points above cut $2d$ consecutive closed arcs $C_1,\dots,C_{2d}$ with the first $d$ covering the closed upper semi-circle and the rest covering the closed lower semi-circle. These arcs are disjoint aside from the adjacent arcs $C_{i}$ and $C_{i+1}$ (indices considered cyclically modulo $2d$) intersecting at a point from \eqref{preimages2}.
Next, we define a self-map $g$ of the unit circle as follows:  
\begin{equation}\label{the circle map}
g\restriction_{C_i}=
\begin{cases}
(\mathfrak{j}_+)^{-1}\circ f\circ\mathfrak{j}_+\restriction_{C_i}\quad \text{if } 1\leq i\leq d\,\&\,i \text{ odd}\\
(\mathfrak{j}_-)^{-1}\circ f\circ\mathfrak{j}_+\restriction_{C_i}\quad \text{if } 1\leq i\leq d\,\&\,i \text{ even}\\
(\mathfrak{j}_+)^{-1}\circ f\circ\mathfrak{j}_-\restriction_{C_i}\quad \text{if } d+1\leq i\leq 2d\,\&\,i \text{ odd}\\
(\mathfrak{j}_-)^{-1}\circ f\circ\mathfrak{j}_-\restriction_{C_i}\quad \text{if } d+1\leq i\leq 2d\,\&\,i \text{ even}
\end{cases}
\quad (1\leq i\leq 2d).
\end{equation}
The map is clearly well defined since adjacent arcs overlap at points whose image under $\mathfrak{j}$ 
(and so under $\mathfrak{j}_+$ or $\mathfrak{j}_-$, as appropriate) is a point from $f^{-1}\left(\left\{\pm 2\right\}\right)$ in 
\eqref{turning}; and we have 
$\mathfrak{j}_\pm^{-1}(2)=\{1\}$ and $\mathfrak{j}_\pm^{-1}(-2)=\{-1\}$. It is evident from \eqref{the circle map} that $g$ bijects each of the arcs $C_1,\dots,C_{2d}$ alternatingly onto either the upper or the lower semi-circle. Notice that these diffeomorphisms are of the same degree; which is,  $1$ if the last lap
$$f\restriction_{[t_{d-1},t_d]=\mathfrak{j}_+(C_1)=\mathfrak{j}_-(C_{2d})}$$
is increasing and $-1$ if the lap is decreasing. For $1\leq i<d$ (respectively or $d+1\leq i<2d$), the diffeomorphisms
$(\mathfrak{j}_+)^{-1}\circ f\circ\mathfrak{j}_+\restriction_{C_i}$ and 
$(\mathfrak{j}_-)^{-1}\circ f\circ\mathfrak{j}_+\restriction_{C_{i+1}}$
(resp. $(\mathfrak{j}_+)^{-1}\circ f\circ\mathfrak{j}_-\restriction_{C_i}$ and 
$(\mathfrak{j}_-)^{-1}\circ f\circ\mathfrak{j}_-\restriction_{C_{i+1}}$)
are of the same degree as the degrees of $\mathfrak{j}_\pm$ and also those of the adjacent laps 
$f\restriction_{\mathfrak{j}(C_i)}$, 
$f\restriction_{\mathfrak{j}(C_{i+1})}$ are opposite. The same is true for $i=d$, when it comes to the degrees of 
$g\restriction_{C_d}=\mathfrak{j}_{(-1)^{d-1}}^{-1}\circ f\circ\mathfrak{j}_+\restriction_{C_d}$ and
$g\restriction_{C_{d+1}}=\mathfrak{j}_{(-1)^d}^{-1}\circ f\circ\mathfrak{j}_-\restriction_{C_{d+1}}$: they are just $(-1)^{d-1}$
 times the degree of the lap 
$$f\restriction_{\mathfrak{j}_+(C_d)=\mathfrak{j}_-(C_{d+1})=[t_0,t_1]}.$$ 
\indent The discussion above implies that \eqref{the circle map} defines an unramified $d$-sheeted self-cover $g$ of the unit circle satisfying \eqref{semi-conjugacy}. The map is clearly real analytic away from the endpoints of $C_i$'s 
(as appeared in \eqref{preimages1}) where the formula for $g$ changes. Thus every interior point $x$ on an arc $C_i$ admits an open neighborhood $U_x$ in the complex plane on which there is an extension of $g$ to a local biholomorphism.  This holomorphic extension of $g$ of course satisfies $\mathfrak{j}\circ g=f\circ\mathfrak{j}$ as well. Now notice that the unit circle is the preimage of the Julia set $[-2,2]$ of $f$ under $\mathfrak{j}$. By the complete invariance of the Julia set, any such extension of $g$ takes points off the unit circle to points off the unit circle. As this holomorphic extension preserves the orientation, it must take points out of the unit disk inside and vice versa if $g\restriction_{C_i}$ is orientation-reversing; while it keeps
$U_x\cap\left\{|z|<1\right\}$ and $U_x\cap\left\{|z|>1\right\}$ invariant if $g\restriction_{C_i}$ is orientation-preserving. We conclude that there is a holomorphic extension of the circle map $g$ to the open neighborhood $U:=\bigcup_xU_x$ of the complement
$$
\left\{|z|=1\right\}-\left\{p_{1\pm},\dots,p_{(d-1)\pm},\pm 1\right\}
$$
of the endpoints \eqref{preimages2} of arcs $C_1,\dots,C_{2d}$; and this extension, denoted by the same symbol $g$, either preserves or interchanges $U\cap\left\{|z|<1\right\}$ and 
$U\cap\left\{|z|>1\right\}$, based on whether the degree of the unramified covering 
\begin{equation}\label{circle map1}
g\restriction_{\left\{|z|=1\right\}}:\left\{|z|=1\right\}\rightarrow\left\{|z|=1\right\}
\end{equation}
is positive or negative respectively. But either on the open unit disk or on its complement the map $\mathfrak{j}$ from \eqref{alpha} restricts to a biholomorphism onto the Fatou set $\hat{\Bbb{C}}-[-2,2]$ of $f$. Consequently, away from the unit circle one can solve 
$\mathfrak{j}\circ g=f\circ\mathfrak{j}$ for $g$: if the degree of the circle map \eqref{circle map1} is $d$,  
$g\restriction_{U\cap\left\{|z|<1\right\}}$ and $g\restriction_{U\cap\left\{|z|>1\right\}}$ are given by 
$$
\left(\mathfrak{j}\restriction_{\left\{|z|<1\right\}}:\left\{|z|<1\right\}\rightarrow \hat{\Bbb{C}}-[-2,2]\right)^{-1}\circ f\circ \left(\mathfrak{j}\restriction_{\left\{|z|<1\right\}}:\left\{|z|<1\right\}\rightarrow \hat{\Bbb{C}}-[-2,2]\right)
$$
and 
$$
\left(\mathfrak{j}\restriction_{\left\{|z|>1\right\}}:\left\{|z|>1\right\}\rightarrow \hat{\Bbb{C}}-[-2,2]\right)^{-1}\circ f\circ \left(\mathfrak{j}\restriction_{\left\{|z|>1\right\}}:\left\{|z|>1\right\}\rightarrow \hat{\Bbb{C}}-[-2,2]\right);
$$
while they are respectively given by 
$$
\left(\mathfrak{j}\restriction_{\left\{|z|>1\right\}}:\left\{|z|>1\right\}\rightarrow \hat{\Bbb{C}}-[-2,2]\right)^{-1}\circ f\circ \left(\mathfrak{j}\restriction_{\left\{|z|<1\right\}}:\left\{|z|<1\right\}\rightarrow \hat{\Bbb{C}}-[-2,2]\right)
$$
and 
$$
\left(\mathfrak{j}\restriction_{\left\{|z|<1\right\}}:\left\{|z|<1\right\}\rightarrow \hat{\Bbb{C}}-[-2,2]\right)^{-1}\circ f\circ \left(\mathfrak{j}\restriction_{\left\{|z|<1\right\}}:\left\{|z|>1\right\}\rightarrow \hat{\Bbb{C}}-[-2,2]\right);
$$
if the degree of \eqref{circle map1} is $-d$. We conclude that $g$ can be holomorphically extended even further to 
$\hat{\Bbb{C}}-\left\{p_{1\pm},\dots,p_{(d-1)\pm},\pm 1\right\}$. By the Identity Principle,  the extension again satisfies the functional equation $\mathfrak{j}\circ g=f\circ\mathfrak{j}$. Consequently, for any sequence $\{a_n\}_n$ tending to one of the points of the unit circle 
$$\left\{p_{1\pm},\dots,p_{(d-1)\pm},\pm 1\right\}=\mathfrak{j}^{-1}\left(f^{-1}\left(\{\pm 2\}\right)\right)$$
excluded above, any convergent subsequence of $\{g(a_n)\}_n$ tends to an element of 
$\mathfrak{j}^{-1}\left(\{\pm 2\}\right)=\{\pm 1\}$. We deduce that none of these isolated points is an essential singularity and so  $g$ extends holomorphically to the whole Riemann sphere. Therefore, we have constructed a rational function $g\in\Bbb{C}(z)$ semi-conjugate to $f$ as in \eqref{semi-conjugacy}. Its Julia set is of course the preimage of $\mathcal{J}(f)=[-2,2]$ under the Joukowsky map $\mathfrak{j}$; namely, the unit the circle.  This finishes the proof. 
\end{proof}

\begin{remark}
The dynamics on the Julia set can be easily described in many cases of Theorem \ref{extremal classification 1}. In 
\ref{extremal classification 1}(a), if the critical points are away from the Julia set (e.g. when the map is hyperbolic), they all lie in a parabolic or attracting immediate basin. The dynamics on the Julia set is then conjugate to the one-sided shift on $d$ symbols \cite{MR1376666}. In \ref{extremal classification 1}(b), the Julia set of $f\in{\rm{Rat}}_d(\Bbb{R})$ is a compact interval on which $f$ restricts to a boundary-anchored expansive interval map with $d$ surjective laps. Such an interval map is conjugate to a piecewise linear map of the same modality. Finally, in  \ref{extremal classification 1}(c) where $\mathcal{J}(f)=\hat{\Bbb{R}}$, once $f$ is without parabolic fixed points (which there means that $f$ is hyperbolic), $f\restriction_{\hat{\Bbb{R}}}:\hat{\Bbb{R}}\rightarrow\hat{\Bbb{R}}$ would be an expanding circle map of degree $d$ and thus conjugate to one of the self-covers $z\mapsto z^d$ or $z\mapsto z^{-d}$ of the unit circle \cite{MR0240824}. In both \ref{extremal classification 1}(b) and \ref{extremal classification 1}(c), the multiplier of any periodic point of period $n$ of $f$ is $\pm d^n$ except the multiplier of the unique non-repelling fixed point; compare with  
\cite[Theorem 8.1]{MR623533}.
\end{remark}

\begin{example}\label{possibilities}
The examples of quadratic rational maps appeared in \cite[Problem 10-e]{MR2193309} embody all possibilities for the Julia set outlined in Theorem \ref{extremal classification 1}. 
The real map $f(z)=z-\frac{1}{z}$ takes 
$\hat{\Bbb{R}}$ to $\hat{\Bbb{R}}$ via a degree two covering and 
has a parabolic fixed point of multiplier $+1$ and multiplicity $n+1=3$ at infinity. The upper and the lower half planes are the corresponding parabolic basins and the Julia set is $\hat{\Bbb{R}}$.
Post-composing this map with its automorphism 
$z\mapsto -z$ does not change the Julia set, so $z\mapsto-\left(z-\frac{1}{z}\right)$ also has $\hat{\Bbb{R}}$ as its Julia set; but, here the multiplier of the parabolic point $\infty$ is $-1$. 
Things are different for the real map  $f(z)=z-\frac{1}{z}+1$ where the fixed point at infinity is of multiplicity $n+1=2$; hence only  one immediate basin. It is not hard to verify that this only Fatou component of $f$ contains $\hat{\Bbb{C}}-[-\infty,1]$; and the Julia set of $f$ is a Cantor subset of the interval $I=[-\infty,1]$. Observe that   
$$f^{-1}\left([-\infty,1]\right)=[-\infty,-1]\cup[0,1]\subsetneq[-\infty,1].$$
Finally, consider the map $f(z)=z+\frac{1}{z}-2$. Again $p=\infty$ is a fixed point of multiplier $+1$ and of multiplicity $n+1=2$. This time the Julia set is the completely invariant interval $I=[0,+\infty]$. \\
\indent
Notice that the first and the third examples above (with circle and interval Julia sets) are semi-conjugate as predicted by Theorem \ref{temp3}: 
$
\left(z-\frac{1}{z}\right)^2=z^2+\frac{1}{z^2}-2.
$
One other example of such semi-conjugacies is given by Chebyshev polynomials: the $d^{\rm{th}}$ Chebyshev polynomial $T_d(z)$ (normalized to be monic via a linear conjugation) satisfies
$T_d\left(z+\frac{1}{z}\right)=z^d+\frac{1}{z^d}$; its Julia set is $[-2,2]$, the image of the unit circle which is the Julia set of 
$z\mapsto z^d$ under the semi-conjugacy $z\mapsto z+\frac{1}{z}$. Compare with \cite[Problems 7-c \& 7-d]{MR2193309}. 
\end{example}

\begin{remark}
In view of Corollary \ref{main-corollary}, one can directly count the dimension of the families exhibited in Theorem \ref{temp3} and compare it with the dimension of the dynamical moduli space in \eqref{dimension formula}. Given a map $f$ with 
$h_{\Bbb{R}}(f)=\log(d)$, the description of the Fatou components in Theorem \ref{extremal classification 1} indicates that there is no rotation domain. There is no invariant line field because the Julia set is contained in the measure zero subset $\hat{\Bbb{R}}$ (or $|z|=1$) of the complex plane. Therefore, away from the post-critical loci, for a generic choice of a map $f$ from one of the families appeared in Theorem \ref{temp3} one can rewrite \eqref{dimension formula} as
\begin{equation}\label{dimension formula 1}
\dim_{\Bbb{C}}{\rm{M}}(f)=\#\text{critical points in the Fatou set}-\#\text{cycles of parabolic basins}.
\end{equation} 
For each family from Theorem \ref{temp3}, subtracting the number of degrees of freedom the corresponding normal form from the real dimension of the family must yield the same number as \eqref{dimension formula 1}.  
\begin{itemize}
\item For the family of maps in \eqref{Blaschke temp2} the formula yields $2d-2$ as all $2d-2$ critical points are Fatou and there is no parabolic cycle. The family is dependent on $d-1$ complex (and hence $2d-2$ real parameters) and there is no continuous group of transformations preserving this normal form.
\item In Theorem \ref{temp3}(iii), the two-dimensional group of real affine transformations acts on the space of maps
$z+\frac{P(z)}{Q(z)}$  
where $P,Q$ are real coprime polynomials of degrees $<d-1$ and $d-1$ respectively. The space of these maps is of dimension 
$$
\left[(d-1)+d\right]-1=2d-2,
$$
and the property of critical points being off the real line is an open one. Hence the projection of the family to $\mathcal{M}'_d$
would be of real dimension $(2d-2)-2=2d-4$. The formula \eqref{dimension formula 1} yields the same: There are $2d-2$ critical points in the Fatou set and two fixed parabolic basins. Things are slightly different for maps $-\left(z+\frac{P(z)}{Q(z)}\right)$ as the degree of $Q$ is allowed to be $d-1$ too, so the dimension $(2d-1)-2=2d-3$
for the projection of this family to $\mathcal{M}'_d$. The multiplier of the parabolic fixed point is $-1$, so the dynamics interchanges the open half planes and there is a $2$-cycle of parabolic basins. Hence one gets $(2d-2)-1=2d-3$ form  
\eqref{dimension formula 1}. 
\item The normal form  \eqref{Blaschke temp3} has no continuous group of automorphisms and amounts to several families (all of the same dimension) based on the choice of the plus or minus sign and the number of real parameters among $a_1,\dots,a_{d-1}$. For instance, we can take 
$$
z.\prod_{i=1}^{d-1}\left(\frac{z-a_i}{1-a_i z}\right)\quad 
\left(a_1,\dots,a_{d-1}\in (-1,1)\right)
$$
which is of real dimension $d-1$. The Julia set for any of hyperbolic maps discussed in  Theorem \ref{temp3}(iv) would be an interval on which the map is $(d-1)$-modal, so exactly 
$(2d-2)-(d-1)=d-1$ critical points in the Fatou set. 
\end{itemize}
\end{remark}

\appendix     
\section{The Julia set and the real circle}\label{appendix A}

The Julia set $\mathcal{J}(f)$ of a rational map $f$ of degree $d\geq 2$ intuitively incorporates  the ``chaotic'' parts of the dynamics of $f$.  So given a rational map $f\in\Bbb{R}(z)$ of degree $d$, it is only the intersection of the Julia set of $f$ with $\hat{\Bbb{R}}$
-- the \textit{real Julia set} $\mathcal{J}_\Bbb{R}(f):=\mathcal{J}(f)\cap \hat{\Bbb{R}}$ -- that affects the entropy of 
$f\restriction_{\hat{\Bbb{R}}}:\hat{\Bbb{R}}\rightarrow\hat{\Bbb{R}}$. 
This is due to the fact that 
$\mathcal{J}(f)\cap \hat{\Bbb{R}}$ contains the portion of the non-wandering set of $f\restriction_{\hat{\Bbb{R}}}$
that really matters to entropy.  The article \cite{MR3311773} for instance relates the real entropy of a quadratic polynomial such as  $x\mapsto x^2+c\,\, (-2\leq c\leq\frac{1}{4})$ 
to the Hausdorff dimension of the set of rays that land on the intersection of the filled Julia set with the real axis. The interplay between the real Julia set and the real entropy is the main theme of this appendix. The main goal of 
\S\ref{structure subsection} is to prove Theorem \ref{temp2} which yields a  description of the real Julia set and the dynamics on it. \S\ref{examples subsection} is devoted to some remarks and examples regarding  Theorems \ref{temp2} and \ref{Structure}. \\
\indent  
Certain  topics from the dynamics of interval and circle maps will be invoked. See books \cite{MR1239171}, \cite{2015arXiv150403001R}  or the survey article \cite{MR2600680} for the background material.\\  
\indent 
At last, to avoid confusion when going back and forth between real and complex contexts, we remind the reader that in the subsystem obtained  by restricting a rational map to an invariant interval or circle, an infinite orbit might tend to a periodic cycle which is  
merely non-repelling for the rational map\footnote{Keep in mind that a repelling cycle of the complex map is topologically repelling 
\cite[Lemma 8.10]{MR2193309}.};  we can have a parabolic cycle lying on the interval or circle that as a cycle of the restricted map is \textit{weakly attracting} or possess a basin which is only $1$-sided. Of course, this possibly $1$-sided basin of the interval or circle map is contained in a parabolic basin of the ambient rational map. 

\subsection{Structure of the real Julia set}\label{structure subsection}
In this appendix we prove a structure theorem for the system  $\left(f,{\mathcal{J}_\Bbb{R}(f)}\right)$:
\begin{theorem}\label{temp2}
The non-wandering set of the dynamical system 
$f\restriction_{\mathcal{J}_\Bbb{R}(f)}:\mathcal{J}_\Bbb{R}(f)\rightarrow\mathcal{J}_\Bbb{R}(f)$ 
is either the whole circle or the union of finitely many intervals and a totally disconnected subsystem. Moreover, the system
has a closed subsystem of the same topological entropy in which preperiodic points of $f$ are dense.\footnote{The subsystem is allowed to be empty when the real entropy vanishes. This may occur even if the Julia set intersects the real axis non-trivially along an interval; check  Example \ref{Blaschke1} and Remark \ref{elaborate} for maps of real entropy zero whose real Julia sets are the whole real circle or an interval of it.}
\end{theorem}
We first establish the following more detailed description of the real Julia set. 

\begin{theorem}\label{Structure}
Let $f\in\Bbb{R}(z)$ be of degree $d\geq 2$ and $\emptyset\neq\mathcal{J}(f)\cap\hat{\Bbb{R}}\subsetneq\hat{\Bbb{R}}$. Then 
$\mathcal{J}_\Bbb{R}(f)=\mathcal{J}(f)\cap\hat{\Bbb{R}}$ can be uniquely written as the  union 
below of closed $f$-invariant subsets:
\begin{equation}\label{union}
\mathcal{J}_\Bbb{R}(f)=\left(\bigsqcup_{I\in\mathcal{I}} I\right)\bigcup Z;
\end{equation}
where
\begin{enumerate}
\item  $\mathcal{I}$ is a countable collection of  pairwise disjoint compact non-degenerate intervals on which   $f$ acts by taking an interval  into another while preserving the boundaries;
\item every $I\in\mathcal{I}$ is eventually periodic under $f$; that is, some iterate of it is contained 
in an interval $I'\in\mathcal{I}$ satisfying $f^{\circ p}(I')\subseteq I'$ for some appropriate $p>0$; 
furthermore, there are only finitely many such periodic intervals $I'\in\mathcal{I}$; 
\item $Z$ is a totally disconnected closed subsystem in which preperiodic points are dense and moreover can intersect an interval $I\in\mathcal{I}$ only at its endpoints;
%;
%\item the system $f\restriction_C:C\circlearrowleft$ is a disjoint union of finitely many $1$-sided subshifts on at most $d$ letters and each such component is a subshift of finite type provided it does not contain any critical point of $f$. 
\end{enumerate}  
Furthermore, there is a Cantor subsystem $C$ of $f\restriction_Z:Z\rightarrow Z$ with the same entropy and a (possibly vacuous) finite set 
$S$ of isolated periodic points of the system $f\restriction_{\mathcal{J}_\Bbb{R}(f)}$ such that $Z$ can be written as the union of $C$ and a closed countable subsystem which consists of  backward iterates and the $\alpha$-limit set of $S$:
\begin{equation}\label{union2}
Z=C\bigcup\overline{\left(\bigcup_{n}f^{\circ -n}(S)\right)\bigcap\hat{\Bbb{R}}}.
\end{equation}
\end{theorem}  
This emphasis on the density of preperiodic points in this theorem and the succeeding proposition is due to the fact that  we intuitively expect 
$f\restriction_{\mathcal{J}_\Bbb{R}(f)}:\mathcal{J}_\Bbb{R}(f)\rightarrow\mathcal{J}_\Bbb{R}(f)$
to exhibit ``chaotic behavior'', at least in some subsystems,  and the density of periodic points is often one of the manifestations  of ``chaos''. Here, the chaos is intended in the sense of Devaney \cite[\S 1.8]{MR1046376}:

\textit{ A continuous map $f:(X,d)\rightarrow (X,d)$ on the metric space $(X,d)$ is called chaotic if:
\begin{enumerate}[I.]
\item $f$ is transitive;
\item the periodic points of $f$ are dense;
\item $f$ has sensitive dependence on initial conditions.
\end{enumerate}}
\noindent
It turns out that for $X$ infinite, the last condition is always implied by the first two and for interval maps solely the first condition suffices;
see \cite{MR1157223} and \cite[Proposition 7.2]{2015arXiv150403001R} respectively. 

\begin{proof}[Proof of Theorem \ref{Structure}]
{The real Julia set  $\mathcal{J}_\Bbb{R}(f)$ is assumed to be a proper non-empty subset of 
$\hat{\Bbb{R}}$, and its connected components are thus either non-degenerate compact intervals or points. The continuous transformation $f\restriction_{\mathcal{J}_\Bbb{R}(f)}:\mathcal{J}_\Bbb{R}(f)\rightarrow \mathcal{J}_\Bbb{R}(f)$
must take a point component to a point component and an interval component to an interval component because it cannot collapse an  interval to a point due to the analyticity of $f$. 
Moreover, if $f(I)\subseteq I'$ for two interval components, then $f\left(\partial I\right)\subseteq\partial I'$
as otherwise, by continuity, there will be an open interval around an endpoint of $I$ that goes into 
$I'\subseteq\mathcal{J}(f)$ via $f$, so it is also a subset of  $\mathcal{J}(f)$; this is a contradiction since then 
the connected component of $\mathcal{J}(f)\cap\hat{\Bbb{R}}$  which has that endpoint would be strictly larger than $I$.
\\
\indent
Therefore,
the real Julia set may be written as a disjoint union 
\begin{equation}\label{union1}
\mathcal{J}_\Bbb{R}(f)=\left(\bigsqcup_{I\in\mathcal{I}} I\right)\bigsqcup Z' 
\end{equation}
of countably many interval components and the  totally disconnected set $Z'$ of point components. The map $f$
keeps $Z'$ invariant and maps an element of $\mathcal{I}$ into another. We claim that every $I\in\mathcal{I}$ has two distinct iterates that collide. If not,  the forward iterates 
$I, f(I), f^{\circ 2}(I),\dots$ are pairwise disjoint. In the terminology of real one-dimensional  dynamics, such an interval is called 
\textit{a wandering interval} unless $f^{\circ n}(I)$ tends to an attracting cycle as $n\to\infty$.
Intervals $f^{\circ n}(I)\subseteq\mathcal{J}_{\Bbb{R}}(f)$ cannot converge an attracting cycle of 
$f\restriction_{\hat{\Bbb{R}}}$ since a Julia point cannot lie in an attracting or parabolic basin of $f$.
Now, we invoke a powerful result from real one-dimensional  dynamics that immediately rules out existence of wandering intervals: 
\begin{itemize}
\item[]\textit{ An analytic map $S^1\rightarrow S^1$ does not admit any wandering interval.\footnote{The non-existence of wandering intervals has been proved by several authors in various generalities and 
was finally settled in \cite{MR1161268} in its full strength. The statement above is an immediate corollary of a version of this important result outlined in \cite[p. 260, Theorem A]{MR1239171} which is formulated for $C^2$ maps without any \textit{flat} critical point.}}
\end{itemize}
\noindent This  establishes that every interval in the union \eqref{union1} is eventually preserved by $f$: If 
$f^{\circ n}(I)\cap f^{\circ m}(I)\neq\emptyset$
for $n>m$, then both of these iterates are contained in the same $I'\in\mathcal{I}$, a connected component which is then invariant under $f^{\circ (n-m)}$ as $I'\cap f^{\circ (n-m)}(I')\neq\emptyset$.\\
\indent
 If $I\in\mathcal{I}$ is preserved by an iterate  $f^{\circ p}$, then the restriction $f^{\circ p}\restriction_{I}$ is boundary-anchored and 
 possesses a critical point  since otherwise  $f^{\circ p}\restriction_{I}$ would be monotonic. It is easy to verify that 
the domain of a continuous monotonic interval map can be decomposed to the disjoint union of a subset consisting of points of period at most two and countably many $1$-sided attracting basins of periodic orbits of period at most two. 
When the map is 
$C^1$ a periodic cycle admitting a $1$-sided attracting basin has its multiplier in $[-1,1]$. But an attracting or parabolic basin of $f$ does not intersect $I\subseteq\mathcal{J}(f)$
and so the only other possibility is for every point in $I$ to be of period at most two  which is absurd. By the chain rule, a critical point of $f^{\circ p}$ has to be in the backward orbit of a critical point of $f$. So each periodic cycle of intervals among members of $\mathcal{I}$ would completely contain a critical orbit of $f$. 
We deduce that there are  at most $2d-2$
such intervals in $\mathcal{I}$. \\
\indent
Next, we decompose $Z'$. Consider the set $Z''$ of points $z\in \mathcal{J}_{\Bbb{R}}(f)$ for which there is a non-degenerate interval $K_z:=[z,z+\epsilon_z]$
or
$K_z:=[z-\epsilon_z,z]$
that  intersects $\mathcal{J}(f)$ only at $z$. 
Since 
$Z'$ is formed by the point components of the closed set 
$\mathcal{J}_{\Bbb{R}}(f)=\mathcal{J}(f)\cap\hat{\Bbb{R}}$,  $Z''$ is  clearly dense in $Z'$. We claim that all points of $Z''$ are preperiodic. 
Without any loss of generality, let $K_z=[z,z+\epsilon]$; the other case is no different.
Due to non-existence of wandering  Fatou components, the Fatou component of $f$  containing $(z,z+\epsilon_z]$ is eventually periodic and we claim that it cannot end up in a cycle of rotation domains of $f$; in fact it is not hard to show that if a rotation domain of the rational map $f\in\Bbb{R}(z)$ intersects the real line then it must be a Herman ring containing the whole circle $\hat{\Bbb{R}}$ (Lemma \ref{no Siegel}) and this cannot occur here since we have assumed $\hat{\Bbb{R}}$
is not entirely Fatou. Consequently, the points of $(z,z+\epsilon_z]$ 
 under iteration tend  to a real attracting or parabolic cycle of the rational map $f$. 
Without any loss of generality and by replacing $f$ with an iterate if necessary, one can assume that the period is one.  
 As a fixed point of the circle map
$f\restriction_{\hat{\Bbb{R}}}:\hat{\Bbb{R}}\rightarrow\hat{\Bbb{R}}$, 
this cycle admits a (possibly $1$-sided) immediate basin $\mathcal{A}_\Bbb{R}$.
The closure 
$\overline{\mathcal{A}_\Bbb{R}}$
of this immediate basin is a closed interval whose endpoints are indeed Julia points. The continuity argument that appeared before shows that $f$ restricts to a boundary-anchored transformation of $\overline{\mathcal{A}_\Bbb{R}}$; in particular, the endpoints of the immediate basin are preperiodic. Sufficiently high iterates of $K_z=[z,z+\epsilon_z]$ must intersect the open interval $\mathcal{A}_\Bbb{R}\subset\mathcal{F}(f)$
but cannot be contained in that since they have Julia points from the orbit $\mathcal{O}(z)$. Therefore, there must be an endpoint of the closure $\overline{\mathcal{A}_\Bbb{R}}$ contained in an iterate $f^{\circ p}(K_z)$ of $K_z$. But
due to the complete invariance of the Fatou and the Julia sets, $f^{\circ p}(z)$ is the unique Julia point of the interval $f^{\circ p}({K_z})$ and so
$f^{\circ p}(z)$ is an endpoint of $\overline{\mathcal{A}_\Bbb{R}}$  and thus preperiodic.\\
\indent
Next, we  exploit what has been established about points of $Z''$ being preperiodic. This implies that $Z'$ and therefore its closure, denoted by $Z$, have the dense subset $Z''$ of preperiodic points. Just like $Z'$, the closure $Z$ must be $f$-invariant as well. This is totally disconnected and can intersect each interval component $I\in\mathcal{I}$ of $\mathcal{J}_\Bbb{R}(f)$ only at its endpoints. Therefore, the union from \eqref{union1} can be modified to the union below of closed subsystems: 
\small
$$
\mathcal{J}_\Bbb{R}(f)=\left(\bigsqcup_{I\in\mathcal{I}}I\right)\bigcup Z.$$
\normalsize
This is the desired decomposition \eqref{union} predicted by the theorem.
\\
\indent
To finish the proof, we merely need to 
decompose the totally disconnected system $Z$ to a Cantor part and a countable part. Lemma \ref{decomposition} provides a canonical way to do so: $Z$ is the union a Cantor set $C$ and the countable closed subspace $D$ which is the closure of the set of isolated points of $Z$.\footnote{There are indeed examples where the real Julia set has isolated points; for instance, 
consider the intersection of the Julia set of $z^2-1$ (the \textit{basilica}) with the real axis.}
The proof of Lemma  \ref{decomposition} establishes that $C$ is in fact the set of \textit{condensation points} of $Z$ and $f$, due to finiteness of its fibers, takes a condensation point of the subsystem $Z$ to another such  point. Hence the Cantor subspace $C$ is indeed a subsystem. Next, as the description of $D$ as the closure of the subset of isolated points of $Z$ suggests, we have to study isolated points of $Z$. Clearly, an isolated point of $Z$ is an isolated point of the real Julia set $\mathcal{J}_\Bbb{R}(f)$ unless it is an endpoint of an interval component in which case it is already included in 
$\bigsqcup_{I\in\mathcal{I}}I$
and excluding it from $Z$ will not affect the compactness of $Z$. Thus we  assume that isolated points of $Z$ remain isolated in the bigger space $\mathcal{J}_\Bbb{R}(f)$. 
Since the preperiodic points of $Z$ are dense in it, any isolated  
 point $z_0$ of it
has to be preperiodic, say $f^{\circ m}(z_0)=f^{\circ (m+p)}(z_0)$. 
So $z_0$ ends up in a cycle of period $p$ and it suffices to show that only finitely many such cycles of $f$ can arise in this manner form isolated points of  $\mathcal{J}_\Bbb{R}(f)$.     
Choose $\delta>0$ so small that either 
$(z_0,z_0+\delta)$ or $(z_0-\delta,z_0)$ is a subset of the Fatou set. Since $f^{\circ m}(z_0)=f^{\circ (m+p)}(z_0)$,
either the $(m+p)^{\rm{th}}$ or the $(m+2p)^{\rm{th}}$ forward iterate of this interval intersects its
$m^{\rm{th}}$ iterate non-trivially based on whether 
$\left(f^{\circ p}\right)'\left(f^{\circ m}(z_0)\right)$ is positive or negative. In particular, there is a Fatou component of period
$p$ or $2p$. But there are only finitely many periodic Fatou components and this puts an upper bound  on the eventual periods of points like $z_0$; consequently, there are only finitely many of them. This establishes the description \eqref{union2} of $Z$.
The entropy of $f\restriction_{Z}$ is the same as that of $f\restriction_{C}$ since the 
compact subsystem  $\overline{\left(\bigcup_{n}f^{\circ -n}(S)\right)\bigcap\hat{\Bbb{R}}}$
appeared in \eqref{union2} is countable and thus of zero entropy \cite[Proposition 5.1]{MR1736362}}. To get the countability,  notice that otherwise Lemma \ref{decomposition} implies that 
$\overline{\left(\bigcup_{n}f^{\circ -n}(S)\right)\bigcap\hat{\Bbb{R}}}$
contains a non-empty closed subspace consisting of condensation points that intersects 
$\left(\bigcup_{n}f^{\circ -n}(S)\right)\bigcap\hat{\Bbb{R}}$. So a real point in the backward orbit of a member of $S$ is a condensation point of 
$\overline{\left(\bigcup_{n}f^{\circ -n}(S)\right)\bigcap\hat{\Bbb{R}}}$
and thus that of the bigger space $\mathcal{J}_\Bbb{R}(f)$; a contradiction because $S$ is a subset of isolated points of 
$\mathcal{J}_\Bbb{R}(f)$ and its iterated preimages cannot be condensation points.
\end{proof}
\noindent
Here are the two lemmas we referred to during the proof: 
\begin{lemma}\label{no Siegel}
A rotation domain of a non-linear rational map with real coefficients which intersects the real axis  is a fixed Herman ring that furthermore contains the whole real circle $\hat{\Bbb{R}}$ as one of the leaves of its natural foliation.\footnote{Compare with the discussion in Example \ref{Blaschke1}.}
\end{lemma}
\begin{proof} 
Suppose a rotation domain $U$ of $f\in\Bbb{R}(z)$ intersects $\hat{\Bbb{R}}$ along an open subset. As
$f(\hat{\Bbb{R}})\subseteq \hat{\Bbb{R}}$, $f(U)\cap U\neq\emptyset$ and thus the period of $U$ is one.
Aside from the center of a Siegel disk, the orbit closure of any point from a fixed rotation domain is an analytic Jordan curve which is the leaf passing through that point in the corresponding $f$-invariant foliation. So the circle $\hat{\Bbb{R}}$ contains such a leaf. But that leaf  is a topological circle as well, therefore they should coincide.
It is not hard to see that $U$ cannot be a Siegel disk: an open  topological disk in $\hat{\Bbb{C}}$ 
containing the circle $\hat{\Bbb{R}}$ and invariant under the reflection with respect to it has to be 
the whole Riemann sphere which is absurd. 
\end{proof}

\begin{lemma}\label{decomposition}
Let $X$ be a Hausdorff Baire space with a countable basis for its topology; for instance, take $X$ to be a separable complete metric space or a second countable locally compact Hausdorff topological space. Then $X$ can be uniquely written as a union $D\cup P$ of its closed subspaces where $D$ is countable and the closure of isolated points of $X$ and $P$ is perfect. In particular, when $X$ is furthermore compact and totally disconnected, $P$ would be a Cantor space.\footnote{Compare with \cite[chap. 2, Exercise 129]{MR1886084}.}
\end{lemma}
\begin{proof} 
Recall that a point of a topological space is called a \textit{condensation point} if  there are uncountably many points in any open neighborhood of it. Obviously, the set of condensation points of $X$ is closed. Denote it by $P$ and let $D$ be the closure of isolated points of $X$. It is easy to show that $D$ is countable: pick a countable basis $\mathcal{U}$ for the topology of $X$ 
such that no two isolated points of $X$ lie in the same element of $\mathcal{U}$. If $D$ is uncountable, there will be a 
$U\in\mathcal{U}$ with uncountably many points of $D$ but with precisely one isolated point of $X$. Removing that point yields an open in $X$ disjoint from the subset of isolated points but intersecting its closure $D$ which is impossible. \\
\indent
We need the following elementary fact:
\begin{itemize}
\item \textit{Each point of a perfect Hausdorff Baire space is a condensation point.}
\vspace{2mm}
\\
\textit{Proof.} Otherwise, there exists a point $x$ admitting an open neighborhood containing only countably many points of $X$, say $x_n$'s. Each open subset $X-\{x_n\}$ is dense since
$x_n$ is not isolated while $\bigcap_n\left(X-\{x_n\}\right)$ is not dense because its closure misses $x$; this contradicts the Baire property.
\end{itemize} 
Equipped with this, the uniqueness part is immediate: If one alternatively write $X$ as $D\cup P'$; $P'$ is a closed subspace of a separable metric or of a second countable locally compact Hausdorff topological space and is thus a space with the same properties. These properties guarantee that $P'$ is a Baire space too.  Hence by the above fact any point of  the perfect space $P'$ is a condensation point of $P'$ and thus that of $X$. Conversely, a condensation point $x$ of $X$ has to lie in $P'$ as otherwise 
$X-P'\subseteq D$ would be a countable open neighborhood of $x$ in $X$.\\
\indent
Next, we establish the remaining properties claimed in the lemma. The union $D\cup P$ is the whole $X$ since if $x\notin D$, then $x$ admits an open neighborhood $U$ with no isolated point of $X$. So in the subspace topology $U$ is a perfect space. Being an open subset of the Baire space $X$, $U$ is a Baire space as well and so the aforementioned fact indicates that $x$ is a condensation point of it and thus a condensation point of the ambient space $X$. Hence $x$ has to lie in $P$. Finally, notice that $P$ is perfect: any open neighborhood in $X$ of a point $x\in P$  has uncountably many points of 
$X=D\cup P$ with at most countably many of them in $X-P\subseteq D$ so it intersects $P$ in infinitely many points.\\
\indent
To finish the proof, notice that when $X$ is totally disconnected, the same has to be true for the subspace $P$ which then would be a second countable compact Hausdorff (thus metrizable) totally disconnected perfect space and hence homeomorphic to the standard Cantor set  
\cite[chap. 2, Theorem 69]{MR1886084}. 
\end{proof}
%\noindent
Going back to the density of preperiodic points established above for the Cantor part of the dynamics $Z$, the same can be said both for those intervals $I$ in \eqref{union1} which are periodic under $f$ and also for the case where $\mathcal{J}_{\Bbb{R}}(f)$
is the whole circle. This is the content of the proposition below that relies on certain results from real one-dimensional dynamics:   
\begin{proposition}\label{technical1}
Let $f\in\Bbb{R}(z)$ be a rational map and $I$ a non-degenerate compact  connected subset of the real Julia set $\mathcal{J}_{\Bbb{R}}(f)$ which is invariant under (an iterate of) $f$. Then there is a closed subsystem $P\subseteq I$ in which preperiodic points are dense and
$$h_{\rm{top}}\left(f\restriction_I:I\rightarrow I\right)=h_{\rm{top}}\left(f\restriction_P:P\rightarrow P\right).$$
\end{proposition}

\begin{proof}
{The proof is based on some of the ideas developed in the proof of Theorem \ref{Structure} and the result below adapted from
\cite[\S7.1]{2015arXiv150403001R}:

\textit{For a continuous self-map of an interval or circle the following are equivalent:
\begin{itemize}
\item the topological entropy is positive;
\item there is an infinite closed subsystem which is chaotic in the sense of Devaney. 
\end{itemize}}
\noindent
Of course, if $f\restriction_I$ does not admit any periodic point, the preceding result indicates that its entropy is zero and so $P=\emptyset$ works. Hence, suppose there are periodic orbits. 
The set of preperiodic points of 
$f\restriction_I:I\rightarrow I$
is completely invariant. Its closure 
$P:=\overline{\left\{x\in I\,|\, x \text{ is preperiodic}\right\}}$
is therefore forward-invariant.
If $P=I$ we are done. Otherwise, the proper   subset $I-P$ of the circle $\hat{\Bbb{R}}$ is a  union of countably many open or half-open subintervals. The image of any of them under $f$ is another non-degenerate subinterval of $I$ whose interior is away from $P$: Otherwise, there would exist an open subinterval $I'$ of $I-P$ with $f(I')$ being an open subinterval that intersects $P$ and thus contains a preperiodic point of $f$. But $I'$ then should have a preperiodic point itself which is impossible.  We conclude that $f$ acts on the set of closed arcs of 
$\mathcal{J}(f)\subseteq\hat{\Bbb{R}}$
that are closures of connected components of $I-P$. Then,
the exact same argument as in the proof of Theorem \ref{Structure} based on the non-existence of wandering intervals implies that only finitely many of these closures are periodic (i.e. preserved by an iterate),
and under iteration any arbitrary interval component of $I-P$ eventually maps into the closure of a component of $I-P$ which is periodic under $f$.
Interval components of $I-P$ which are not periodic are away from the non-wandering set of  $f\restriction_I$
and hence do not contribute to the entropy.
Picking an interval component of $I-P$ whose closure is preserved by an iterate of $f$, and next 
applying the theorem just quoted to the system obtained by restricting that iterate to the compact arc of the circle which is the closure,  we deduce that the entropy of this restriction should be zero; if not, that component of $I-P$ must contain periodic points. Therefore, 
${\rm{NW}}\left(f\restriction_I:I\rightarrow I\right)$
is contained in the union of $P$ and the $f$-invariant union of finitely many closed arcs that form a subsystem of zero entropy. Consequently, the entropy of $f\restriction_I$ 
coincides with that of its subsystem $f\restriction_P$.\footnote{Here and later in this section we invoke this simple fact that in a compact topological system written as a union of finitely many closed subsystems, the topological entropy coincides with the maximum of the entropies of these subsystems \cite[Proposition 2.5.5]{MR1963683}.}}
\end{proof}

\begin{corollary}\label{technical2}
Given a rational map $f\in\Bbb{R}(z)$ of degree $d\geq 2$, the system 
$f\restriction_{\hat{\Bbb{R}}}:\hat{\Bbb{R}}\rightarrow\hat{\Bbb{R}}$
has a closed (possibly vacuous) subsystem of the same topological entropy in which preperiodic points are dense.  
\end{corollary}
\begin{proof}
{This has been established when the whole real circle is Julia in Proposition \ref{technical1} and is an immediate consequence of Theorem \ref{Structure} when the Julia set intersects $\hat{\Bbb{R}}$ non-trivially: 
Intervals $I\in\mathcal{I}$ that are not preserved by any iterate of $f$ are away from the non-wandering set of dynamics and aside from them, we have only finitely many intervals in $\mathcal{I}$, say $I_1,\dots,I_n$ and all of them are invariant under some sufficiently divisible iterate $f^{\circ p}$.
Thus the non-wandering set of the system 
$f^{\circ p}\restriction_{\mathcal{J}_\Bbb{R}(f)}:\mathcal{J}_\Bbb{R}(f)\rightarrow\mathcal{J}_\Bbb{R}(f)$ is contained  the union of  closed subsystems $Z$ (admitting a dense subset of preperiodic points according to Theorem \ref{Structure}) and $I_i$'s $(1\leq i\leq n)$.  Moreover, by Proposition \ref{technical1} the system
$f^{\circ p}\restriction_{I_i}$
has a  closed subsystem $P_i$ of the same entropy in which preperiodic points are dense. 
So $A:=\left(\bigsqcup_{i=1}^nP_i\right)\bigcup Z$ is a closed subsystem of 
$f^{\circ p}\restriction_{\hat{\Bbb{R}}}$ in which preperiodic points are dense and then
$B:=\bigcup_{v=1}^{p}f^{-\circ v}(A)$ 
would be a closed invariant subsystem of 
$f\restriction_{\hat{\Bbb{R}}}$
satisfying the same property. This is the desired subsystem:
\small
\begin{equation*}
\begin{split}
&h_{\rm{top}}\left(f\restriction_{\hat{\Bbb{R}}}\right)=
h_{\rm{top}}\left(f\restriction_{\mathcal{J}_\Bbb{R}(f)=\mathcal{J}(f)\cap\hat{\Bbb{R}}}\right)=
\max\left\{h_{\rm{top}}\left(f\restriction_{\sqcup_{i=1}^nI_i}\right),
h_{\rm{top}}\left(f\restriction_{Z}\right)\right\}\\
&=\frac{1}{p}\max\left\{h_{\rm{top}}\left(f^{\circ p}\restriction_{\sqcup_{i=1}^nI_i}\right),
h_{\rm{top}}\left(f^{\circ p}\restriction_{Z}\right)\right\}
=\frac{1}{p}\max\left\{\max_{1\leq i\leq n}\left\{h_{\rm{top}}\left(f^{\circ p}\restriction_{I_i}\right)\right\}
,h_{\rm{top}}\left(f^{\circ p}\restriction_{Z}\right)\right\}
\\
&
=\frac{1}{p}\max\left\{\max_{1\leq i\leq n}\left\{h_{\rm{top}}\left(f^{\circ p}\restriction_{P_i}\right)\right\}
,h_{\rm{top}}\left(f^{\circ p}\restriction_{Z}\right)\right\}
=\frac{1}{p}\,h_{\rm{top}}\left(f^{\circ p}\restriction_{(\sqcup_{i=1}^n P_i)\cup Z}\right)
=\frac{1}{p}\,h_{\rm{top}}\left(f^{\circ p}\restriction_A\right)\\
&\leq \frac{1}{p}\,h_{\rm{top}}\left(f^{\circ p}\restriction_{\cup_{v=1}^{p}f^{-\circ v}(A)}\right)=
h_{\rm{top}}\left(f\restriction_{\cup_{v=1}^{p}f^{-\circ v}(A)}\right)
=h_{\rm{top}}\left(f\restriction_B\right).
\end{split}
\end{equation*}
\normalsize
}
\end{proof}

\begin{proof}[Proof of Theorem \ref{temp2}]
Immediately follows from Theorem \ref{Structure} and Corollary \ref{technical2}.
\end{proof}

\subsection{Examples and remarks}\label{examples subsection}
\begin{remark}
Theorem \ref{Structure} indicates that the real entropy  
may be written as:
\begin{equation}\label{Cantor and Interval}
h_\Bbb{R}=\max\{h_{\rm{Interval}},h_{\rm{Cantor}}\}
\end{equation}
where $h_{\rm{Interval}}$, $h_{\rm{Cantor}}$
are the contributions of interval and Cantor subsystems of $\mathcal{J}_\Bbb{R}(f)$ to the real entropy.
Unlike $h_\Bbb{R}$, the functions 
$$h_{\rm{Interval}}:\mathcal{M}'_d-\mathcal{S}'\rightarrow \left[0,\log(d)\right],
\quad h_{\rm{Cantor}}:\mathcal{M}'_d-\mathcal{S}'\rightarrow \left[0,\log(d)\right]$$
are not  continuous: The real Julia set of a hyperbolic rational map with real coefficients lacks interval components (cf. Remark \ref{elaborate}) and so  at the points determined by hyperbolic rational maps one has 
$h_{\rm{Interval}}=0, h_{\rm{Cantor}}=h_\Bbb{R}$. Among real polynomials, such points constitute a dense subset 
(the \textit{density of hyperbolicity} \cite{MR2342693}) while  there definitely exist real polynomials whose real Julia sets are union of intervals; e.g. the $d^{\rm{th}}$ Chebyshev polynomial whose Julia set is the interval $[-2,2]$ 
where $h_{\rm{Interval}}=h_\Bbb{R}=\log(d), h_{\rm{Cantor}}=0$.
\end{remark}

\begin{remark}\label{elaborate}
The dynamics of the interval subsystems of the real Julia set can be investigated more thoroughly. Suppose $I\subseteq\mathcal{J}_{\Bbb{R}}(f)$ is a compact non-degenerate interval invariant under (an iterate of) $f$. 
There is a description of ``typical'' ``attractors'' of  interval maps available from  real one-dimensional  dynamics that can be applied to $f\restriction_I:I\rightarrow I$. It has to be mentioned that here ``typical'' can be interpreted in either the topological sense or the metric (measure-theoretic) sense and ``attractors'' also arise in both contexts and they might be more complicated than an attracting periodic orbit; see \cite{MR790735} for a treatment of the general notion of an attractor for a continuous self-map of a differentiable manifold.
According to \cite[Theorems 1.2, 1.3]{MR2600680}, 
for points from a  ``large''  subset ($2$nd category  or of full Lebesgue measure according to the context) $Y$ of $I$,  the $\omega$-limit sets $\omega(x)\, (x\in Y)$ can be described as either of the following:
\begin{itemize}
\item $\omega(x)$ is a periodic orbit; 
\item $\omega(x)$ is a minimal Cantor set of Lebesgue measure zero and coincides with $\omega(c)$ for a critical point $c$;
\item $\omega(x)$ is a finite union of intervals on which $f$ acts as a topologically transitive transformation. 
\end{itemize}
\noindent
The first possibility definitely does not occur in our situation since a periodic orbit to which another orbit accumulates is attracting for the interval map and thus admits a (perhaps parabolic) basin that must be away from the Julia set and hence from the $I$; in particular, for a hyperbolic $f\in\Bbb{R}(z)$ there is no such interval component and 
the real Julia set is totally disconnected.
So a typical $\omega$-limit set fits in one of the last two categories.
In the latter, there is a subinterval on which some iterate of $f$ restricts to a transitive (and hence chaotic) interval map. Such a map is of positive entropy \cite[Theorem 4.77]{2015arXiv150403001R}; furthermore, it is conjugate to a piecewise linear map 
\cite[chap. III, Theorem 4.1]{MR1239171}. 
A quadratic polynomial example  would be $x\mapsto 4x(1-x)$ 
from the logistic family or equivalently 
the Chebyshev polynomial $x\mapsto x^2-2$ which on the Julia set $[-2,2]$ of the $z\mapsto z^2-2$ restricts to a unimodal map conjugate to the tent map $x\in [-1,1]\mapsto 1-|2x-1|\in [-1,1]$. \\
\indent
As for Cantor attractors, first notice that if an interval subsystem of a real rational map admits either a residual or full Lebesgue measure subset of points whose  $\omega$-limit sets are Cantor, then all points of the interval are definitely Julia because the $\omega$-limit set of a Fatou point is either finite or a topological circle. 
If there is a residual subset of points whose $\omega$-limit sets coincide with the Cantor $\omega$-limit set $\omega(c)$ of a critical point, then it can be shown that map $f$ is \textit{infinitely renormalizable at $c$} and the attractor is called \textit{solenoidal}; see \cite{MR1239171, MR2600680} for details.  Such an example from the logistic family is provided by the
\textit{Feigenbaum map} at the end of the first period-doubling cascade whose entropy is zero and its real Julia set is $[0,1]$.
  It turns out that there are also examples of the so called \textit{wild  attractor} which are non-solenoidal Cantor metric attractors: There are non-renormalizable unimodal polynomial maps on the unit interval for which $\omega(x)$ is a Cantor set for points from a full Lebesgue measure subset but is the whole interval for points from a $2$nd category subset. The entropy is positive due to transitivity. There is no such an example in the quadratic family. See \cite{MR1370759} for the details. \\
\end{remark}

\begin{example}\label{last}
For a real quadratic polynomial, the Julia set cannot intersect the real axis along more than one interval. By contrast, 
the real Julia set of a real cubic polynomial can admit infinitely many interval components:
Suppose  the Julia set of $f$ is disconnected and there is an invariant interval component  $I$ of $\mathcal{J}_{\Bbb{R}}(f)$ for which the dynamics of $f\restriction_I$ is unimodal. The interval would have infinitely many interval components in its backward orbit due to the fact that a real cubic equation  always has an odd number of real roots and 
also because no component of $f^{-\circ k}(I)$ other than $I$ itself has a critical point (otherwise both critical orbits would be in the Julia set); a fact that shows that for every $l>k>0$ the subsets
 $f^{-\circ k}(I)\cap\hat{\Bbb{R}}-I$ and $f^{-\circ l}(I)\cap\hat{\Bbb{R}}-I$ are disjoint since otherwise there would be an interval component of $\mathcal{J}_\Bbb{R}(f)$ other than $I$  invariant under an iterate which then according to the first part of Theorem \ref{Structure} must contain a critical point of $f$. An inductive argument now shows that there are countably many disjoint intervals $\left\{I_0\right\}_{k\geq 0}$ with $I_0=I$ and $f^{-\circ k}(I)\cap\hat{\Bbb{R}}=I_0\sqcup I_k$
 for every $k>0$.
 Figure \ref{fig:2} demonstrates such an example.

\begin{figure}[ht!]
\center
\includegraphics[width=12cm, height=5cm]{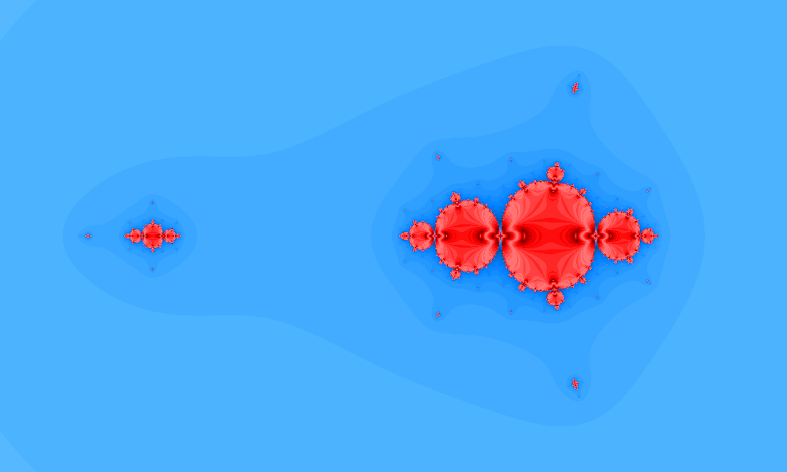}
\caption{The Julia set of $0.19 z^3 + z^2 - 1$; the red areas are attracting basins in the filled Julia set. See Example \ref{last}.}
\label{fig:2}
\end{figure}

\end{example}

\section{Entropy values over $\mathcal{M}'_{d,s}$}\label{appendix B}
The goal of this appendix is to prove a more refined version of Proposition \ref{dimension} that addresses the real entropy values realized by degree $d$ real rational maps whose restrictions to $\hat{\Bbb{R}}$ are of degree $s$:

\begin{proposition}\label{entropy values}
Let $s\in\left\{-d,\dots,0,\dots,d\right\}$ be of the same parity as $d$. Restricted to $\mathcal{M}'_{d,s}-\mathcal{S}'$,
the range of the real entropy function $h_\Bbb{R}$ is 
$\left[\log\left(\max(|s|,1)\right),\log(d)\right]$.
\end{proposition} 

As mentioned in Remark \ref{degree-bound}, the lower bound is a special case of \cite{MR0458501}, but we include an easy proof here:

\begin{proposition}\label{lower bound}
Let $f\in\Bbb{R}(z)$ be non-constant and of degree $d\geq 2$. Then for any point $x\in\hat{\Bbb{R}}$:
\begin{equation}\label{inequality'}
h_\Bbb{R}(f)\geq \limsup_{n\to\infty}\frac{1}{n}\log\left(\#\left\{y\in\hat{\Bbb{R}}\,|\, f^{\circ n}(y)=x\right\}\right),
\end{equation}
and  the equality is always achieved for some $x$. 
Moreover, in the absence of post-critical relations, the number of solutions of $f^{\circ n}(y)=x$ can be counted with multiplicities. 
\end{proposition}

\begin{proof} 
The  multimodal circle map 
$f^{\circ n}\restriction_{\hat{\Bbb{R}}}:\hat{\Bbb{R}}\rightarrow\hat{\Bbb{R}}$ 
(after being pulled back to an interval map via the bijection 
$x\in[0,1)\mapsto{\rm{e}}^{2\pi{\rm{i}}x}\in\hat{\Bbb{R}}$) has strictly monotonic pieces as it is analytic.
Hence given $x\in\hat{\Bbb{R}}$, no two points of 
$\left\{y\in\hat{\Bbb{R}}\,|\, f^{\circ n}(y)=x\right\}$ are in the same lap of $f^{\circ n}$, so 
%\small
$$h_\Bbb{R}(f)=h_{\rm{top}}\left(f\restriction_{\hat{\Bbb{R}}}\right)
=\lim_{n\to\infty}\frac{1}{n}\log\left(l\left(\left(f\restriction_{\hat{\Bbb{R}}}\right)^{\circ n}\right)\right)\geq 
\limsup_{n\to\infty}\frac{1}{n}\log\left(\#\left\{y\in\hat{\Bbb{R}}\,|\, f^{\circ n}(y)=x\right\}\right),
$$
%\normalsize
where the lap number of laps of a multimodal map $g$ is denoted by $l(g)$.
Next, we claim that the equality can be achieved. Let 
$c_1,\dots,c_{l-1}\in \hat{\Bbb{R}}$ be turning points of  $f\restriction_{\hat{\Bbb{R}}}$.
The set of turning points of the iterate 
$\left(f\restriction_{\hat{\Bbb{R}}}\right)^{\circ n}$
is 
$$\bigcup_{1\leq i\leq l-1}\bigcup_{0\leq k\leq n-1}\left\{y\in \hat{\Bbb{R}}\,|\, f^{\circ k}(y)=c_i\right\};$$
which yields the following upper bound for the number of laps of this iterate:
\begin{equation}\label{inequality1}
l\left(\left(f\restriction_{\hat{\Bbb{R}}}\right)^{\circ n}\right)-1
\leq\sum_{i=1}^{l-1}\sum_{k=0}^{n-1}\#\left\{y\in \hat{\Bbb{R}}\,|\, f^{\circ k}(y)=c_i\right\}.
\end{equation}
Now, aiming for a contradiction, if 
$$\forall 1\leq i\leq l-1: \limsup_{n\to\infty}\frac{1}{n}\log\left(\#\left\{y\in\hat{\Bbb{R}}\,|\, f^{\circ n}(y)=c_i\right\}\right)<h_\Bbb{R}(f),$$
then one can find $\alpha<h_\Bbb{R}(f)$ and $C\gg 0$ with 
$$\forall 1\leq i\leq l-1, \forall n\geq 0: \#\left\{y\in\hat{\Bbb{R}}\,|\, f^{\circ n}(y)=c_i\right\}\leq C{\rm{e}}^{\alpha n}.$$
But then the right-hand side of \eqref{inequality1} would be at most $nC(l-1){\rm{e}}^{\alpha n}$ and consequently:
$$h_\Bbb{R}(f)=\lim_{n\to\infty}\frac{1}{n}\log\left(l\left(\left(f\restriction_{\hat{\Bbb{R}}}\right)^{\circ n}\right)\right)\leq 
\lim_{n\to\infty}\frac{1}{n}\log\left(nC(l-1){\rm{e}}^{\alpha n}\right)=\alpha;$$
a contradiction. To finish the proof, observe that when there is not any post-critical relation, there is at most one critical point $c$ in the backward orbit of $x$ under $f\restriction_{\hat{\Bbb{R}}}$
and counting the number of solutions of $f^{\circ n}(y)=x$ with multiplicities raises this number through multiplication by a factor which is at most the multiplicity of $c$; this cannot affect the exponential growth rate.
\end{proof}

\begin{remark}
The preceding proposition provides an interesting interpretation of real entropy: Given $f\in\Bbb{R}(z)$ and $\epsilon>0$, for every 
$x\in\hat{\Bbb{R}}$, as $n\to\infty$ the cardinality 
$$\#\left\{y\in\hat{\Bbb{R}}\,|\, f^{\circ n}(y)=x\right\}$$
 is $${\rm{o}}\left({\rm{exp}}\left(\left(h_\Bbb{R}(f)+\epsilon\right)n\right)\right);$$
and $h_\Bbb{R}(f)$ is the smallest number for which this holds.\\
\indent The inequality \eqref{inequality'} can also be written as 
\begin{equation}\label{inequality}
h_\Bbb{R}(f)\geq\log(d)+\limsup_{n\to\infty}\frac{1}{n}\log\left(\frac{1}{d^n}\#\left\{y\in\hat{\Bbb{R}}\,|\, f^{\circ n}(y)=x\right\}\right).
\end{equation}
The term $\frac{1}{d^n}\#\left\{y\in\hat{\Bbb{R}}\,|\, f^{\circ n}(y)=x\right\}$  on the right-hand side of \eqref{inequality}
is reminiscent of the constructions of the measure of maximal entropy $\mu_f$ for the rational map $f:\Bbb{C}\rightarrow\Bbb{C}$ as outlined in \cite{MR736568}: for any non-exceptional point $x\in\Bbb{C}-\mathcal{E}(f)$, the empirical measures $\frac{1}{d^n}\sum_{\left\{y\in\hat{\Bbb{C}}\,|\, f^{\circ n}(y)=x\right\}}\delta_y$ weakly converge to $\mu_f$ as $n\to\infty$ where  in averaging the Dirac masses the multiplicities have been taken into account.  Taking $x$ in  \eqref{inequality} to be from $\hat{\Bbb{R}}-\mathcal{E}(f)$, we conclude that:
$$\limsup_{n\to\infty}\frac{1}{d^n}\#\left\{y\in\hat{\Bbb{R}}\,|\, f^{\circ n}(y)=x\right\}\leq\mu_f(\hat{\Bbb{R}}).$$ 
If $\mu_f(\hat{\Bbb{R}})=0$, the argument of the logarithm in \eqref{inequality} tends to zero, so an interesting behavior could be expected from  $\limsup_{n\to\infty}\frac{1}{n}\log\left(\frac{1}{d^n}\#\left\{y\in\hat{\Bbb{R}}\,|\, f^{\circ n}(y)=x\right\}\right)$. The situations where the measure of maximal entropy assigns a positive value to the real circle have been exhaustively studied in Theorems \ref{extremal classification}, \ref{extremal classification 1}.
\end{remark}

We  next need to show that the real entropy value $\log(d)$ is realized in $\mathcal{M'}_{d,s}$. Notice that according to Theorem \ref{extremal classification}, for any such map the Julia set must be contained in $\hat{\Bbb{R}}$. This motivates the following:

\begin{proposition}\label{upper bound}
Suppose $d$ and $s\in\{-d,\dots,d\}$ are of the same parity. There exists an open subset of real rational maps $f$ of degree $d$ with $\deg\left(f\restriction_{\hat{\Bbb{R}}}\right)=s$  whose Julia sets are completely real.
\end{proposition}

\begin{proof}
We outline a construction of such an $f$ for which the Julia set is contained in a subinterval of the real line.  Let us first assume $s>0$ and start with a prototype for a map $f\in{\rm{Rat}}_d(\Bbb{R})$ for which 
$f\restriction_{\hat{\Bbb{R}}}:\hat{\Bbb{R}}\rightarrow\hat{\Bbb{R}}$
is of topological degree $s$. One can for instance consider
\begin{equation}\label{the function}
f(x)=\frac{p(x)}{\prod_{i=1}^{s-1}(x-u_i)\prod_{j=1}^{\frac{d-s}{2}}(x-v_j)^2};
\end{equation}
where $p(x)$ is a degree $d$ polynomial whose sign at points $-\infty<u_1<\dots<u_{s-1}<v_1$ vary alternatingly while it gets positive values at $v_1<v_2<\dots<v_{\frac{d-s}{2}}<+\infty$. This is a map with $\frac{d+s}{2}-1$ vertical asymptotes at 
$\frac{d-s}{2}$ of which $f$ takes $\infty$ with multiplicity $2$; see Figure \ref{fig:3}.
Restricted to any of the $s$ subintervals below
\begin{equation}\label{former}
[-\infty,u_1], [u_1,u_2],\dots,[u_{s-2},u_{s-1}], [u_{s-1},v_1]
\end{equation}
$f$ is of degree $+1$ and attains every value in $\hat{\Bbb{R}}$ because it tends to $-\infty$ at the left end and to $+\infty$ at the right end. 
 On the other hand, the restriction of $f$ to each of the subintervals 
\begin{equation}\label{latter}
[v_1,v_2],\dots,\left[v_{\frac{d-s}{2}-1},v_{\frac{d-s}{2}}\right], \left[v_{\frac{d-s}{2}},+\infty\right]
\end{equation}
tends to $+\infty$ at the endpoints and is hence of topological degree zero. Adding up the degrees of restrictions to the intervals in \eqref{former} and \eqref{latter} (that cover $\hat{\Bbb{R}}$), we deduce that 
$\deg\left(f\restriction_{\hat{\Bbb{R}}}\right)=s$. \\
\indent
\begin{figure}[ht!]
\center
\includegraphics[height=13.5cm]{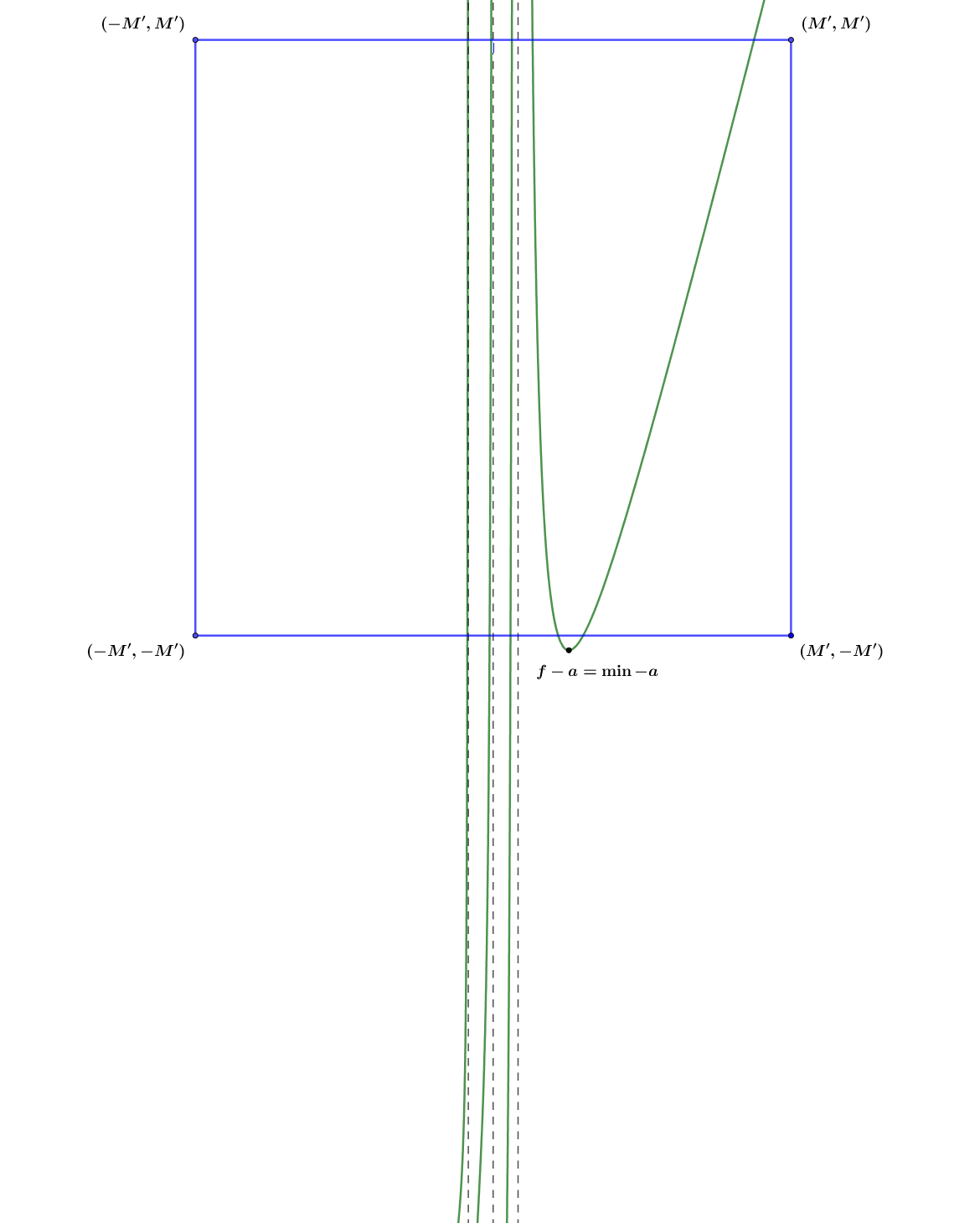}
\caption{The graph of a real rational map $f(x)-a$ of degree $d=5$ whose restriction to the reals is of topological degree $s=3$. The vertical asymptotes are shown by dotted lines. The interval $[-M',M']$ is backward-invariant and thus contains the Julia set. In this example
$$
f(x)=\frac{4(x^5+3x^2-1)}{(x^2+x)(x-1)^2},\,  a=35,\, M'=12;
$$
and the minimum value that $f$ attains over $[1,+\infty)$ is $\min=22.411$.
As for the other parameters appeared in the proof, one has $|f(x)|>\lambda|x|$ once $|x|>M$ where $\lambda=4$ and $M=1$.}
\label{fig:3}
\end{figure}
Next, we  argue that the Julia sets of certain maps of the form \eqref{the function} are contained in $\hat{\Bbb{R}}$. Notice that $x=\infty$ is a fixed point of $f$ with multiplier given by the reciprocal of the leading coefficient of $p$. Hence we can pick $p(x)$ so that  $\infty$ is an attracting fixed point whose multiplier is an arbitrarily chosen number 
$\frac{1}{\lambda}$ from $\left(0,\frac{1}{2}\right)$; so
$|f(x)|>\lambda|x|$ for $|x|$ large enough where $\lambda>2$ is fixed. Pick an interval $[-M,M]$ outside which this inequality is satisfied.  We claim that there are parameters $a$ and $M'>M$ for which $[-M',M']$ is backward-invariant under $f(x)-a$; the property that implies the Julia set of $f(x)-a$ (which is in the same normal form \eqref{the function}) is a subset of $[-M',M']$, and thus finishes the proof.   Denote the smallest value that $f$ attains over the intervals in \eqref{latter} by $\min$. If $\min-a<-M'$, then every value in $[-M',M']$ is realized over each of  the intervals in \eqref{former}, \eqref{latter} as their image under $f(x)-a$ contains $[\min-a,+\infty]\supset [-M',M']$. Every such a value, if regular,  is realized at least once over each of the former $s$ subintervals and at least twice over each of the latter  $\frac{d-s}{2}$ subintervals due to the discussion above on the degrees of the corresponding restrictions. But, aside from finitely many exceptions, the fibers of the ambient map $f:\hat{\Bbb{C}}\rightarrow\hat{\Bbb{C}}$ are of cardinality $d=s+2.\frac{d-s}{2}$. Consequently, $f^{-1}\left([-M',M']\right)$ must be a subset of $\Bbb{R}$ in such a situation. We do not want points of this preimage to lie outside $[-M',M']$. Notice that if $|x|>M'>M$ one can write:
$$
|f(x)-a|>\lambda|x|-a>\lambda M'-a;
$$    
and hence it suffices to have $\lambda M'-a>M'$. Combining all these, we first choose $M'>M$ so large that 
$M'+\min<(\lambda-1)M'$ -- which is possible as $\lambda>2$ -- and we then take $a$ to be between these two numbers. Solving these inequalities for $a$, we
conclude that starting with a real map $f$ in the form \eqref{the function} with the open conditions imposed on it -- that comes with the associated constants $\lambda, M, \min$ -- for all $a$ satisfying 
$$a>\min+M,\frac{\lambda-1}{\lambda-2}.\min$$ 
the Julia set of $f(x)-a$ is contained in $\hat{\Bbb{R}}$. \\
\indent
For negative degrees,  one needs to slightly modify the normal form \eqref{the function} as   
\begin{equation}\label{the function'}
f(x)=\frac{-p(x)}{\prod_{i=1}^{s-1}(x-u_i)\prod_{j=1}^{\frac{d-s}{2}}(x-v_j)^2};
\end{equation}
with the same conditions on the sign of the degree $d$ polynomial $p$ at 
$$
-\infty<u_1<\dots<u_{s-1}<v_1<\dots<v_{\frac{d-s}{2}}<+\infty.
$$
The degree of the induced map $\hat{\Bbb{R}}\rightarrow\hat{\Bbb{R}}$ would be $-s<0$; and in the subsequent discussion, one only needs to replace $\min$ appeared above with the largest value $\max$ that $f$ attains over the subintervals in \eqref{latter}. \\
\indent
Finally, if $s=0$, the degree $d$ must be even and it suffices to construct a real hyperbolic polynomial $f$ of degree $d$ with real entropy $\log(d)$ as then by Theorem \ref{entropy constant over hyperbolic} the entropy would be $\log(d)$ throughout some sufficiently small neighborhood of $f$ which, according to Theorem \ref{extremal classification}, means that the corresponding Julia sets have no non-real points. Consider the $d^{\rm{th}}$ Chebyshev polynomial $T_d$ given by 
$T_d\left(z+\frac{1}{z}\right)=z^d+\frac{1}{z^d}$. The Julia set is $[-2,2]$ and $\pm2$ are the only finite critical values.  The polynomial $f(z):=2.T_d(z)$ is then hyperbolic since its critical values $\pm4$ lie outside the interval $[-2,2]$ and it is not hard to see that $\left|2.T_d(x)\right|>2|x|$ if $x\in\Bbb{R}-[-2,2]$; so the critical orbits converge to $\infty$.  The interval $[-2,2]$ is backward-invariant under the Chebyshev polynomial $T_d(z)$ and hence the preimage of $[-4,4]$ under $2.T_d(z)$ lies in the smaller interval $[-2,2]$. We conclude that the Julia set of $f(z)=2.T_d(z)$ is a subset of $[-2,2]$.
\end{proof}

\begin{proof}[Proof of Proposition \ref{entropy values}]
Given a real rational map $f$ for which $f\restriction_{\hat{\Bbb{R}}}$ is of degree $s$ with $|s|>1$, for a point $x\in\hat{\Bbb{R}}$ away from the forward orbits of critical points, $\left(f\restriction_{\hat{\Bbb{R}}}\right)^{-n}(x)$ has at least $|s|^n$ elements for each $n$ and thus $h_{\Bbb{R}}(f)\geq\log(|s|)$ by inequality \eqref{inequality'}. This lower bound in Proposition \ref{entropy values} is always achieved by the constructions of Example \ref{Blaschke1}. For 
$s=0\text{ or }\pm1$, the restrictions of polynomials $\pm z^d$ to $\hat{\Bbb{R}}$ is of entropy zero and of topological degree $0$ (for $d$ even) or $\pm 1$ (for $d$ odd). \\
\indent
The upper bound $\log(d)$ for $h_\Bbb{R}$ is automatic and Proposition \ref{upper bound} indicates that it is indeed sharp. 
\end{proof}

\bibliographystyle{alpha}
\bibliography{bib2}

\end{document}